\documentclass[11pt,reqno,oneside]{amsart}

\usepackage{graphicx,subfigure,threeparttable}
\usepackage{amssymb}
\usepackage{epstopdf}

\usepackage[a4paper, total={6in, 9in}]{geometry}

\usepackage{booktabs} % for much better looking tables
\usepackage{subfig} % make it possible to include more than one captioned figure/table in a single float

\usepackage{amsfonts}

\usepackage{amsmath,amsthm,mathrsfs,amssymb,cite}
\usepackage[usenames]{color}

\newtheorem{thm}{Theorem}[section]

\newtheorem{lem}{Lemma}[section]

\theoremstyle{definition}
\newtheorem{defn}{Definition}[section]

\theoremstyle{remark}

\newtheorem{rem}{Remark}[section]
\numberwithin{equation}{section}

\numberwithin{equation}{section}

\newcounter{saveeqn}
\newcommand{\eqnref}[1]{(\ref {#1})}

\newcommand{\beq}{\begin{equation}}
\newcommand{\eeq}{\end{equation}}

\title[Surface-localized transmission eigenstates and applications]{Surface-localized transmission eigenstates, super-resolution imaging and pseudo surface plasmon modes}

\author{Yat Tin Chow}
\address{Department of Mathematics, University of California, Riverside, USA}
\email{yattinc@ucr.edu}

\author{Youjun Deng}
\address{School of Mathematics and Statistics, Central South University, Changsha, Hunan, China}
\email{youjundeng@csu.edu.cn; dengyijun\_001@163.com}

\author{Youzi He}
\address{Department of Mathematics, Hong Kong Baptist University, Kowloon, Hong Kong, China}
\email{18481469@life.hkbu.edu.hk}

\author{Hongyu Liu}
\address{Department of Mathematics, City University of Hong Kong, Kowloon, Hong Kong, China}
\email{hongyu.liuip@gmail.com, hongyliu@cityu.edu.hk}

\author{Xianchao Wang}
\address{Department of Mathematics, City University of Hong Kong, Kowloon, Hong Kong, China}
\email{xcwang90@gmail.com}

\date{} % Activate to display a given date or no date (if empty),
\begin{document}

\maketitle

\begin{abstract}

We present the discovery of a novel and intriguing global geometric structure of the (interior) transmission eigenfunctions associated with the Helmholtz system. It is shown in generic scenarios that there always exists a sequence of transmission eigenfunctions with the corresponding eigenvalues going to infinity such that those eigenfunctions are localized around the boundary of the domain. We provide a comprehensive and rigorous justification in the case within the radial geometry, whereas for the non-radial case, we conduct extensive numerical experiments to quantitatively verify the localizing behaviours. The discovery provides a new perspective on wave localization. As significant applications, we develop a novel inverse scattering scheme that can produce super-resolution imaging effects and propose a method of generating the so-called pseudo surface plasmon resonant (PSPR) modes with a potential sensing application.

\medskip



\end{abstract}

\section{Introduction}

We start with the mathematical formulation of the interior transmission eigenvalue problems. Let $\Omega$ be a bounded Lipschitz domain in $\mathbb{R}^d, d=2, 3$, with a connected complement $\mathbb{R}^d\backslash\overline{\Omega}$. Let $\mathbf{n}(x)\in L^\infty(\Omega)$ and $k\in\mathbb{R}_+$. Consider the following system of partial differential equations (PDEs) for $w \in H^1(\Omega)$ and $v\in H^1(\Omega)$:
\begin{equation}\label{eq:trans1}
\left\{
\begin{array}{ll}
\Delta w+k^2\mathbf{n}^2(x)w=0  &\text{in} \ \Omega,\medskip \\
\Delta v+k^2 v =0 &\text{in} \ \Omega, \medskip \\
\displaystyle{w=v,\ \ \frac{\partial w}{\partial\nu}=\frac{\partial v}{\partial\nu} } &\text{on} \ \partial \Omega, \\
\end{array}
\right.
\end{equation}
where $\nu$ is the exterior unit normal vector to $\partial\Omega$. \eqref{eq:trans1} is referred to as the interior transmission eigenvalue problem associated with the Helmholtz equation. Physically, $\mathbf{n}$ signifies the refractive index of an inhomogeneous medium supported in $\Omega$ and $k\in\mathbb{R}_+$ signifies a wavenumber. Clearly, $w\equiv v\equiv 0$ are a pair of trivial solutions to \eqref{eq:trans1}. If there exists a non-trivial pair of solutions $(w, v)$ to \eqref{eq:trans1}, $k\in\mathbb{R}_+$ is called a transmission eigenvalue, and $w, v$ are the associated transmission eigenfunctions.
%In a similar manner, one can formulate the transmission eigenvalue problem associated with the time-harmonic Schr\"odigner equation by replacing the first equation in \eqref{eq:trans1} to be $\Delta u+k^2 u-V(x) u=0$,
%%as follows:
%%\begin{equation}\label{eq:trans2}
%%\left\{
%%\begin{array}{ll}
%%\Delta u+k^2 u-V(x) u=0  &\text{in} \ \Omega,\medskip \\
%%\Delta v+k^2 v =0 &\text{in} \ \Omega, \medskip \\
%%\displaystyle{u=v,\ \ \frac{\partial u}{\partial\nu}=\frac{\partial v}{\partial\nu} } &\text{on} \ \partial \Omega, \\
%%\end{array}
%%\right.
%%\end{equation}
%where $V(x)\in L^\infty(\Omega)$ denotes a potential field and $E:=k^2$ signifies an energy level. The latter transmission eigenvalue problem is connected to \eqref{eq:trans1} by formally taking $n^2=1-V/k^2$ for a given $k\in\mathbb{R}_+$. In what follows, we shall mainly present our study for the transmission eigenvalue problem \eqref{eq:trans1} and the corresponding applications to the wave scattering governed the Helmholtz equation. However, through the aforementioned connection, all of the subsequent results can be readily extended to the Schor\"odinger transmission eigenvalue problem as well as applications to the quantum scattering governed by the Schr\"odinger equation.
The transmission eigenvalue problem arises in the scattering theory of time-harmonic waves and was first proposed in \cite{Kir} in the context of a reconstruction scheme for the inverse scattering problem. It also relates to the non-scattering phenomenon, a.k.a invisibility cloaks \cite{BL1, BL2, BPS}. An alternative formulation of the transmission eigenvalue problem is given for $w_0\in H_0^2(\Omega)$ as:
\begin{equation}\label{eq:transn2}
\left(\Delta+k^2\mathbf{n}^2\right)\left(\Delta+k^2 \right)w_0=0\quad\mbox{in}\ \ \Omega.
\end{equation}
In fact, it can be shown that if $w, v\in H^2(\Omega)$ in \eqref{eq:trans1}, then $w_0:=w-v\in H_0^2(\Omega)$ satisfies the transmission eigenvalue problem \eqref{eq:transn2}. The transmission eigenvalue problem is non self-adjoint, non-elliptic and nonlinear (in the sense of the formulation \eqref{eq:transn2} which is quadratic in terms of $k^2$). Hence, its study is practically important and mathematically challenging.

The study of the transmission eigenvalue problem has a long and colourful history. The spectral properties of the transmission egienvalues have been intensively and extensively studied in the literature. Roughly speaking, the spectral properties of the transmission eigenvalues resemble those of the classical Dirichlet/Neumann Laplacian in many aspects. In fact, under certain generic conditions on $\mathbf{n}$, particularly including the case with $\mathbf{n}\neq 1$ being a positive constant, it is known that there exists an infinite and discrete set of eigenvalues satisfying $0<k_1\leq k_2\leq\cdots \leq k_{\ell}\leq\cdots\rightarrow +\infty$, with $+\infty$ the only accumulating point. For each eigenvalue, the corresponding eigenspace is finite dimensional. We refer to \cite{CCH, CK, PS} and the references therein for the related studies of the aforementioned and other properties of transmission eigenvalues. Recently, it is revealed in \cite{Blasten18a, BLin, BL1, BL2, BL3, BXL, CX, DCL} that the transmission eigenfunctions possess rich and peculiar geometric structures; see also a recent survey paper \cite{Liu} for more related discussions. The studies indicate that near a point on $\partial\Omega$ where the magnitude of the extrinsic curvature is sufficiently large, the transmission eigenfunctions must be nearly vanishing. In particular, in the extreme case where the high-curvature part degenerates to become a corner, then under a certain mild regularity conditions on the transmission eigenfunctions locally around the corner, they must be vanishing near the corner point. The regularity conditions are characterised by the H\"older-continuity of the transmission eigenfunctions or a certain blowup rate of the Herglotz kernels via the Herglotz approximations of the transmission eigenfunctions (cf. \cite{BL1, BL2, BL3, DCL}). It is noted that in \eqref{eq:trans1},  only $H^1$-regularity is imposed on the transmission eigenfunctions and by the standard Sobolev embedding, the H\"older-continuity of the transmission eigenfunctions is not always fulfilled. It is numerically shown in \cite{BL3} that if the required regularity conditions are not fulfilled, the singularity of the transmission eigenfunctions around the corner point leads to a certain localizing phenomenon locally around the corner point. Nevertheless, it is pointed out that in the physical situation where the transmission eigenfunctions are connected to the acoustic wave scattering problems, the regularity conditions are fulfilled and hence one always has the locally vanishing properties; see \cite{Liu} for more related discussion on this aspect. In fact, the vanishing properties have been used in establishing several novel unique recovery results for the inverse scattering problems in different scenarios by a single far-field measurement \cite{Blasten18a,BL2,BL4, BXL,CX,CDL, DCL}.

The purpose of the present article is twofold. First, we present an intriguing discovery of a certain global geometric property of the transmission eigenfunctions. It is clear that the geometric structures of the transmission eigenfunctions discussed earlier are of local features. In this paper, we show that under generic scenarios, either the transmission eigenfunction $w$ or $v$ is localized on the boundary surface in $\mathbb{R}^3$ or the boundary curve in $\mathbb{R}^2$. That is, the energy of $w$ (resp. $v$) is localized around $\partial\Omega$ and barely enters into the bulk $\Omega$. For terminological convenience, those peculiar eigen-modes are referred to as the surface-localized eigenstates (SLEs). More precisely, it is shown that if $\mathbf{n}>1$, there exists a sequence of $v_j, j=1, 2, \ldots$, which are SLEs such that the corresponding eigenvalues $k_j\rightarrow +\infty$, and if $0<\mathbf{n}<1$, the same geometric property holds for the transmission eigenfunction $w$. In the case that $\Omega$ is a ball in $\mathbb{R}^d$ and $\mathbf{n}$ is a constant, we rigorously justify such a spectral property, whereas for the general geometry with a variable $\mathbf{n}$, we conduct extensive numerical experiments to verify this spectral property. The reasons for us to do so are as follows. For the radial geometry and constant $\mathbf{n}$, we can have a thorough understanding of the SLEs for the transmission eigenvalue problem \eqref{eq:trans1}, whereas for the general case,
we can only derive a qualitative understanding of the SLEs, which we shall present in a forthcoming paper \cite{CHLW}. The numerical examples in this paper not only verify the existence of the SLEs in the generic scenario, but also demonstrate their intriguing quantitative behaviours. It is shown that the existence of SLEs are topologically robust against large deformation or even twisting of the boundary surface/curve $\partial \Omega$. However, the SLEs themselves are topologically sensitive to the change of the boundary.  Our study unveils a significant physical phenomenon that is completely unknown before. Moreover, it provides a new perspective on wave localization, which is a central topic to many practical applications including surface plasmon resonances \cite{BS, FM, Kli, OI, Z}, topologically robust states in quantum Hall effect \cite{Elg, Hat, Tho}, directional optical waveguide \cite{Hal, Hald}, photonic transport \cite{Kha, RZ, Wang}, and cloaking due to anomalous localized resonance \cite{Amm2, LL, MN}. The second purpose of this paper is to propose two novel applications of the newly discovered SLEs including producing a super-resolution imaging scheme for the inverse acoustic scattering problem and generating the so-called pseudo surface plasmon resonant (PSPR) modes with a potential sensing application. We choose to first focus on the SLEs for the transmission eigenvalue problem and shall present more background introduction on inverse scattering problems and plasmon resonances in Sections 3 and 4.

The rest of the paper is organized as follows. In Section 2, we present the SLEs with both theoretical and numerical verifications in different scenarios. In Section 3, we consider the application of SLEs for inverse scattering imaging. In Section 4, the application of SLEs for pseudo plasmon resonances is proposed. The paper is concluded in Section 5 with some relevant discussions.

\section{Surface-localized transmission eigenstates}\label{sect:2}

In this section, we first give the definite description of the SLEs. Next, we illustrate the existence of the SLEs for a special case theoretically. Then, for general situations, we provide extensive numerical examples to verify the existence of the SLEs and show their intriguing quantitative behaviours.

\begin{defn}
Consider a function $w\in L^2(\Omega)$. It is said to be surface-localized if there exists a sufficiently small $\epsilon_0\in\mathbb{R}_+$ such that
\begin{equation*}%\label{eq:localized1}
\frac{\|w\|_{L^2(\mathcal{N}_{\epsilon_0}(\partial\Omega))}}{\|w\|_{L^2(\Omega)}}=1-\mathcal{O}(\epsilon_0),
\end{equation*}
where
\begin{equation*}%\label{eq:localized2}
\mathcal{N}_{\epsilon_0}(\partial\Omega):=\{x\in\Omega;\ \mathrm{dist}(x, \partial\Omega)<\epsilon_0\}.
\end{equation*}
\end{defn}

\subsection{Spherical geometry and constant $\mathbf{n}$ }

In this part, we let $\Omega:=\{x\in \mathbb{R}^d: |x|<r_0\in\mathbb{R}_+\},\ d=2,3,$ and $\mathbf{n}$ be a positive constant. For this case, we provide a comprehensive and accurate characterisation of the SLEs. To that end, we let $m\in \mathbb{N}$ be a positive integer, $J_m(|x|)$ be the first kind Bessel function of order $m$, and $J_m^{'}(|x|)$ be the derivative of $J_m(|x|)$. We also let $j_{m,s}$ denote the $s$-th positive zero of $J_m(|x|)$, and $j_{m,s}^{'}$ be the $s$-th positive zero of $J_m^{'}(|x|)$.
Recalling  from \cite[Section 9.5, p. 370]{Abramowitz}, one has
\begin{align}
  %&  j_{m,1}<j_{m+1,1}<j_{m,2}<j_{m+1,2}<j_{m,3}< \cdots, \notag \\
   \label{eq:zeros_derivative}&  m\leq  j_{m,1}^{'}<j_{m,1}<j_{m,2}^{'}<j_{m,2}<j_{m,3}^{'}< \cdots,\\
   \label{eq:Bessel}&J_{m}(|x|)=\frac{(|x|/2)^m}{\Gamma(m+1)}\mathop{\Pi} \limits_{s=1}^{\infty}\left(1-\frac{|x|^2}{j_{m,s}^2}   \right).
  \end{align}

To prove the localizing phenomena, we suppose that the order $m$ of the Bessel function $J_m(x)$ is sufficiently large, and choose two sequences of integers $s$ and $s'$ as follows:
\beq\label{eq:choices01}
s(m):=[m^{\gamma_1}], \quad s'(m)=[m^{\gamma_2}], \quad 0<\gamma_1<\gamma_2<1,
\eeq
where $[t]$ signifies the rounding of a real number $t$, namely $t=[t]+\epsilon_t$ with $0<\epsilon_t<1$.

\begin{lem}\label{lem:zeros}
Let  $\Omega=\{x\in \mathbb{R}^d: |x|<r_0\in\mathbb{R}_+\},\, d=2,3,$ and $\mathbf{n}>1$ be a constant. Then the transmission eigenfunctions of \eqnref{eq:trans1} are given in terms of the Bessel  functions.
Let $k_{m, \ell}$, $\ell=1, 2, \ldots$, be the transmission eigenvalues of \eqnref{eq:trans1}, where $m$ is the order of the Bessel  function. Then there exists a subsequence of $\{k_{m, \ell}\}$, denoted as $\{k_{m,s(m)}\}$, such that for $m$ sufficiently large, it holds that
\beq\label{eq:eigspan01}
k_{m,s}\in (j_{m,s(m)}, j_{m,s'(m)}),
\eeq
where $s(m)$ and $s'(m)$ are defined in \eqnref{eq:choices01}. Moreover, one has
\beq\label{eq:leasymk01}
\frac{k_{m,s}}{m}\rightarrow 1, \quad \mbox{as} \quad m\rightarrow \infty.
\eeq
More specifically, there exists $\varsigma\in (-2/3, 0)$ such that
\beq\label{eq:leasymk02}
k_{m,s}=m(1+m^{\varsigma}+o(m^{\varsigma})).
\eeq
\end{lem}
\begin{proof}
Without loss of generality, we assume that the radius of $\Omega$ is $r_0=1$.
Since $\mathbf{n}$ is a positive constant, we can expand the solutions $w$ and $v$ of the system \eqref{eq:trans1} into Fourier series in terms of the Bessel  functions $J_m(x)$ or the spherical Bessel  functions $j_m(x)$ of the first kind and the spherical harmonics:
\begin{equation*}
\begin{aligned}
  &w(x)=
  \begin{cases}
  \sum\limits_{m=-\infty}^{\infty}\alpha_{m} J_m(k\mathbf{n}|x|)\mathrm{e}^{\mathrm{i}m\theta}, & d=2,\medskip\\
  \sum\limits_{m=0}^{\infty}\sum\limits_{l=-m}^{m}\alpha_{m}^{l} j_m(k\mathbf{n}|x|)Y_m^{l}(\hat x),& d=3,\\
  \end{cases}\\
 & v(x)=\begin{cases}
  \sum\limits_{m=-\infty}^{\infty}\beta_{m} J_m(k|x|)\mathrm{e}^{\mathrm{i}m\theta}, & d=2,\medskip\\
  \sum\limits_{m=0}^{\infty}\sum\limits_{l=-m}^{m}\beta_{m}^{l} j_m(k|x|)Y_m^{l}(\hat x),& d=3,\\
  \end{cases}
  \end{aligned}
  \end{equation*}
 where $ Y_m^{l} $ is a spherical harmonic function of degree $m$ and order $l$. Since
\begin{equation*}
    j_m(|x|)=\sqrt{\frac{\pi}{2|x|}}J_{m+1/2}(|x|),
  \end{equation*}
we only consider  the two-dimensional case and  the three-dimensional case can be proved in a similar manner.

 For a fixed $m\in \mathbb{N}_+$, the solutions of $\Delta w+k^2 \mathbf{n}^2 w=0$ and $\Delta v+k^2 v=0$ in $\Omega$ can be written as
 \begin{equation*}\label{eq:mEigFunc}
 \begin{aligned}
  &w_{k}= \alpha_m J_m(k\mathbf{n}|x|)\mathrm{e}^{\mathrm{i}m\theta}, \\
  & v_{k}=  \beta_{m} J_m(k|x|)\mathrm{e}^{\mathrm{i}m\theta}.
  \end{aligned}
  \end{equation*}
In order to guarantee that $w=v$ on the boundary $\partial \Omega$, i.e., when $|x|=1$,  we choose
\begin{equation*}
  \beta_{m}=\frac{J_m(k\mathbf{n})}{J_m(k)}\alpha_m.
\end{equation*}
Moreover, the transmission eigenvalues $k$'s are determined from the relation:
\begin{equation*}
  \frac{\partial w}{\partial \nu}=\frac{\partial v}{\partial \nu} \quad \mbox{on }\partial \Omega.
\end{equation*}
Thus, with the help of the recurrence relation of the derivatives of the Bessel  functions, the eigenvalues $k$'s are positive zeros of the following function (see \cite{Colton})
\begin{equation*}\label{eq:kFunc}
  f_m(k)=J_{m-1}(k)J_m(k\mathbf{n})-\mathbf{n}J_m(k)J_{m-1}(k\mathbf{n}),\quad m\geq 1.
\end{equation*}
From \eqref{eq:zeros_derivative} and  \eqref{eq:Bessel},  one can deduce that
\begin{equation*}\label{eq:Jm}
J_m(|x|)\geq 0, \quad  |x|\in[0,j_{m,1}^{'}] .
\end{equation*}
Next, we compute the following identity:
\begin{equation*}\label{eq:appfm01}
\begin{aligned}
f_m\left( j_{m,s} \right)f_m\left( j_{m,s'} \right)
&= J_{m-1}\left( j_{m,s} \right)J_m\left( \mathbf{n}j_{m,s} \right) J_{m-1}(j_{m,s'})J_m\left( \mathbf{n}j_{m,s'} \right).
\end{aligned}
\end{equation*}
It can be shown \cite{Qu} that the zeros $j_{m,s}$ of the Bessel  function $J_m(|x|)$ has the following sharp upper and lower bounds
\begin{equation}\label{eq:jms}
  m-\frac{a_s}{2^{1/3}}m^{1/3}<j_{m,s}<m-\frac{a_s}{2^{1/3}}m^{1/3}+\frac{3}{20}a_s^2\frac{2^{1/3}}{m^{1/3}},
\end{equation}
where $a_s$ is the $s$-th negative zero of the Airy function and has the representation
\begin{equation*}
  a_s=-\left(\frac{3\pi}{8}(4s-1)  \right)^{2/3}(1+\sigma_s).
\end{equation*}
Here, $\sigma_s$ can be estimated by
\begin{equation*}
 0\leq \sigma_s\leq0.130\left(\frac{3\pi}{8}(4s-1.051)  \right)^{-2}.
\end{equation*}
By noting the choice of $s(m)$ and $s'(m)$ in \eqnref{eq:choices01} for $m$ sufficiently large, there holds
\beq\label{eq:approxj01}
\begin{split}
j_{m,s}= m(1+C_0 m^{2(\gamma_1-1)/3}+o(m^{2(\gamma_1-1)/3})), \\
 j_{m,s'}=m(1+C_0 m^{2(\gamma_2-1)/3}+o(m^{2(\gamma_2-1)/3})),
\end{split}
\eeq
where $C_0$ is a positive constant.
Note that the Bessel  function admits the following asymptotic formula (see \cite{Kor02}, p 129):
\beq\label{eq:approxj02}
J_m(|x|)=\sqrt{\frac{2}{\pi \sqrt{|x|^2-m^2}}}\cos(\sqrt{|x|^2-m^2}-\frac{m\pi}{2}+m\arcsin(m/|x|)-\frac{\pi}{4})\Big(1+o(1)\Big),
\eeq
for $|x|>m$ and $m\rightarrow \infty$.
By combing \eqnref{eq:approxj01}, \eqnref{eq:approxj02} and some straightforward computations, one obtains
\begin{equation*}
J_m\left( \mathbf{n}j_{m,s} \right)=C_{m,\mathbf{n}}\cos\left(m\big(\sqrt{\mathbf{n}^2-1}-\frac{\pi}{2}+\arcsin\frac{1}{\mathbf{n}}+\mathcal{O}(m^{\varsigma_1})\big)-\frac{\pi}{4})\right)\Big(1+o(1)\Big),
\end{equation*}
where $\varsigma_1:={2(\gamma_1-1)/3}$ and
$$
C_{m,\mathbf{n}}:=\sqrt{\frac{2}{\pi}} \Big(\mathbf{n}^2j_{m,s}^2-m^2\Big)^{-1/4}=\sqrt{\frac{2}{\pi}}\Big((\mathbf{n}^2-1)m\Big)^{-1/4}+\mathcal{O}(m^{2\varsigma_1}).
$$
And similarly,
\begin{equation*}
J_{m-1}\left(j_{m,s'} \right)=C_{m,1}\cos\left(m\mathcal{O}(m^{2\varsigma_2})-\frac{\pi}{4}\right)\Big(1+o(1)\Big),
\end{equation*}
where $\varsigma_2:={2(\gamma_2-1)/3}$ and
$$
C_{m,1}:=\sqrt{\frac{2}{\pi}} \Big(j_{m,s'}^2-m^2\Big)^{-1/4}=\sqrt{\frac{2}{\pi}}\mathcal{O}(m^{-1/2}m^{-\varsigma_2/4}).
$$
Without loss of generality we suppose that
$$J_{m-1}\left( j_{m,s} \right)J_m\left( \mathbf{n}j_{m,s} \right)>0.$$
We next show that there exists at least one choice of $s'=m^{\gamma_2}$ such that
$$J_{m-1}(j_{m,s'})J_m\left( \mathbf{n}j_{m,s'}\right)<0,$$
that is,
\begin{equation}\label{eq:nnn1}
\cos\left(m\mathcal{O}(m^{2\varsigma_2})-\frac{\pi}{4}\right)\cos\left(m\big(\sqrt{\mathbf{n}^2-1}-\frac{\pi}{2}+\arcsin\frac{1}{\mathbf{n}}+\mathcal{O}(m^{\varsigma_2})\big)-\frac{\pi}{4})\right)
<0.
\end{equation}
Indeed, \eqref{eq:nnn1} can be easily realized by modifying $\gamma_2$ (noting that the two cosine functions above are always of different frequencies w.r.t $\varsigma_2(\gamma_2)$). Thus one has
$k_{m,s}\in(j_{m,s(m)}, j_{m,s'(m)})$ and by using \eqnref{eq:approxj01} one has \eqnref{eq:leasymk01}, which completes the proof.
\end{proof}

With the help of Lemma \ref{lem:zeros}, we are able to show the following important results about the existence of the SLEs.

\begin{thm}\label{thm:2.1}
Let  $\Omega=\{x\in \mathbb{R}^d: |x|<r_0\in\mathbb{R}_+\},\, d=2,3,$ and $\mathbf{n}>1$ be a constant. Consider the transmission eigenvalue problem \eqref{eq:trans1}.   Then there exists a subsequence $\{k_{m, s}\}$ of  the transmission eigenvalues $\{k_{\ell}\}$ such that $\infty$ is the only accumulation point of the sequence  $\{k_{m,s}\}$ and the corresponding eigenfunctions $v_{k_{m, s}}$ are surface-localized around the boundary $\partial \Omega$.
%For $0<n<1$ or $n>1$ and sufficiently large $m$, there exist infinitely many discrete transmission eigenvalues and SLEs. Specifically,
%\begin{itemize}
%  \item if $n>1$, the interval of the $s$-th eigenvalue is $k_{s}\in ({j_{m,s}}/{n}, {j_{m,s+1}}/{n})$, $s\in \mathbb{N}_+$, and the corresponding eigenfunction $v_{m,k_{s}}$ is localized on the boundary $\partial \Omega$;
%  \item if $0<n<1$, then the interval of the $s$-th eigenvalue is $k_{s}\in ({j_{m,s}}, {j_{m,s+1}})$, $s\in \mathbb{N}_+$, and the corresponding eigenfunction $w_{m,k_{s}}$ is localized on the boundary $\partial \Omega$.
%\end{itemize}
\end{thm}
\begin{proof}
Without loss of generality, we assume that the radius of $\Omega$ is $r_0=1$, we only consider  the two-dimensional case and  the three-dimensional case can be discussed similarly.
We choose the sequence of $k_{m,s}$ satisfy \eqnref{eq:eigspan01} such that $s$ and $s'$ are chosen in \eqnref{eq:choices01}, then such sequence $k_{m,s}\rightarrow m$, $m\rightarrow \infty$.
We prove that the corresponding eigenfunctions $v_{k_{m, s}}$ are
surface-localized around the boundary $\partial \Omega$. Let $\Omega_{\tau}:=\{x: |x|<\tau, \ \tau<1\}$ with $\varepsilon:=1-\tau$ being a sufficiently small constant. We first prove that $J_m(k_{m,s}r)$ is monotonously increasing with respect to $r\in (0, 1)$. {By using \eqnref{eq:zeros_derivative} and \eqnref{eq:jms} one can show that
$$j'_{m,1}= m(1+\mathcal{O}(m^{-2/3}))$$}
as $m\rightarrow \infty$. Since we also have from \eqnref{eq:approxj01} that
$$k_{m,s}= m(1+ m^{\varsigma}),\quad \varsigma\in [\varsigma_1,\varsigma_2],$$
where $\varsigma_j={2(\gamma_j-1)/3}$, $j=1, 2$. One thus has
$$k_{m,s}r=mr(1+m^{\varsigma})< j'_{m,1},$$
for any fixed $r<1$ and sufficiently large $m$.
One thus has $J_m'(k_{m,s}r)>0$ for $0<r<1$, since $j'_{m,1}$ is the first maximum of the Bessel  function $J_m(k_{m,s}r)$. Furthermore, we note that there admits the following asymptotic formula (see \cite{Kor02}, p. 129):
\begin{equation*}
J_m(k_{m,s}r)=\frac{z^m e^{m\sqrt{1-z^2}}}{(2\pi m)^{1/2}(1-z^2)^{1/4}(1+\sqrt{1-z^2})^m}\Big(1+o(1)\Big), \quad z=\frac{k_{m,s}r}{m}, \quad 0<r<1.
\end{equation*}
Thus one can derive the following asymptotic expansion:
\beq\label{eq:asympjn02}
\begin{split}
J_m(k_{m,s}|x|)<J_m(k_{m,s}\tau)=&Cm^{-1/2}\left(\frac{k_{m,s}\tau}{m}\frac{e^{\sqrt{1-\Big(\frac{k_{m,s}\tau}{m}\Big)^2}}}{1+\sqrt{1-\Big(\frac{k_{m,s}\tau}{m}\Big)^2}}\right)^m\Big(1+o(1)\Big)\\
=&Cm^{-1/2}\left((1-\varepsilon)\frac{e^{\sqrt{2\varepsilon(1+o(1))+Cm^{\varsigma}}}}{1+\sqrt{2\varepsilon(1+o(1))+Cm^{\varsigma}}}\right)^m\Big(1+o(1)\Big)\\
=&Cm^{-1/2}(1-\varepsilon)^m(1+o(1))^m\Big(1+o(1)\Big)
\end{split}
\eeq
for $x\in\Omega_\tau$, where $C$ is a positive constant. On the other hand, for $m$ sufficiently large, one can choose $r_1$ by
$$r_1=\frac{j_{m,1}'}{k_{m,s}},$$
where $\tau< r_1<1$. Since
$$
m<j_{m,1}'=k_{m,s}r_1,
$$
by using the asymptotic expansion \eqnref{eq:approxj02}, there holds
\beq\label{eq:asympjn03}
\begin{aligned}
&m^{1-\varsigma/2}\int_\tau^1 J_m^2(k_{m, s}r) \, r\mathrm{d}r\\
&\quad \geq m^{1-\varsigma/2}\int_{r_1}^{1}J_m^2(k_{m, s}r) \, r\mathrm{d}r\\
&\quad\geq\frac{2}{\pi}\frac{m^{1-\varsigma/2}}{\sqrt{k_{m,s}^2-m^2}}\int_{r_1}^{1}\cos^2(r\sqrt{k_{m,s}^2-\frac{m^2}{r^2}}-\frac{m\pi}{2}+m\arcsin\Big(\frac{m}{k_{m,s}r}\Big)-\frac{\pi}{4})dr\\
&\quad=\frac{2}{\pi}\frac{m^{1-\varsigma/2}}{\sqrt{k_{m,s}^2-m^2}}\frac{1-r_1}{2}\Big(1+\mathcal{R}(m)\Big)=\frac{1}{\sqrt{2}\pi}\Big(1+o(1)+\mathcal{R}(m)\Big)\geq\frac{1}{\sqrt{2}\pi}\Big(1+\mathcal{R}(m)\Big),
\end{aligned}
\eeq
where the remaining term $\mathcal{R}(m)$ fulfils the following estimate as $m\rightarrow\infty$,
\begin{equation*}
\begin{aligned}
\mathcal{R}(m)=&\frac{1}{1-r_1}\int_{r_1}^{1}\sin\left(2r\sqrt{k_{m,s}^2-m^2}\Big(1+\mathcal{O}(m^{-2/3-\varsigma}+m^{2/3\varsigma})\Big)\right)dr\\
=&\frac{1}{2(1-r_1)\sqrt{k_{m,s}^2-m^2}}\int_{r_1\sqrt{k_{m,s}^2-m^2}}^{\sqrt{k_{m,s}^2-m^2}}\sin\left(r'\Big(1+\mathcal{O}(m^{-2/3-\varsigma}+m^{2/3\varsigma})\Big)\right)dr'\rightarrow 0.
\end{aligned}
\end{equation*}
Note that
\beq
\lim_{m\rightarrow \infty} (1-\varepsilon)^m m^{-\frac{\varsigma}{2}}=0.
\eeq
Therefore, from \eqref{eq:asympjn02} and \eqref{eq:asympjn03}, it holds that
  \begin{equation*}%\label{eq:v-tau}
  \begin{aligned}
   \frac{\|v_{k_{m, s}}\|_{L^2(\Omega_{\tau} )}^2}{\|v_{k_{m, s}}\|^2_{L^2(\Omega)}}
   =& \frac{\int_0^\tau J_m^2(k_{m, s}r) \, r\mathrm{d}r}{\int_0^1 J_m^2(k_{m, s}r) \, r\mathrm{d}r}\leq\frac{m^{1-\varsigma/2}J_m^2(k_{m, s}\tau)\int_0^\tau \, r\mathrm{d}r}{m^{1-\varsigma/2}\int_{r_1}^1 J_m^2(k_{m, s}r) \, r\mathrm{d}r}\rightarrow 0\ \ \mbox{as}\ \ m\rightarrow\infty.
   \end{aligned}
  \end{equation*}
Hence, the transmission eigenfunction $v_{k_{m,s}}$  is surface-localized on the boundary $\partial \Omega$.

The proof is complete.
\end{proof}

\begin{thm}\label{thm:2.2n}
Under the same assumptions in Theorem \ref{thm:2.1} and suppose that $\{w_{k_{m, s}}\}$ are the corresponding eigenfunctions with respect to $w$ in \eqref{eq:trans1}. Then $\{w_{k_{m, s}}\}$ are not surface-localize around $\partial \Omega$.
\end{thm}
\begin{proof}
Since $w_{k_{m,s}}=\beta_m J_m(k_{m,s}\mathbf{n}|x|)$, one can immediately obtain that for $|x|=\frac{1}{\mathbf{n}}\frac{j_{m,1}'}{k_{m,s}}$, $w_{k_{m,s}}$ attains its maximum value. Since $\frac{1}{\mathbf{n}}\frac{j_{m,1}'}{k_{m,s}}<\frac{1}{\mathbf{n}}<1$ holds uniformly, thus there exists $\tau$ which is independent of $m$, such that $\frac{1}{\mathbf{n}}<\tau<1$. Noting that $j_{m,1}'/k_{m,s}<1$ and by using Theorem \ref{thm:2.1}, one has
\beq
\begin{split}
\int_0^1 J_m^2(k_{m, s}\mathbf{n}r) \, r\mathrm{d}r=&\frac{1}{n^2}\int_0^{\mathbf{n}} J_m^2(k_{m, s}r') \, r'\mathrm{d}r'\\
=&\frac{1}{n^2}\int_0^{\frac{j_{m,1}'}{k_{m,s}}} J_m^2(k_{m, s}r') \, r'\mathrm{d}r'+\frac{1}{n^2}\int_{\frac{j_{m,1}'}{k_{m,s}}}^{\mathbf{n}} J_m^2(k_{m, s}r') \, r'\mathrm{d}r'\\
<&\frac{2}{n^2}\int_{\frac{j_{m,1}'}{k_{m,s}}}^{\mathbf{n}} J_m^2(k_{m, s}r') \, r'\mathrm{d}r'=2\int_{\frac{1}{\mathbf{n}}\frac{j_{m,1}'}{k_{m,s}}}^1 J_m^2(k_{m, s}r) \, r\mathrm{d}r.
\end{split}
\eeq
By following a similar arguments as in Theorem \ref{thm:2.1}, one has for $\tau\leq t\leq 1$
\beq\label{eq:noloc01}
\begin{split}
Q(t):=&\int_{\frac{1}{\mathbf{n}}\frac{j_{m,1}'}{k_{m,s}}}^{t} \, J_m^2(k_{m, s}\mathbf{n}r) r\mathrm{d}r\\
=&\frac{2}{\pi}\int_{\frac{1}{\mathbf{n}}\frac{j_{m,1}'}{k_{m,s}}}^{t} \frac{\cos^2(r\sqrt{k_{m,s}^2\mathbf{n}^2-m^2/r^2}-\frac{m\pi}{2}+m\arcsin\Big(\frac{m}{k_{m,s}\mathbf{n}r}\Big)-\frac{\pi}{4})}{\sqrt{k_{m, s}^2\mathbf{n}^2-m^2/r^2}} \, \mathrm{d}r\\
\sim & \frac{1}{\pi k_{m,s}^2}\left(\sqrt{k_{m, s}^2\mathbf{n}^2 t^2-m^2}-\sqrt{j_{m, 1}'^2-m^2}\right).
\end{split}
\eeq
Thus one has
  \begin{equation*}%\label{eq:w-tau}
  \begin{aligned}
   &\displaystyle \frac{\|w_{k_{m, s}}\|_{L^2(\Omega_{\tau} )}^2}{\|w_{k_{m, s}}\|^2_{L^2(\Omega)}}   =\frac{\int_0^\tau J_m^2(k_{m, s}\mathbf{n}r) \, r\mathrm{d}r}{\int_0^1 J_m^2(k_{m, s}\mathbf{n}r) \, r\mathrm{d}r}\\
   \geq& \frac{\int_{\frac{1}{\mathbf{n}}\frac{j_{m,1}'}{k_{m,s}}}^{\tau} \, J_m^2(k_{m, s}\mathbf{n}r) r\mathrm{d}r}{2\int_{\frac{1}{\mathbf{n}}\frac{j_{m,1}'}{k_{m,s}}}^{1} J_m^2(k_{m, s}\mathbf{n}r) \, r\mathrm{d}r}= \frac{Q(\tau)}{2Q(1)}\\
    \sim & \frac{\sqrt{k_{m, s}^2\mathbf{n}^2 \tau^2-m^2}-\sqrt{j_{m, 1}'^2-m^2}}{2\Big(\sqrt{k_{m, s}^2\mathbf{n}^2 -m^2}-\sqrt{j_{m, 1}'^2-m^2}\Big)}\\
    \sim&\frac{\sqrt{\mathbf{n}^2 \tau^2-1}}{2\sqrt{\mathbf{n}^2-1}}>0,
   \end{aligned}
  \end{equation*}
which completes the proof.
\end{proof}

Theorems~\ref{thm:2.1} and \ref{thm:2.2n} indicate that we find a sequence of eigenfunction pairs $\{(w_{k_m, s}, v_{k_m, s})\}$ such that $\{v_{k_m, s}\}$ are SLEs, whereas $\{w_{k_m, s}\}$ are not SLEs. However, we would like to point out that we did not exclude the possibility of finding a sequence of eigenfunction pairs with both eigenfunctions being SLEs. Nevertheless, we can easily have the following result. 

%\begin{cor}
%Assume that  $\Omega=\{x\in \mathbb{R}^d: |x|<r_0\},\, d=2,3,$  and $n>1$. For arbitrary $m,s\in \mathbb{N}_+$, there exists a sufficient large $n(m,s)$, such that all of eigenfunctions $v_{k_{\ell}}$ are localized on the boundary $\partial \Omega$ for $k_{\ell}<K_0\in \mathbb{R}_+$.
%\end{cor}
%\begin{proof}

%\end{proof}

\begin{thm}\label{thm:2.2}
Let  $\Omega=\{x\in \mathbb{R}^d: |x|<r_0\in\mathbb{R}_+\},\, d=2,3,$ and $0<\mathbf{n}<1$ be a constant. Consider the transmission eigenvalue problem \eqref{eq:trans1}. Then there exists a subsequence $\{k_{m,s}\}$ of  the transmission eigenvalues $\{k_{\ell}\}$ such that $\infty$ is the only accumulation point of the sequence  $\{k_{m,s}\}$ and the corresponding eigenfunctions $w_{k_{m,s}}$ are surface-localized around the boundary $\partial \Omega$.
\end{thm}
\begin{proof}
Set $\widetilde{k}=k/\mathbf{n}$. It is directly verified that the first two equations in \eqref{eq:trans1} become  $(\Delta+\widetilde{k}^2) w=0$ and $(\Delta+\widetilde{k}^2 \mathbf{n}^{-2})v=0$, while the transmission conditions on $\partial\Omega$ remain unchanged. Noting that $\mathbf{n}^{-1}>1$, all of the previous results in Theorem \ref{thm:2.1} hold  with $v$ replaced by $w$. Thus, one has that $w_{k_{\ell_m}}$ is surface-localize on the boundary $\partial \Omega$.
\end{proof}

%\begin{rem}\label{rem:2.1}
%In what follows, if a transmission eigenfunction is surface-localized, we call it a surface localized eigenstate (SLE). According to \eqref{eq:zeros_derivative}, \eqref{eq:choices01}, we derive that under the condition
%\begin{equation}\label{eq:k_times_n}
%k\cdot \mathbf{n} >m \to \infty,
%\end{equation}
%and $\mathbf{n}>1$, there exists a subsequence $\{k_{m,s}\}$ of  the transmission eigenvalues $\{k_{\ell}\}$ to  \eqref{eq:trans1}, such that $\{v_{k_{\ell_m}}\}$ are SLEs. The similar result holds when $0<\mathbf{n}<1$.
%\end{rem}

\subsection{General case}\label{subsect:gene}

In this part, we consider the case where $\Omega$ is not necessarily of radial geometry and the refractive index $\mathbf{n}$ might be variable (but real). By extensive numerics, we shall verify that the spectral properties in Theorems~\ref{thm:2.1} and \ref{thm:2.2} still hold in the general case. Moreover, through our numerical experiments, we can find more quantitative properties of the SLEs. In principle, we shall see that the SLEs are topologically robust against large deformation or even twisting of the material interface $\partial\Omega$, and moreover the behaviours of the SLEs are mainly related to the refractive index $\mathbf{n}$ in a neighbourhood of $\partial\Omega$. It is also observed that a pair of transmission eigenfunctions $w$ and $v$ cannot be SLEs simultaneously. Furthermore, the occurrence of SLEs can be more often if $\mathbf{n}$ is sufficiently large locally around $\partial\Omega$ (for $v$) or sufficiently small locally around $\partial\Omega$ (for $w$).

To begin with, we note that if $(w, v)$ is a pair of eigenfunctions to \eqref{eq:trans1}, so is $\alpha\cdot(w, v)$ for any $\alpha\in\mathbb{C}\backslash\{0\}$. Hence, throughout the rest of this section, we shall normalize $v$ (or $w$ in some occasions). Moreover, we note the following scaling property of $k$ in \eqref{eq:trans1} with respect to the size of $\Omega$: for $\rho\in\mathbb{R}_+$, $\rho\cdot k$ is an eigenvalue to \eqref{eq:trans1} associated with $\Omega_\rho:=\frac 1 \rho\Omega$. Hence, we always assume that $\mathrm{diam}(\Omega)\sim 1$ in order to calibrate our study. Next, we present some typical numerical results to verify and demonstrate the SLEs in different scenarios.
   We mainly discuss the case $\mathbf{n}>1$ and briefly remark the case $0<\mathbf{n}<1$.

%
% then the transmission eigenfunction $v$ must be localized around $\partial\Omega$. In fact, in the case that $n$ is a variable positive function, \eqref{eq:cond1} can be relaxed to be held only in a neighbourhood of $\partial\Omega$. According to \eqref{eq:cond1}, SLEs occur in the following two generic scenarios. First, the material parameter $n$ is sufficiently large, and then $v$ is a SLE even for a low wavenumber $k$ (being a transmission eigenvalue). This corresponds to the case that one has a high-contrast medium within $\Omega$ (the medium outside $\Omega$ possesses an $n\equiv 1$). Second, the material parameter is a regular one, namely $n\sim 1$, then the transmission eigenfunction $v$ associated with a high wavenumber $k$ is an SLE. The rigorous justification of the above localizing result is highly technical and lengthy, and shall be presented in a forthcoming theoretical paper \cite{CHLW}. Next, we present some typical numerical results to verify and demonstrate the SLEs in different scenarios.

\subsection{SLEs for high-contrast mediums}\label{sect:2.3}

First, we consider the scenario that $\mathbf{n}$ is sufficiently large, which corresponds to the case that a high-contrast medium is located inside $\Omega$ (the medium outside $\Omega$ possesses $\mathbf{n}\equiv 1$). In Fig.~\ref{fig:circle}, we calculate a transmission eigenvalue $k=1.0080$ for $\Omega$ being a unit disk and plot the corresponding eigenfunctions $w$ and $v$. It is clearly seen that $v$ is an SLE. However, it is pointed out that the eigenfunction $w$ is not a SLE. That is, $w$ and $v$ are not SLEs simultaneously. It is emphasized that this is only a numerical observation and there might exist a pair of transmission eigenfunctions which are SLEs simultaneously, which deserves further investigation.  Moreover, in Fig.~\ref{fig:circle}, we note that $\mathrm{diam}(\Omega)$, being $2$, is much smaller than the underlying wavelength, being $2\pi/k\approx 2\pi$. Such an observation is critical for our subsequent development of the super-resolution imaging scheme. Fig.~\ref{fig:circle}, (c) presents the SLE of a triangle and in particular, in (d), (e), and (f) we note significant localization phenomena at the concave part of $\partial\Omega$, which is also a critical ingredient for our subsequent development of the super-resolution wave imaging.
\begin{figure}[h]
\centering
\subfigure[$w$]{\includegraphics[width=0.32\textwidth]{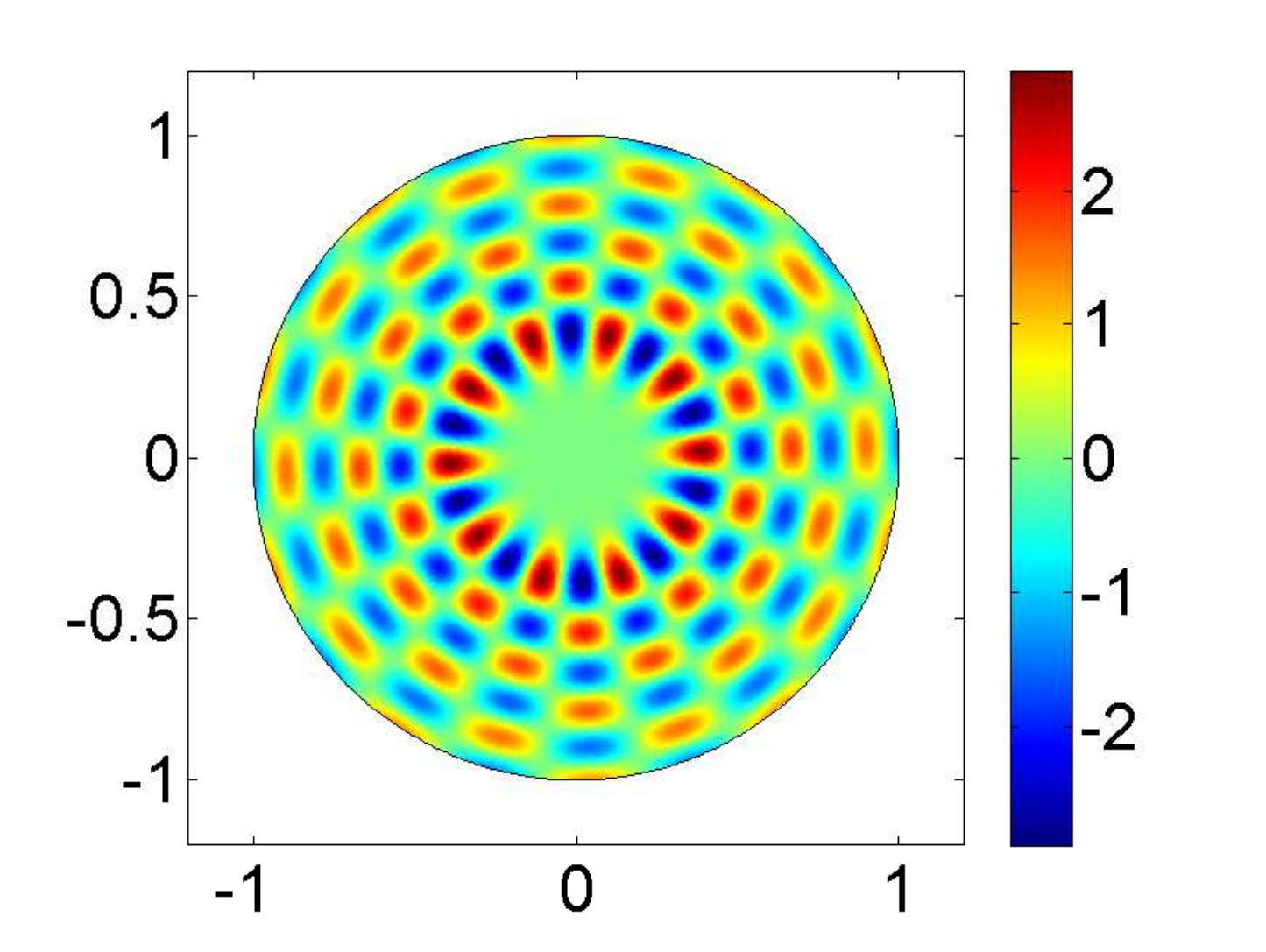}}% lleHere is how to import EPS art
\subfigure[$v$]{\includegraphics[width=0.32\textwidth]{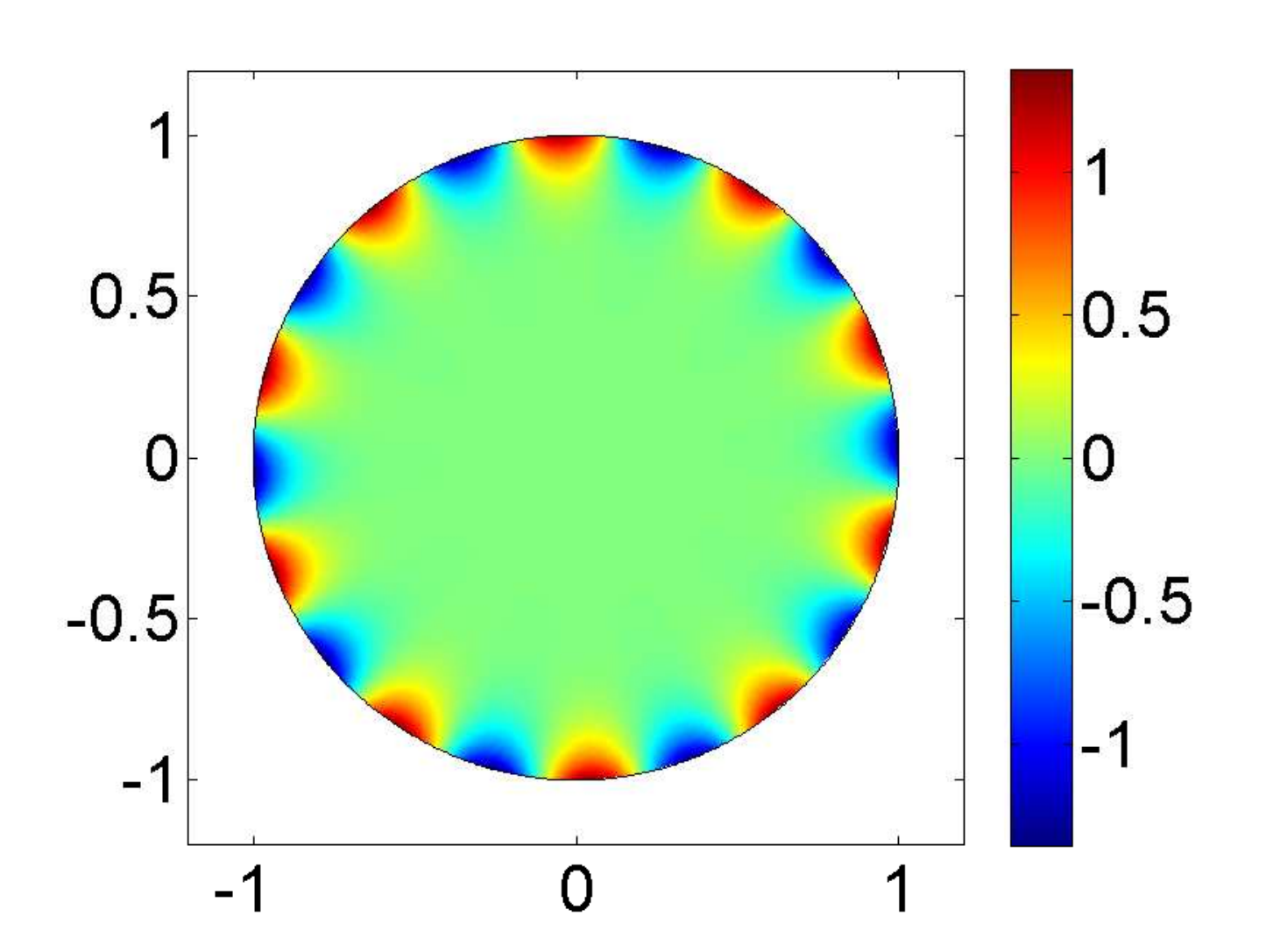}}% lleHere is how to import EPS art
\subfigure[$k=1.1932$]{\includegraphics[width=0.32\textwidth]{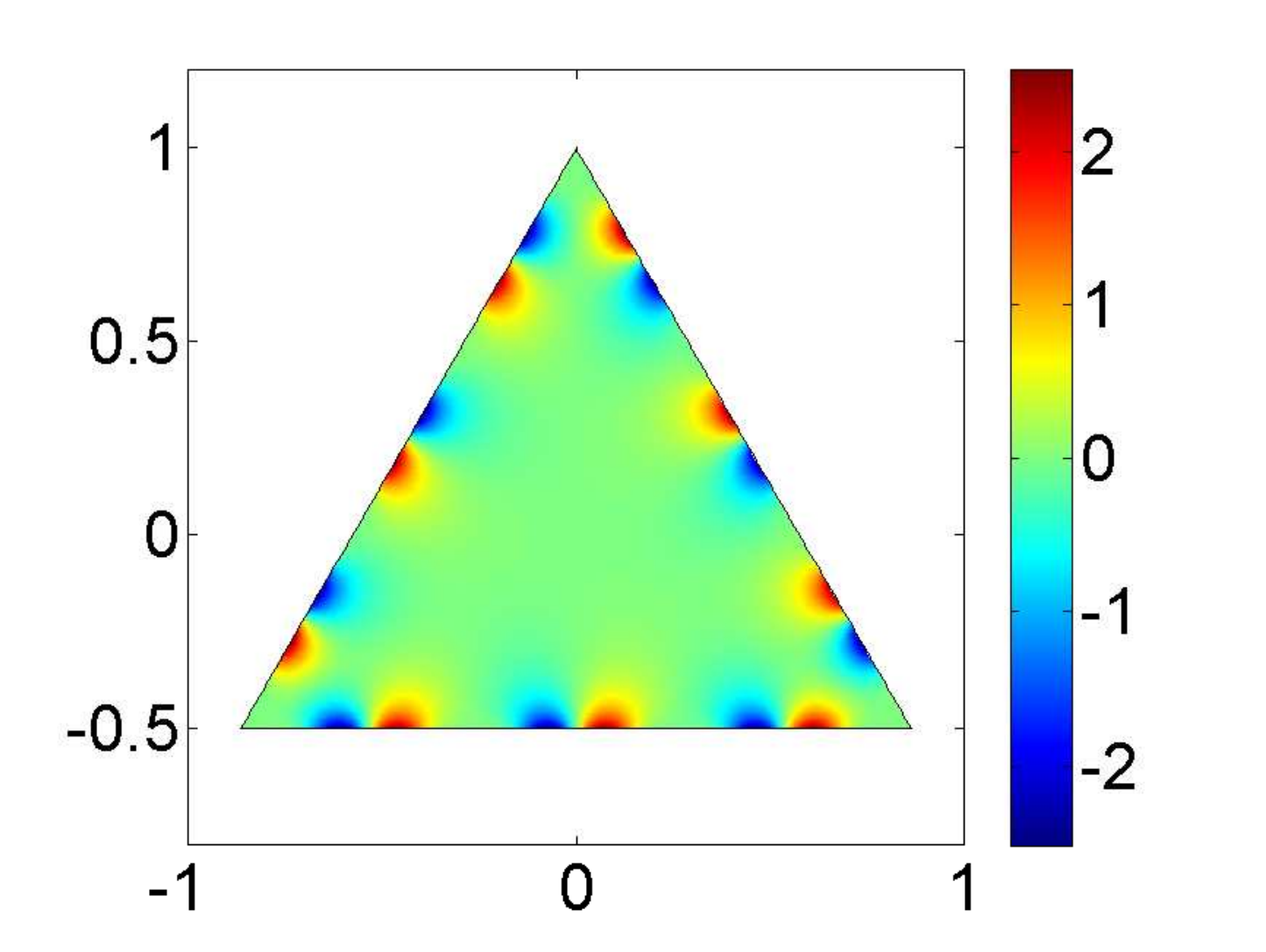}}\\
\subfigure[$k=1.0007$]{\includegraphics[width=0.32\textwidth]{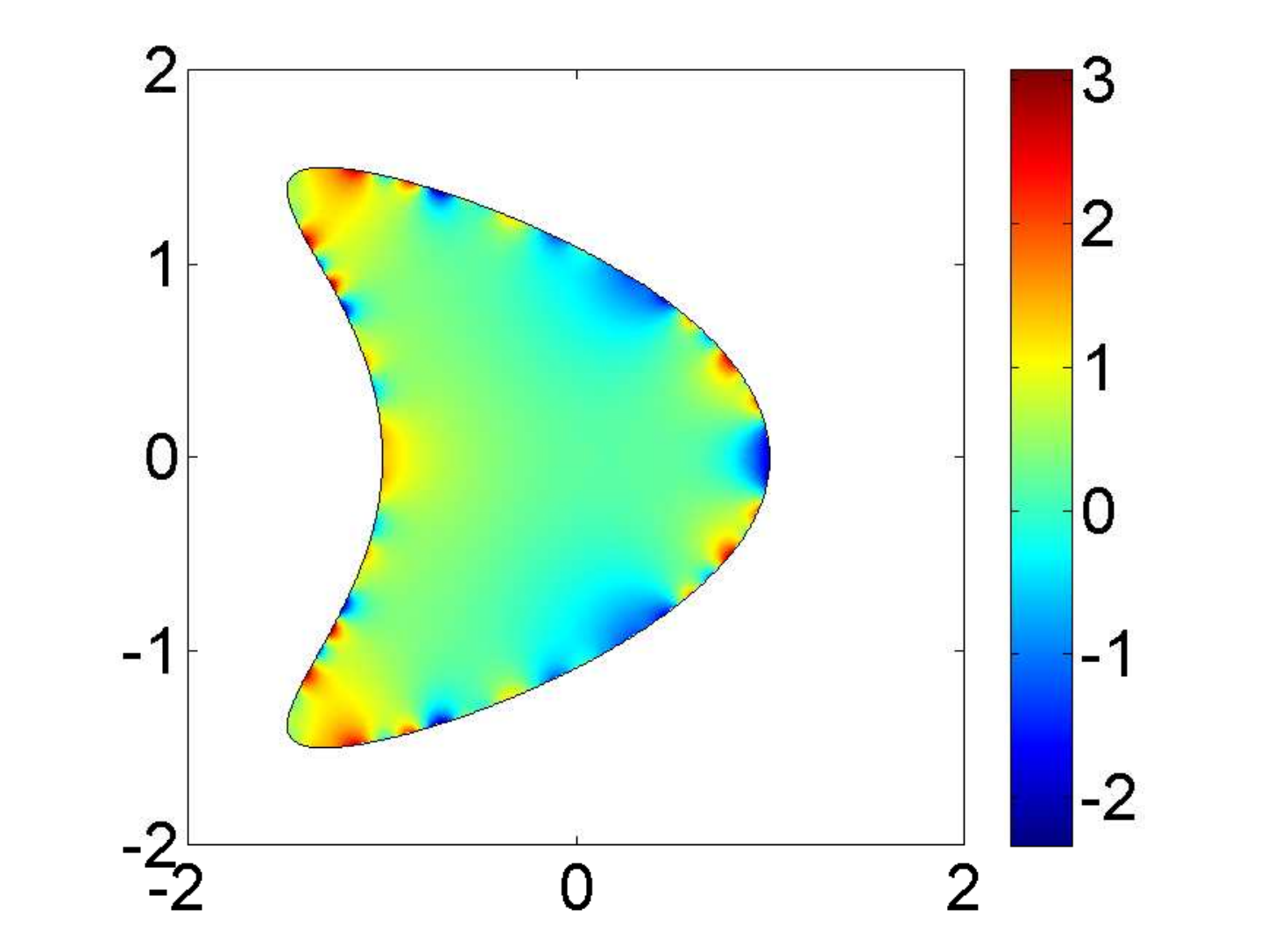}}
\subfigure[$k=1.1370$]{\includegraphics[width=0.32\textwidth]{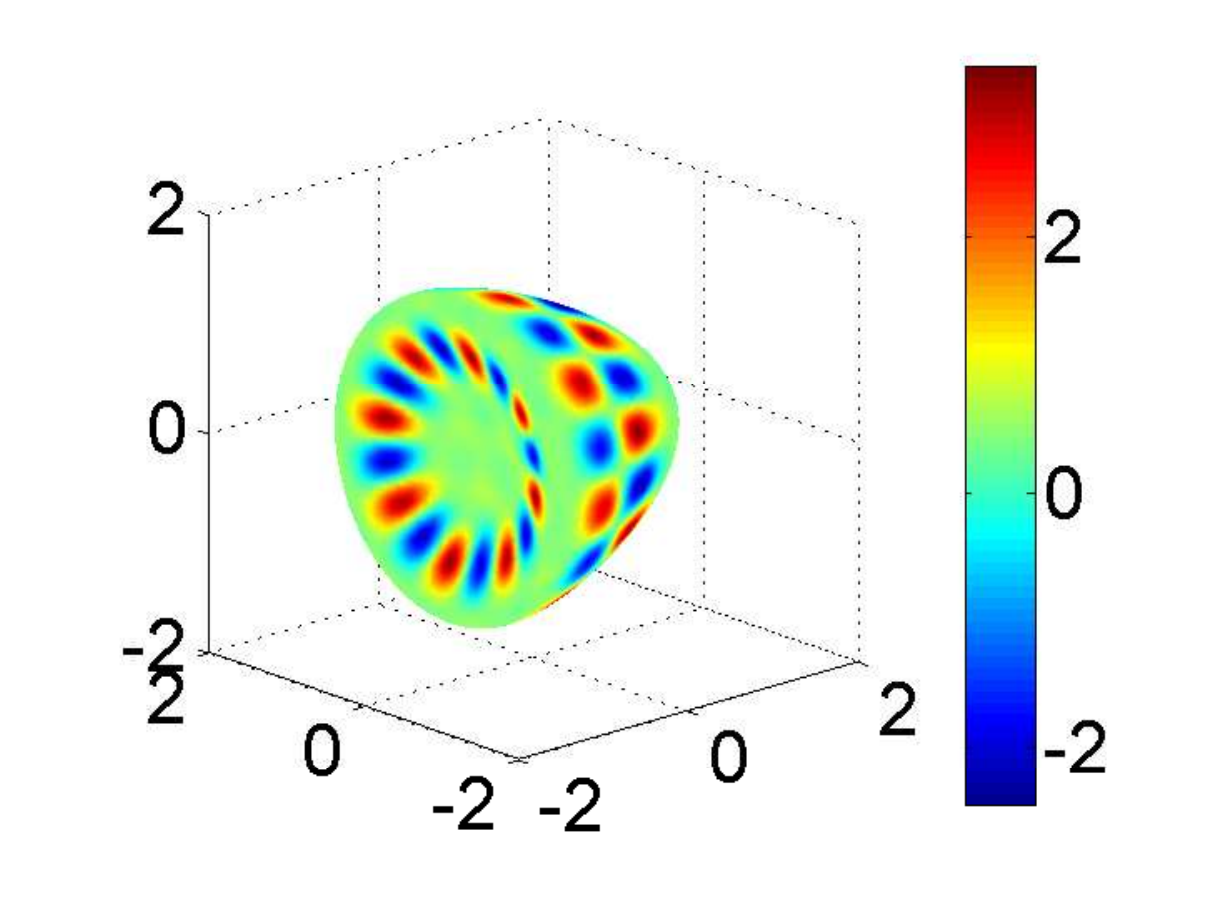}}
\subfigure[$k=1.1370$]{\includegraphics[width=0.33\textwidth]{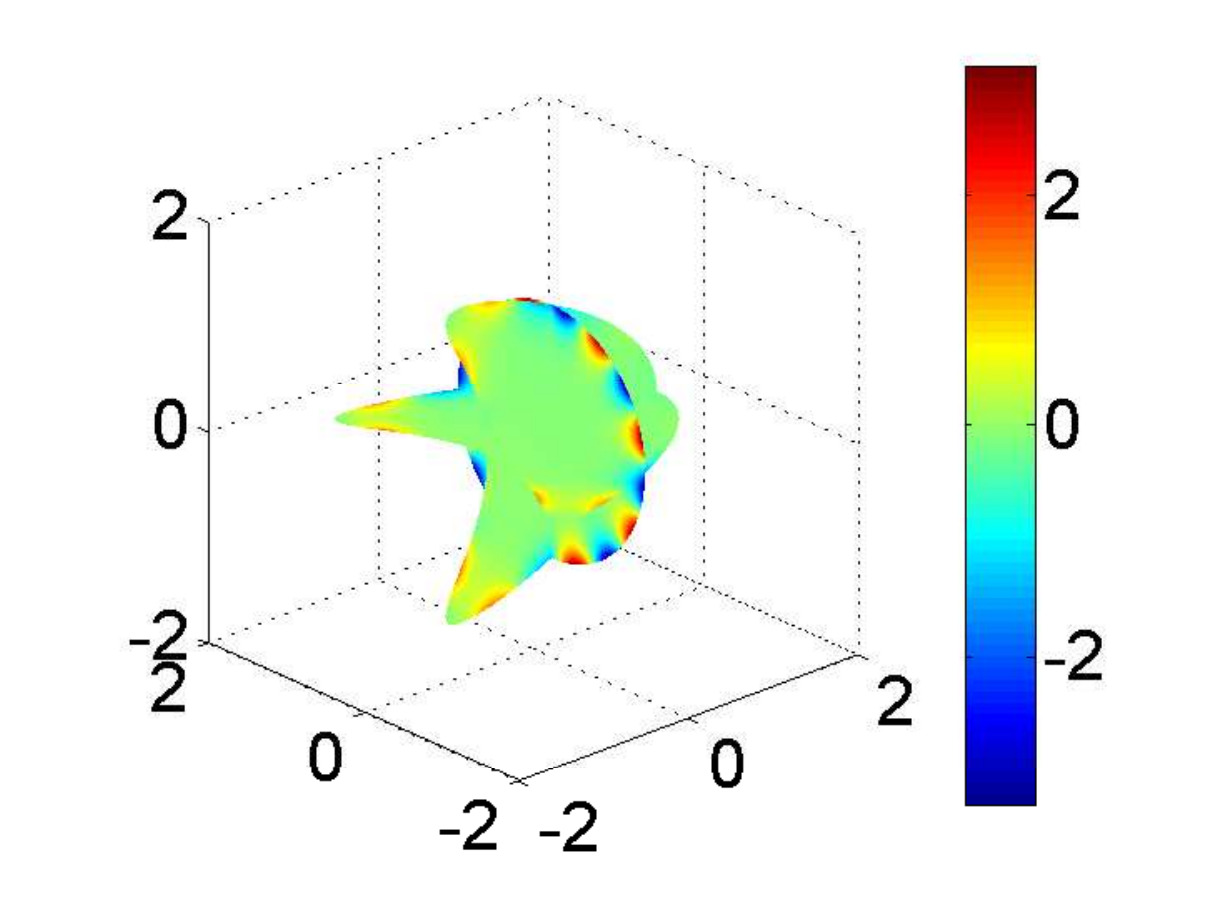}}\\
\caption{\label{fig:circle} (a)\&(b): eigenfunctions $w$ and $v$ to \eqref{eq:trans1} associated with $\mathbf{n}=30$, where $k=1.0080$; (c)\&(d): eigenfunctions $v$'s to \eqref{eq:trans1} of different shapes with $\mathbf{n}=30$; (e)\&(f): eigenfunction $v$ to \eqref{eq:trans1} with $\mathbf{n}=10$. }
\end{figure}

%\begin{figure}[!htpb]
%\centering
%\subfigure[$k=1.1932$]{\includegraphics[width=0.32\textwidth]{triangle}}% lleHere is how to import EPS art
%\subfigure[$k=2.4454$]{\includegraphics[width=0.33\textwidth]{cone}}% lleHere is how to import EPS art
%\subfigure[$k=2.4454$]{\includegraphics[width=0.33\textwidth]{cone_slice}}\\% lleHere is how to import EPS art
%\subfigure[$k=1.0007$]{\includegraphics[width=0.32\textwidth]{kite}} % lleHere is how to import EPS art
%\subfigure[$k=1.1370$]{\includegraphics[width=0.33\textwidth]{kite3D}} % lleHere is how to import EPS art
%\subfigure[$k=1.1370$]{\includegraphics[width=0.33\textwidth]{kite3D_slice}}\\ % lleHere is how to import EPS art
%\caption{\label{fig:triangle} Transmission eigenfunctions $v$'s to \eqref{eq:trans1} associated with different shapes and eigenvalues $k$'s: $n=900$ for the 2D cases, and $n=100$ for the 3D cases. Figures~(c) and (f) are slice plottings of (b) and (e), respectively. }
%\end{figure}

\subsection{SLEs for high-wavenumber modes}

Next, we consider the case that $\mathbf{n}$ is relatively small, namely $\mathbf{n}\sim 1$.
Fig.~\ref{fig:highenergy} presents several examples in both 2D and 3D.

\begin{figure}[!htpb]
\centering
\subfigure[$k=7.1925$]{\includegraphics[width=0.4\textwidth]{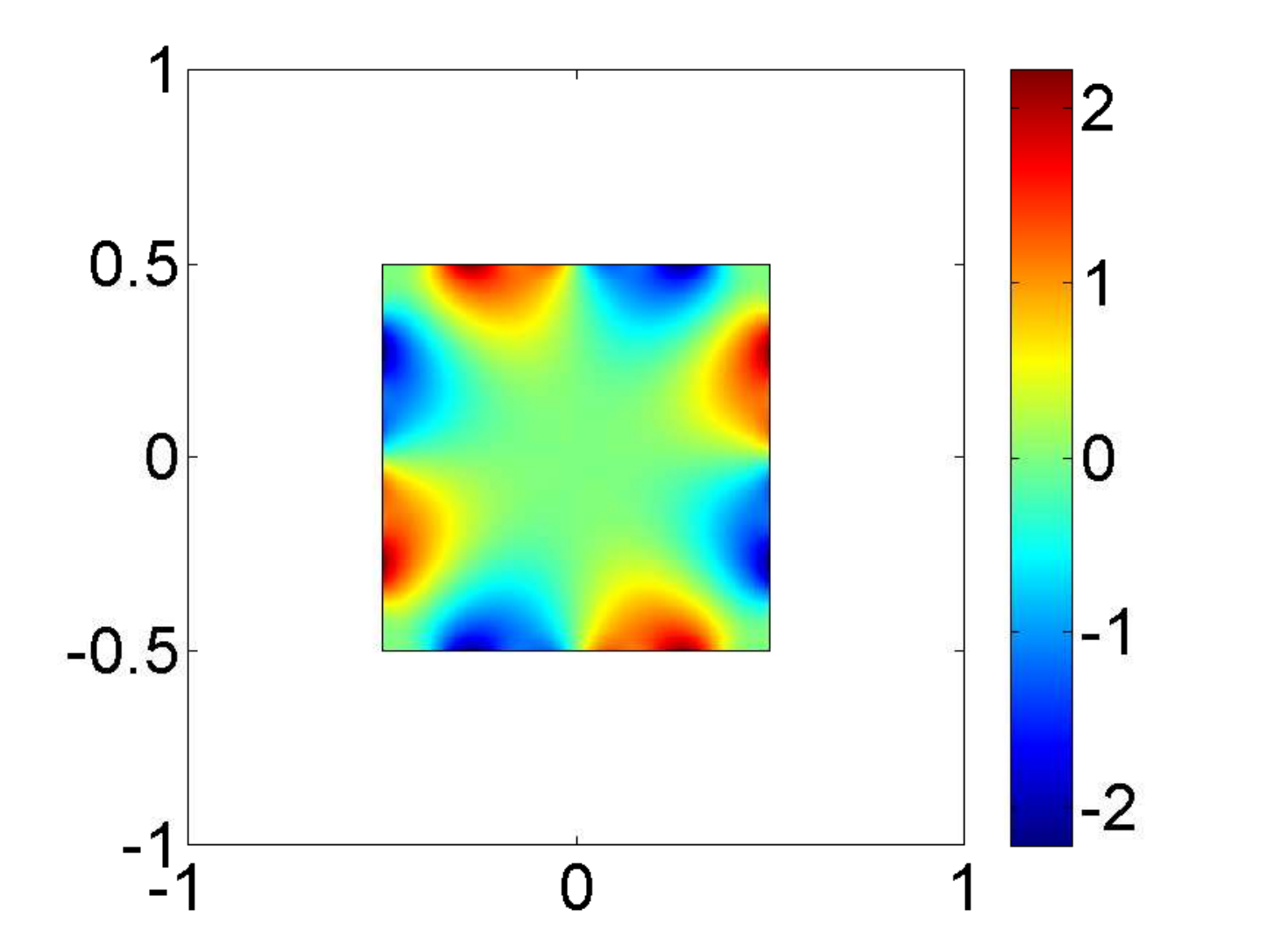}}% lleHere is how to import EPS art
\subfigure[$k=7.9604$]{\includegraphics[width=0.4\textwidth]{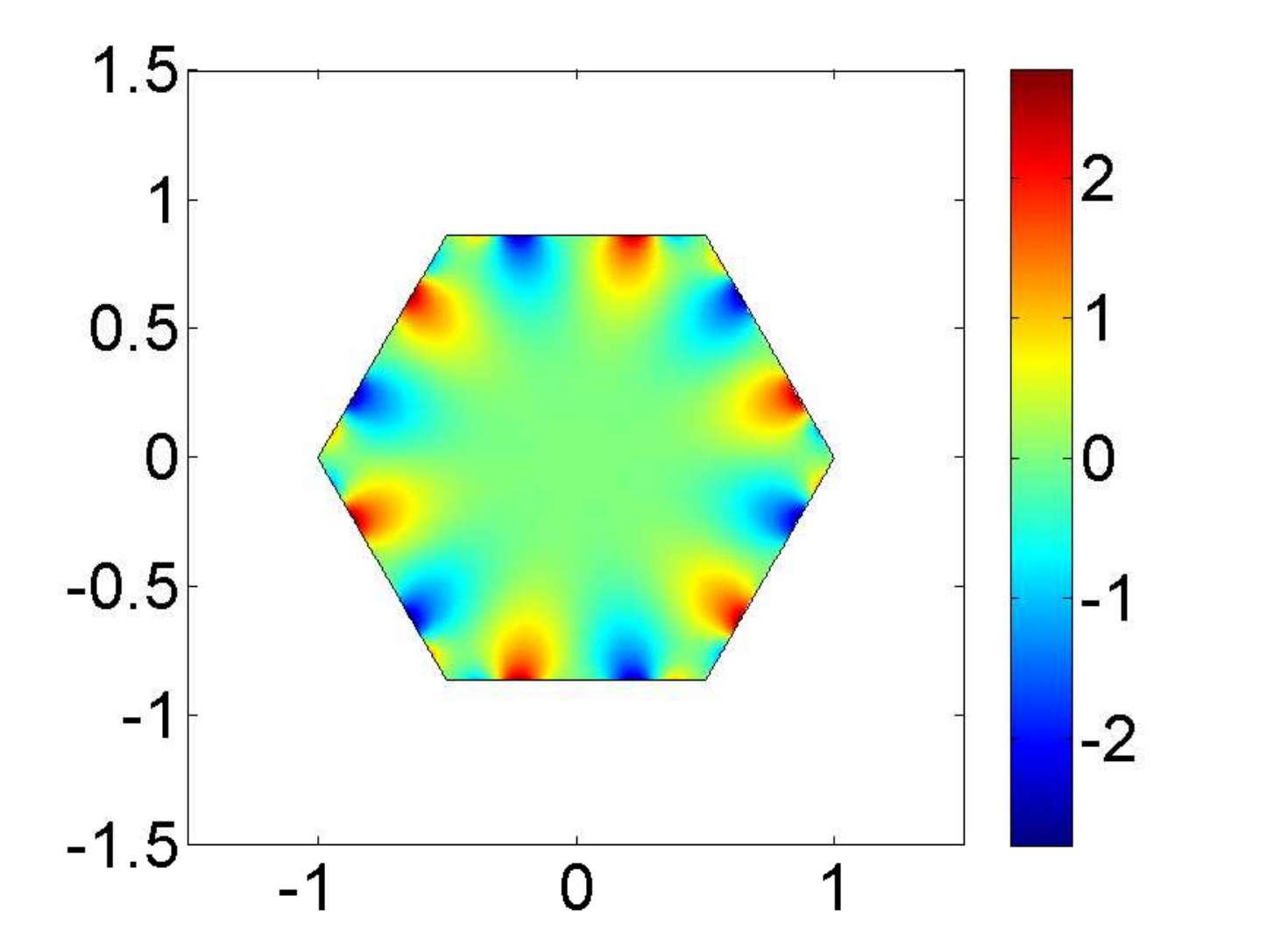}}\\% lleHere is how to import EPS art
\subfigure[$k=7.1023$]{\includegraphics[width=0.4\textwidth]{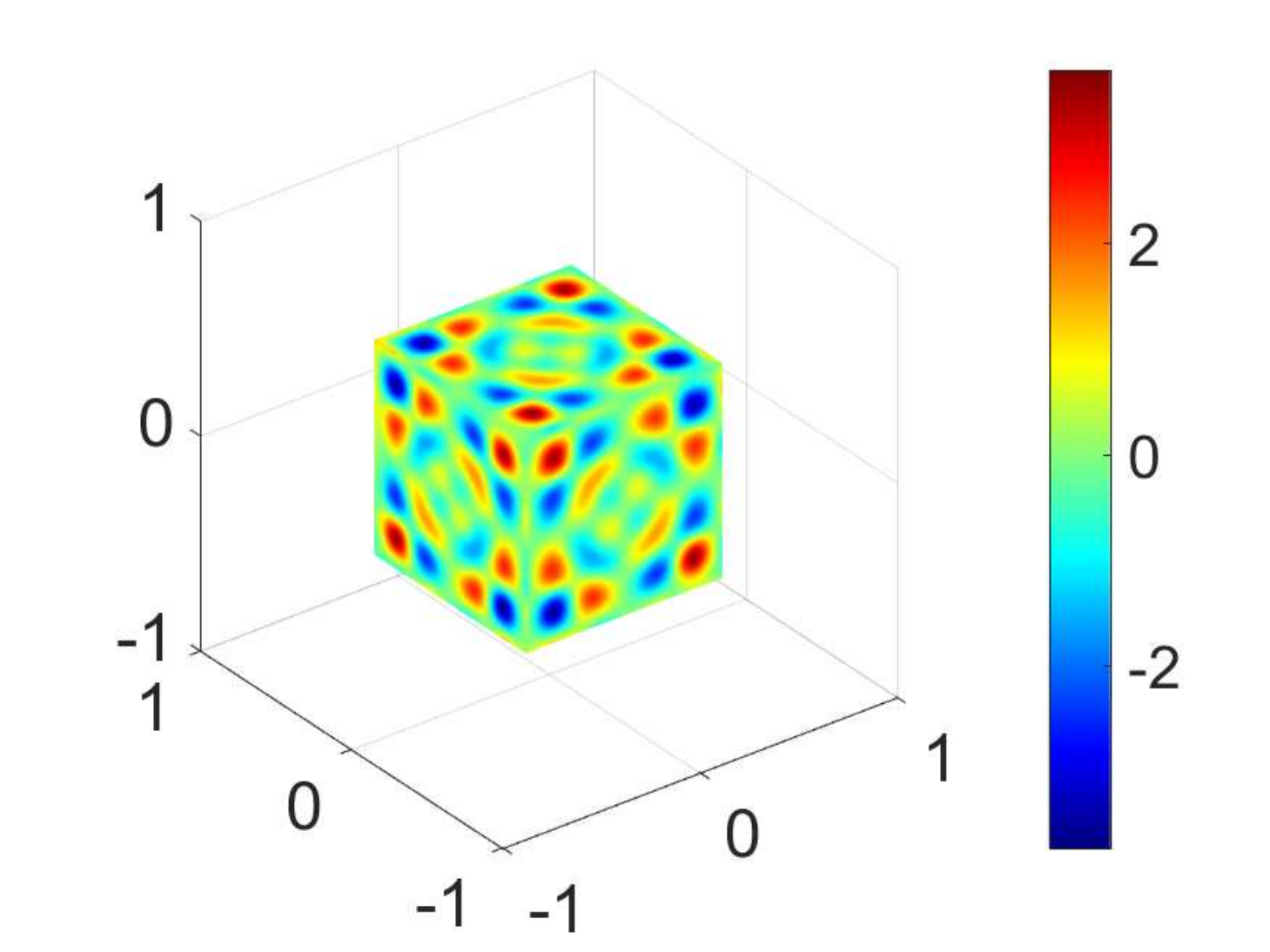}}% lleHere is how to import EPS art
\subfigure[$k=7.0818$]{\includegraphics[width=0.4\textwidth]{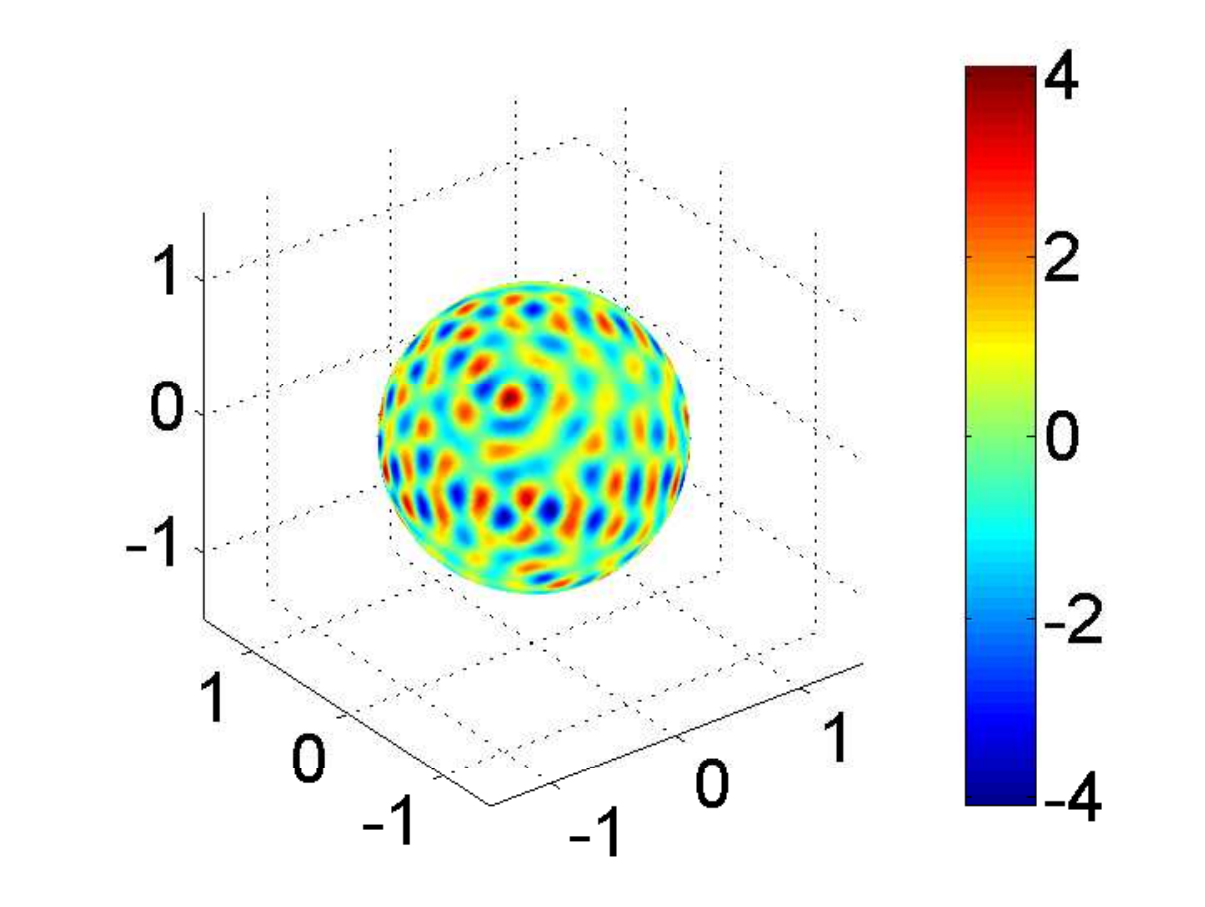}}\\ % lleHere is how to import EPS art
\caption{\label{fig:highenergy} Transmission eigenfunctions $v$'s to \eqref{eq:trans1} associated with $\mathbf{n}=4$, for different $\Omega$'s and $k$'s. }
\end{figure}

\subsection{Topological robustness of the existence of SLEs}

The existence of the SLEs is topologically very robust against large deformation or even twisting of the material interface $\partial \Omega$; see Fig. \ref{fig:robust}.

\begin{figure}[h]
\centering
\subfigure[$k=1.2802$]{\includegraphics[width=0.4\textwidth]{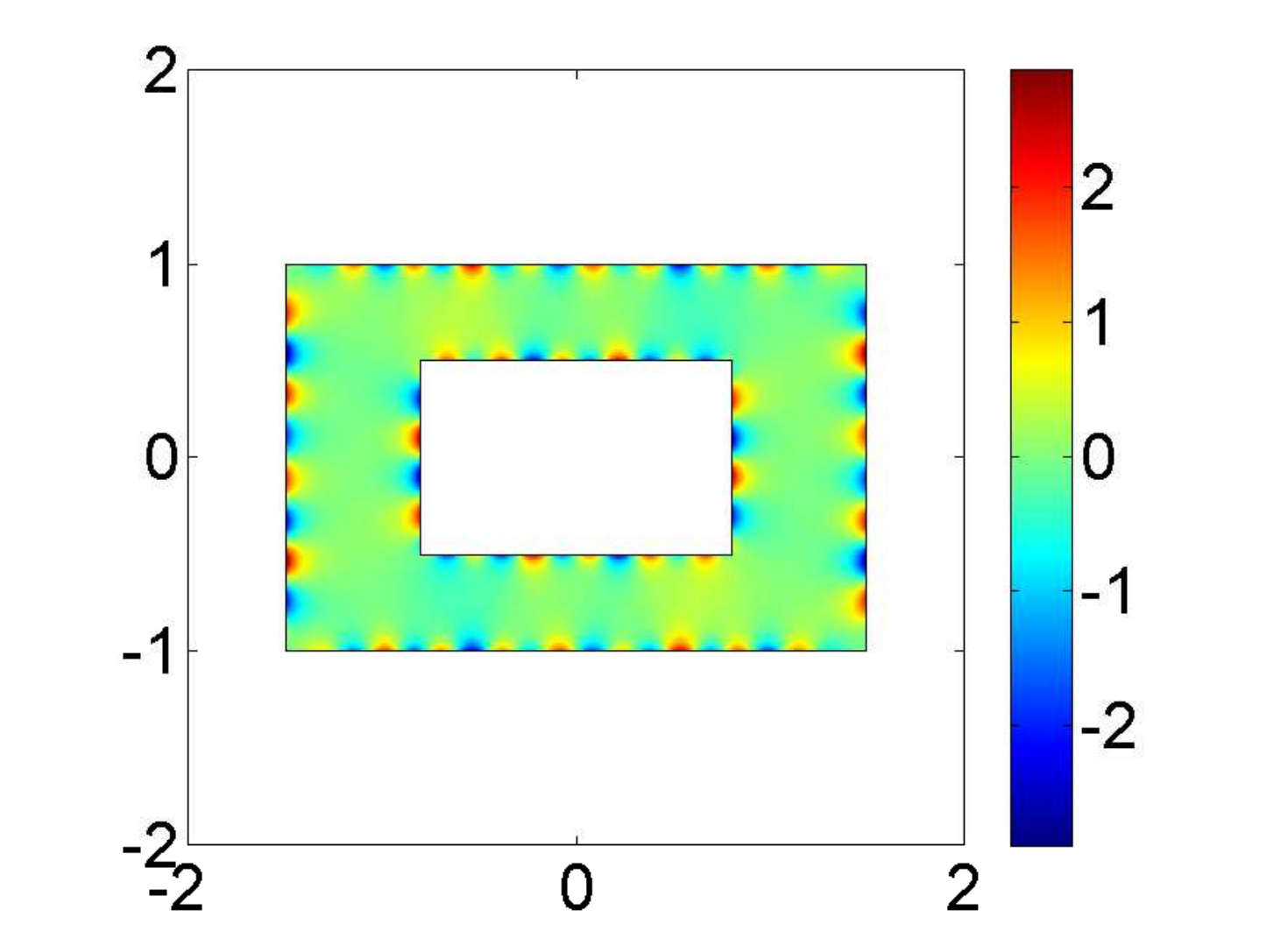}}% lleHere is how to import EPS art
\subfigure[$k=1.9604$]{\includegraphics[width=0.4\textwidth]{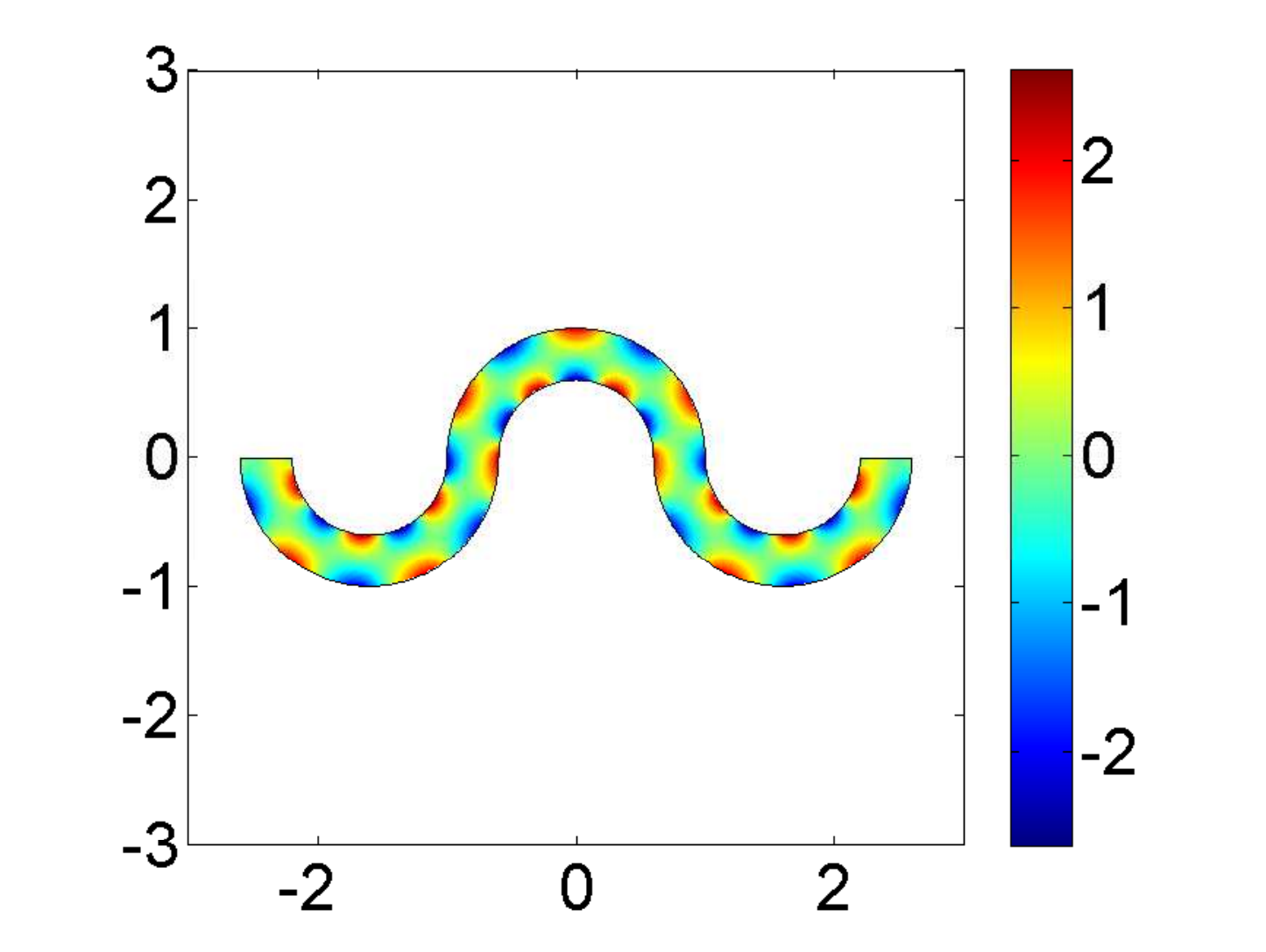}}\\ % lleHere is how to import EPS art
\subfigure[$k=1.0007$]{\includegraphics[width=0.4\textwidth]{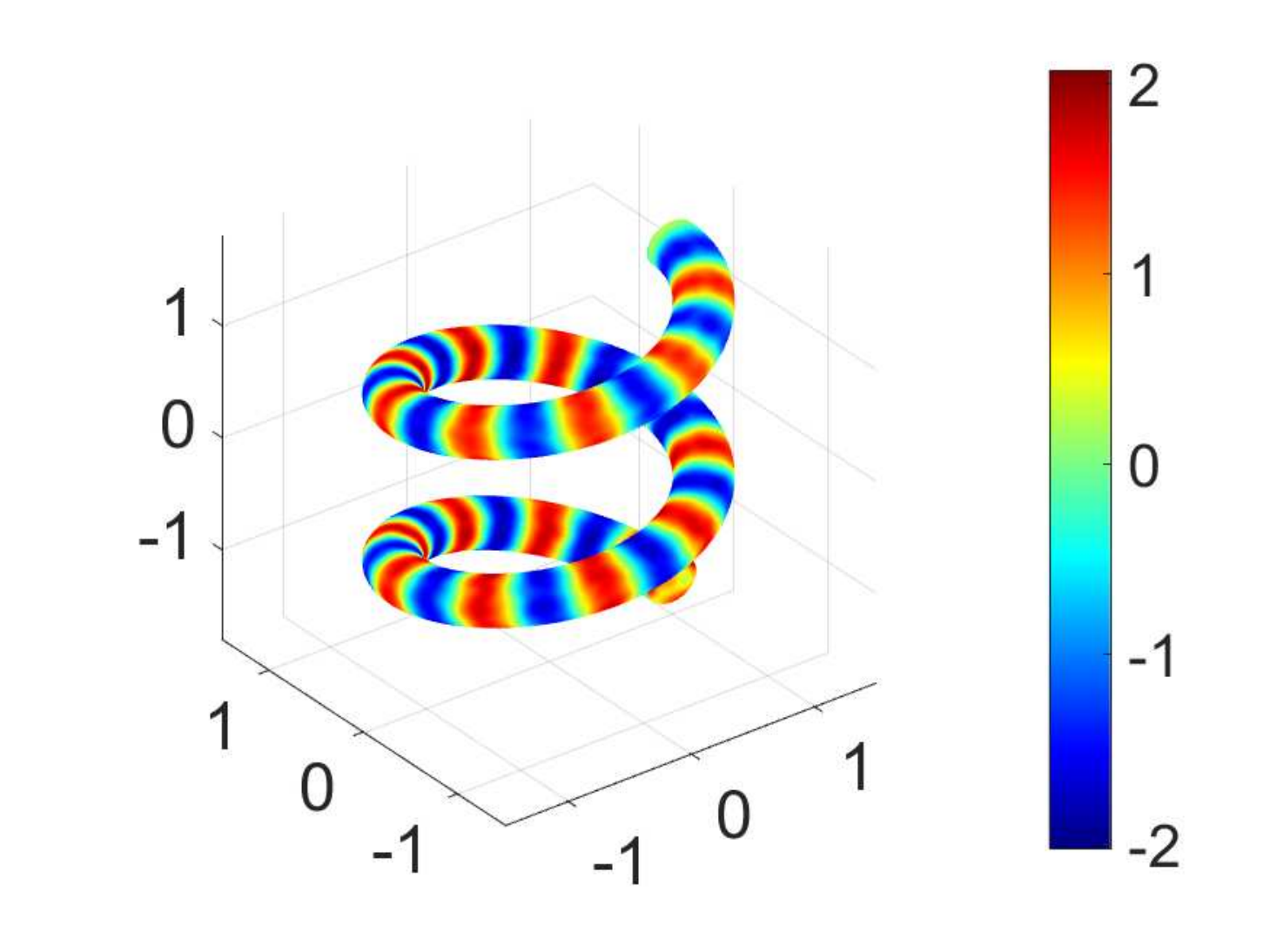}}% lleHere is how to import EPS art
\subfigure[$k=1.0133$]{\includegraphics[width=0.4\textwidth]{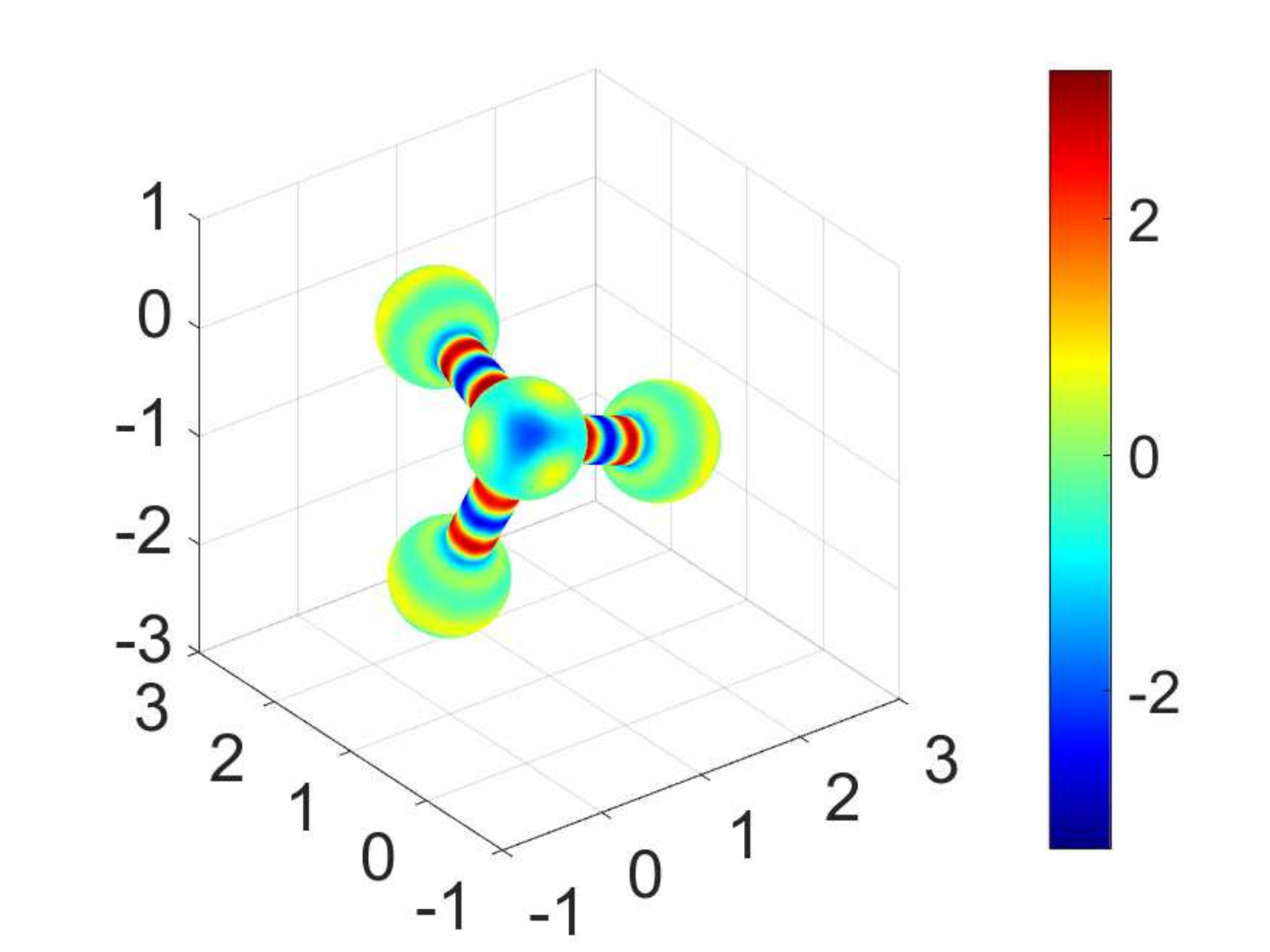}}\\ % lleHere is how to import EPS art
\caption{\label{fig:robust} The existence of SLEs is topologically robust. Here, $\mathbf{n}=20$ for all cases.
}
\end{figure}

\subsection{SLEs for variable refractive inhomogeneities and coated objects}\label{sect:2.6}
As remarked earlier, the SLEs also exist for variable refractive inhomogeneities. In Fig.~\ref{fig:medium_layer}~(a), the eigenfunction $v$ is associated with $\mathbf{n}=30$ in the outside thin layer and $\mathbf{n}=4$ in the inside triangle. It is emphasized that the outside layer is not required to be very thin in order to exhibit the SLEs.
Indeed, as long as $\mathbf{n}$ is sufficiently large locally around $\partial\Omega$, the SLEs can be found even for relatively small eigenvalues (cf. Fig.~\ref{fig:circle}). This specific example shall be used again in our subsequent study. Fig.~\ref{fig:medium_layer}~(b), corresponds to a coated object, where the inside kite-domain is a sound-soft obstacle. That is, in \eqref{eq:trans1}, $w$ does not exist in the inside kite-domain and we impose a zero Dirichlet condition of $w$ on the boundary of the inside kite-domain.

\begin{figure}[h]
%\subfigure[ $k=2.0521$]
%{\includegraphics[width=0.32\textwidth]{triangle_layer_u_K1}}
%\subfigure[$k=2.1900$]
%{\includegraphics[width=0.32\textwidth]{triangle_layer_u_K2}} %
%\subfigure[$k=2.4454$]
%{\includegraphics[width=0.32\textwidth]{triangle_layer_u_K3}}\\%
%\subfigure[$k=2.0521$]
%{\includegraphics[width=0.32 \textwidth]{triangle_layer_v_K1}}%
%\subfigure[$k=2.1900$]
%{\includegraphics[width=0.32\textwidth]{triangle_layer_v_K2}} %
\subfigure[$k=2.4454$]
{\includegraphics[width=0.32\textwidth]{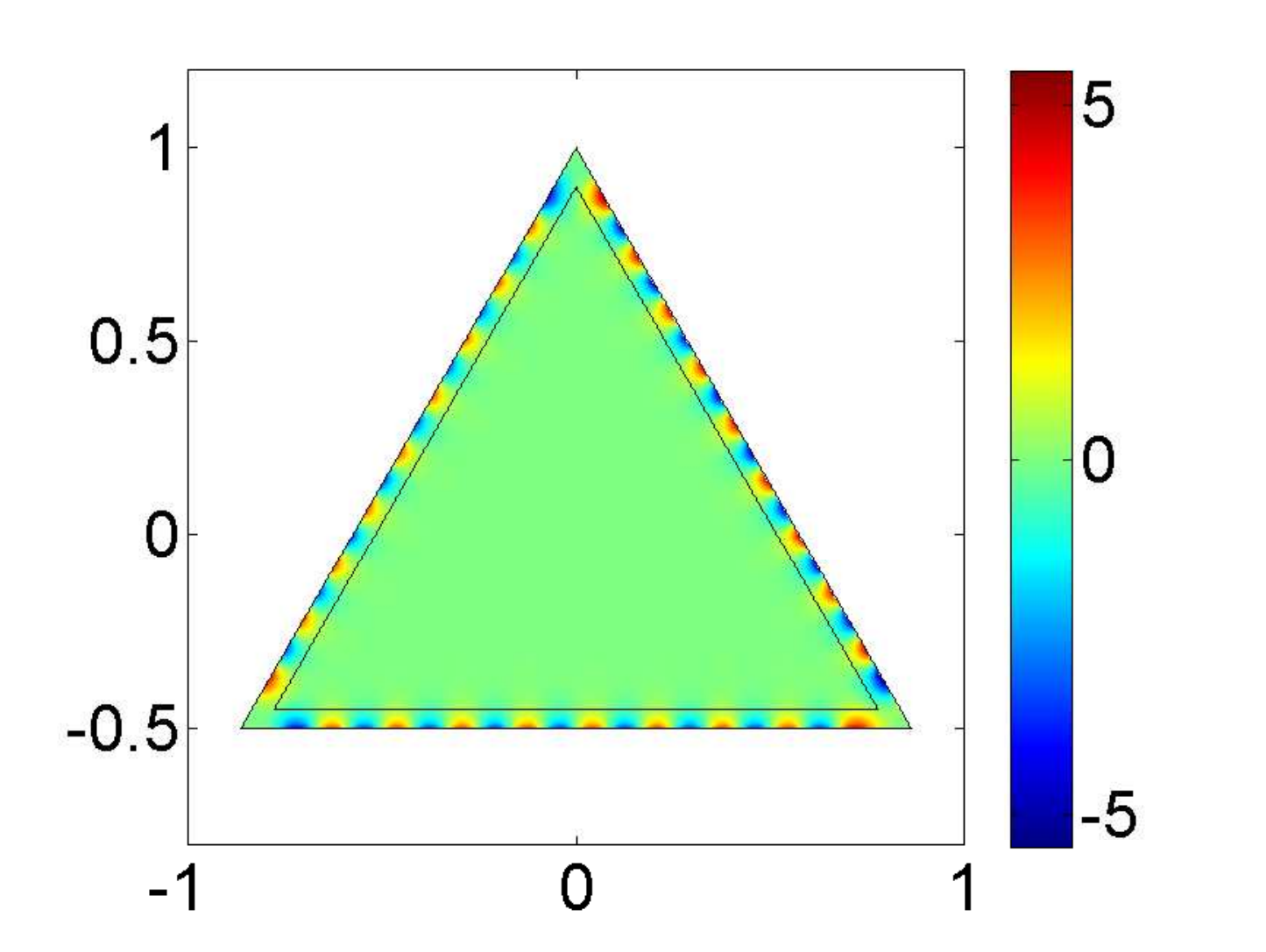}}%
\subfigure[$k=1.0001$]
{\includegraphics[width=0.32 \textwidth]{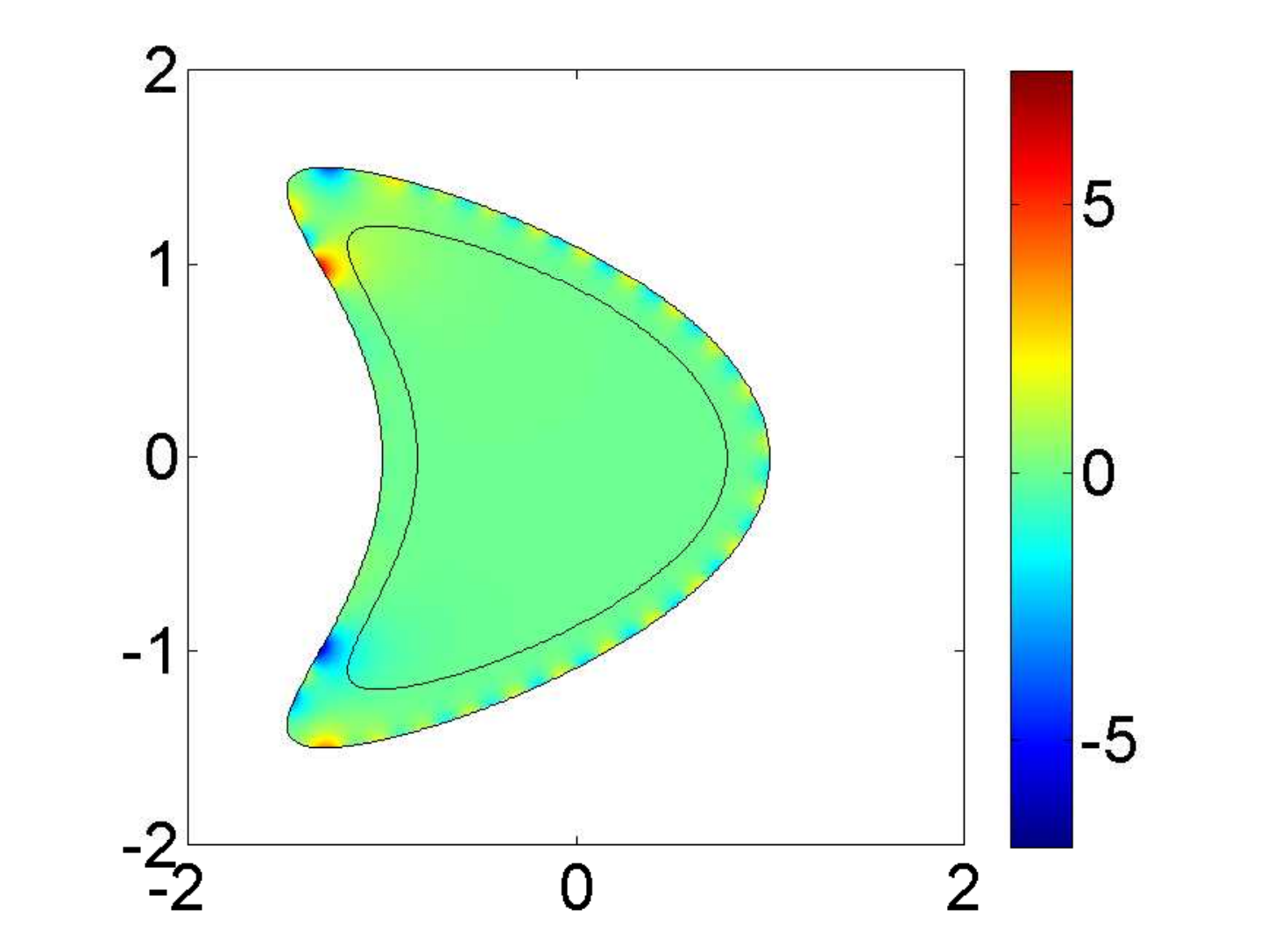}}\\
%\subfigure[$k=1.2004$]
%{\includegraphics[width=0.32\textwidth]{kite_layer_v_K2}} %
%\subfigure[$k=1.3684$]
%{\includegraphics[width=0.32\textwidth]{kite_layer_v_K3}}\\
\caption{\label{fig:medium_layer} Transmission eigenfunctions $v$'s. (a): $\mathbf{n}=30$ in the outside layer and $\mathbf{n}=4$ in the inside triangle; (b): $\mathbf{n}=30$ in the outside layer and the inside kite is a sound-soft obstacle. }
\end{figure}

Finally, we would like to remark that if $0<\mathbf{n}<1$, all the surface localization results presented in the above numerical examples still hold with $v$ replaced by $w$.

\section{Super-resolution wave imaging}

In this section, we consider an interesting and practically important application of the SLEs presented in the previous section. We propose an inverse scattering scheme that makes use of the longly neglected interior resonant modes to recover the unknown or inaccessible scatterer. It turns out that the proposed scheme can produce strikingly high imaging resolution compared to many existing imaging schemes. Let $\Omega \subset \mathbb{R}^d, d=2,3$, be a bounded domain with a Lipschitz boundary $\partial \Omega$ and a connected complement $\mathbb{R}^d\backslash\overline{\Omega}$. Let $\nu$ denote the unit outward normal to $\partial \Omega$. Throughout the rest of this section, we shall assume that the refractive index $\mathbf{n}^2$ is a real-valued bounded function such that $\mathbf{n}^2(x)\equiv 1$ for $x\in \mathbb{R}^d \setminus \overline{\Omega}$ and $1/|\mathbf{n}^2-1| \in L^{\infty}(\Omega)$. We take the incident field $u^i$ to be a time-harmonic plane wave of the form
\[
u^i:=u^i(x,\theta,k)=\mathrm{e}^{\mathrm{i}k x\cdot \theta}, \quad  x \in \mathbb{R}^d,
\]
where $\mathrm{i}=\sqrt{-1}$ is the imaginary unit, $k=\omega / c$ the wavenumber, $\omega \in \mathbb{R}_{+}$ and $c \in \mathbb{R}_{+}$ the angular frequency and sound speed, respectively, $\theta \in \mathbb{S}^{d-1}$ the direction of propagation and $\mathbb{S}^{d-1}:=\{x \in \mathbb{R}^d: |x|=1 \}$ is the unit sphere in $\mathbb{R}^d$. Clearly, the incident field $u^i$ satisfies the Helmholtz equation
\[
\Delta u^i + k^2 u^i =0\quad \text{in} \  \mathbb{R}^d.
\]
Physically, the presence of the scatterer $\Omega$ interrupts the propagation of the incident wave $u^i$, giving rise to the scattered field $u^s$. Let $u:=u^i+u^s$ denote the total wave field. The forward scattering problem is modeled by the following system
\begin{equation}\label{eq:forwardpro}
\left\{
\begin{array}{ll}
\Delta u+k^2 \mathbf{n}^2(x)u =0  & \text{in} \  \mathbb{R}^d, \medskip \\
u = u^i+u^s  & \text{in} \  \mathbb{R}^d, \medskip \\
\displaystyle \lim\limits_{r\to\infty} r^{\frac{d-1}{2}}\left(\frac{\partial u^s}{\partial r}-\mathrm{i}ku^s\right)=0,
\end{array}
\right.
\end{equation}
where $r=|x|$ and the last limit in \eqref{eq:forwardpro} characterizes the outgoing nature of the scattered wave field $u^s$. The well-posedness of the scattering system \eqref{eq:forwardpro} is established \cite{CK}, and in particular, there exists a unique solution $u\in H_{loc}^2(\mathbb{R}^d)$. Furthermore, the scattered field has the following asymptotic expansion:
\[
u^s(x,\theta,k)=\frac{\mathrm{e}^{\mathrm{i}\frac{\pi}{4}}}{\sqrt{8k\pi}}\left(\mathrm{e}^{-\mathrm{i}\frac{\pi}{4}} \sqrt{\frac{k}{2\pi}}  \right)^{d-2}\frac{\mathrm{e}^{\mathrm{i}kr}}{r^{\frac{d-1}{2}}} \bigg\{u^{\infty}(\hat{x},\theta,k)+\mathcal{O}\bigg(\frac{1}{r} \bigg) \bigg\} \quad \text{as} \ r\to\infty,
\]
which holds uniformly for all directions $\hat{x}:=x/|x|\in \mathbb{S}^{d-1}$. The complex-valued function $u^{\infty}(\hat{x},\theta,k)=u^{\infty}(\hat{x},\mathrm{e}^{\mathrm{i}kx\cdot \theta  } ) $, defined on the unit sphere $\mathbb{S}^{d-1}$, is known as the far-field pattern of $u$, which encodes the information of the refractive index $\mathbf{n}^2$. We are concerned with the inverse problem of imaging the support of the inhomogeneity, namely $\Omega$, by knowledge of $u^{\infty}(\hat{x},\theta,k)$ for $\hat{x},\theta\in \mathbb{S}^{d-1}$ and $k\in I:=(\kappa_0,\kappa_1)$, which is an open interval in $\mathbb{R}_{+}$. It can be recast as the following nonlinear operator equation
\begin{equation}\label{eq:ip1}
\mathcal{F}(\Omega, \mathbf{n})=u^{\infty}(\hat x, \theta, k),\quad   \hat x\in\mathbb{S}^{d-1},\ \theta\in\mathbb{S}^{d-1}, \ k\in I,
\end{equation}
where $\mathcal{F}$ is defined by the Helmholtz system \eqref{eq:forwardpro}. Such an inverse problem is a prototypical model for many industrial and engineering applications including medical imaging and nondestructive testing. There is the well-known Abbe diffraction limit for imaging the fine details of $\partial\Omega$ \cite{LLT}. In fact, one has a minimum resolvable distance of $\lambda/(2 \mathcal{N})$, where $\lambda$ and $\mathcal{N}$ stand for the wavelength and numerical aperture respectively. In modern optics, the Abbe resolution limit is roughly about half of the wavelength. Here, based on the use of the SLEs, we develop an imaging scheme that can break the Abbe resolution limit in recovering the fine details of $\partial\Omega$ for \eqref{eq:ip1}, independent of $\mathbf{n}$, in certain scenarios of practical interest.

The proposed imaging scheme consists of three phases. In Phase $\mathrm{I}$, we determine the transmission eigenvalues within the interval $I$ by knowledge of the far-field data in \eqref{eq:ip1}, namely $u_\infty(\hat x, \theta, k)$ for $\hat x\in\mathbb{S}^{d-1},\ \theta\in\mathbb{S}^{d-1}$ and $k\in I$. In Phase $\mathrm{II}$, we determine the corresponding transmission eigenfunctions associated to the transmission eigenvalues computed from Phase $\mathrm{II}$. Finally, in Phase~$\mathrm{III}$, we make use the transmission eigenfunctions from Phase~$\mathrm{II}$ to design an imaging functional which can be used to determine the shape of the medium scatterer, namely $\Omega$.

\subsection{Phase $\mathrm{I}$: determination of transmission eigenvalues}

In this part, we consider the determination of the transmission eigenvalues within the interval $I$ by knowledge of the far-field data in \eqref{eq:ip1}. In fact, this problem has been addressed in \cite{Cakoni10b}. Nevertheless, for completeness and self-containedness, we briefly discuss the main procedure as well as the rationale behind the method. To that end, for any given $z\in \mathbb{R}^d$, we let $\Psi(x,z,k)$ be the fundamental solution \cite{CK} to the PDO $-\Delta-k^2$:
\[
\Psi(x,z,k)=\left\{
\begin{array}{ll}
\frac{\mathrm{i}}{4}H_0^{(1)}(k|x-z|),   &{d=2,} \medskip \\
\displaystyle \frac{1}{4\pi}\frac{\mathrm{e}^{\mathrm{i}k|x-z|}}{|x-z|},   &{d=3,}
\end{array}
\right.
\]
where $H_0^{(1)}$ is the first-kind Hankel function of zeroth-order. Let $\Psi^{\infty}(\hat{x},z,k)$ signify the far-field pattern of $\Psi(x,z,k)$, which is given by
\[
\Psi^{\infty}(\hat{x},z,k)=\left\{
\begin{array}{ll}
\displaystyle \frac{\mathrm{e}^{\frac{\mathrm{i}\pi}{4}}}{\sqrt{8k\pi}}\mathrm{e}^{-\mathrm{i}k\hat{x}\cdot z},   &{d=2,} \medskip \\
\displaystyle \frac{1}{4\pi}\mathrm{e}^{-\mathrm{i}k\hat{x}\cdot z},   &{d=3}.
\end{array}
\right.
\]
The determination of the transmission eigenvalues is based on the so-called linear sampling method (LSM), which is a qualitative method in inverse scattering theory \cite{Colton00}. The core of the LSM is the following far-field equation
\begin{equation}\label{eq:ffeqn}
(F_k g)(\hat{x})=\Psi^{\infty}(\hat{x},z,k),\quad  z\in\mathbb{R}^d,\ \ g\in L^2(\mathbb{S}^{d-1}),
\end{equation}
where $F_k: L^2(\mathbb{S}^{d-1})\mapsto L^2(\mathbb{S}^{d-1})$ is the far-field operator defined by
\begin{equation}\label{eq:ffope}
(F_kg)(\hat{x}):=\int_{\mathbb{S}^{d-1}} u^{\infty}(\hat{x},\theta,k)g(\theta)\, \text{d}s(\theta), \quad \hat{x}\in \mathbb{S}^{d-1}.
\end{equation}
Let $F_k^{\epsilon}$ be the far-field operator corresponding to noisy measurement of the far-field data $u^{\infty,\epsilon}(\hat{x},\theta,k)$, where $\epsilon\in\mathbb{R}_+$ signifies the noise level. Define the Herglotz wave function
\begin{equation}\label{eq:Herglotz}
v_{g,k}(x):= H_k g(x) =\int_{\mathbb{S}^{d-1}} g(\theta)\mathrm{e}^{\mathrm{i}kx\cdot\theta}\, \text{d}s(\theta), \quad x\in\mathbb{R}^d.
\end{equation}
%Let $\mathcal{B}$ signify the extension by continuity of the mapping $u^i(\cdot,\theta,k)|_{\partial \Omega}\to u^{\infty}(\cdot,\theta,k)$ from $H^{1/2}(\partial \Omega)$ into $L^2(\mathbb{S}^{d-1})$. Then, one has $F_k=\mathcal{B}H_k$. Let $\mathcal{B}^{\delta}$ denote the noisy bounded linear operator associated with $\mathcal{B}$. Due to the well-posedness of the forward scattering problem, we can assume that $F_k^{\delta}=\mathcal{B}^{\delta}H_k$ and $\|\mathcal{B}^{\delta}-\mathcal{B}\|\leq \delta$.
We have the following result:

\begin{lem}[ \cite{Arens03}]\label{le:noteigen}
If $k$ is not a transmission eigenvalue, then there exists an approximate solution $g_{\epsilon}(\cdot,z)\in L^2(\mathbb{S}^{d-1})$ of the far-field equation \eqref{eq:ffeqn} such that $H_k g_{\epsilon}(\cdot,z)$ converges in the $H^1(\Omega)$ norm as $\epsilon \to 0$ when $z\in \Omega$.
\end{lem}

However, if $k$ is a transmission eigenvalue, the statement in Lemma \ref{le:noteigen} is not true. In this case, one can show that for all points $z\in \Omega$, there is
\begin{equation}\label{eq:normlim}
\lim\limits_{\epsilon \to 0}\|F_k^{\epsilon}g_{\epsilon}(\cdot,z) -\Psi^{\infty}(\cdot,z,k)\|_{L^2(\mathbb{S}^{d-1})}=0.
\end{equation}
Moreover, the following lemma characterizes the solution to \eqref{eq:normlim} when $k$ is a transmission eigenvalue.

\begin{lem}\label{le:iseigen}\cite{Cakoni10b}
Suppose that $k$ is a transmission eigenvalue and assume that \eqref{eq:normlim} holds. Then for almost every $z\in \Omega$, $\|H_k g_{\epsilon}(\cdot,z)  \|_{L^2 (\Omega)}$ can not be bounded as $\epsilon \to 0$.
\end{lem}
Since $H_k$ is compact, there exists a constant $M> 0$ such that
\[
\|H_k g_{\epsilon}(\cdot,z)  \|_{L^2 (\Omega)} \\ \leq M\|g_{\epsilon}(\cdot,z) \|_{L^2(\mathbb{S}^{d-1})} .
\]
Thus, $\|g_{\epsilon}(\cdot,z) \|_{L^2(\mathbb{S}^{d-1})}$ can not be bounded as $\epsilon \to 0$ if $k$ is a transmission eigenvalue.

By the above two lemmas, we note that $\|g_{\epsilon}(\cdot,z) \|_{L^2(\mathbb{S}^{d-1})} $ behaves quite differently when $k^2$ is a transmission eigenvalue or not. Hence, one can use $\|g_{\epsilon}(\cdot,z) \|_{L^2(\mathbb{S}^{d-1})}$ as an indicator to identify if $k$ is a transmission eigenvalue or not. We formulate the following scheme, dubbed as Algorithm \uppercase\expandafter{\romannumeral1}, to determine the transmission eigenvalues.

\begin{table}[htbp]
\begin{tabular}{p{1.5cm} p{11cm}}
\toprule
\multicolumn{2}{l}{\textbf{Algorithm \uppercase\expandafter{\romannumeral1}:}\ Determination of transmission eigenvalues}\\
\midrule
Step 1 & Collect a family of far-field data $u^{\infty,\epsilon}(\hat{x},\theta,k)$ for $(\hat{x},\theta,k)\in \mathbb{S}^{d-1}\times \mathbb{S}^{d-1} \times I$, where $I$ is an open interval in $\mathbb{R}_{+}$.\\
Step 2 & Pick a point $z \in \Omega$ (a-priori information) and for each $k\in I$, solve \eqref{eq:ffeqn} to obtain the solution $g_{\epsilon}(\cdot,z)$.\\
Step 3 & Plot $\|g_{\epsilon}(\cdot,z) \|_{L^2(\mathbb{S}^{d-1})}$ against $k\in I$ and find the transmission eigenvalues where peaks appear in the graph.\\
\bottomrule
\end{tabular}
\end{table}

We note that $\Omega$ is unknown, so we consider $\|g_{\epsilon}(\cdot,z) \|_{L^2(\mathbb{S}^{d-1})}$, instead of $\|H_k g_{\epsilon}(\cdot,z)  \|_{L^2 (\Omega)}$ in Step 3 of Algorithm \uppercase\expandafter{\romannumeral1} though $\|H_k g_{\epsilon}(\cdot,z)  \|_{L^2 (\Omega)}$ also behaves differently when $k$ is a transmission eigenvalue or not.

%\begin{rem}\label{rk:eigenvalue}
%Due to the fact that when $k$ is not a transmission eigenvalue, there exists an approximate solution $g_{\delta}(\cdot,z) \in L^2(\mathbb{S}^{d-1})$ of the far field equation \eqref{eq:ffeqn}, we says that $\| g_{\delta}(\cdot,z) \|_{L^2(\mathbb{S}^{d-1})} \leq M$, where $M$ is a constant. On the other hand, if $k$ is a transmission eigenvalue, Lemma \ref{le:iseigen} implies that $\|g_{\delta}(\cdot,z) \|_{L^2(\mathbb{S}^{d-1})}$ can not be bounded. These results give the theory ground for Algorithm \uppercase\expandafter{\romannumeral1}.
%\end{rem}

\subsection{Phase $\mathrm{II}$: determination of transmission eigenfunctions}

In Phase~$\mathrm{I}$, we determine the transmission eigenvalues within the interval $I$ by knowledge of the far-field data in \eqref{eq:ip1}. We proceed to determine the corresponding transmission eigenfunctions. To that end, we first recall the Herglotz wave introduced in \eqref{eq:Herglotz}
where $g\in L^2(\mathbb{S}^{d-1})$ is referred to as the Herglotz kernel of $v_{g,k}$.

\begin{lem}\label{le:HergDense}\cite{Weck04}
Let $\Omega$ be a bounded domain of class $C^{\alpha, 1}$, $\alpha\in\mathbb{N}\cup\{0\}$, in $\mathbb{R}^d$. Denote by $\mathbb{H}$ the space of all Herglotz wave functions of the form \eqref{eq:Herglotz}. Define, respectively,
\[
\mathbb{H}(\Omega):=\{ u|_{\Omega}:u\in\mathbb{H} \},
\]
and
\[
\mathfrak{H}(\Omega):=\{u\in C^{\infty}(\Omega): \Delta u +k^2 u=0 \text{ in } \Omega\}.
\]
Then $\mathbb{H}(\Omega)$ is dense in $\mathfrak{H}(\Omega)\cap H^{\alpha+1}(\Omega)$ with respect to the $H^{\alpha+1}(\Omega)$-norm.
\end{lem}

The following theorem states that if $k\in\mathbb{R}_+$ is a transmission eigenvalue, then there exists a Herglotz wave function $v_{g_{\epsilon},k}$ such that the scattered field corresponding to this $v_{g_{\epsilon},k}$ as the incident field is nearly vanishing.

\begin{thm}\label{th:Eigenf}
Suppose that $k\in\mathbb{R}_+$ is a transmission eigenvalue in $\Omega$. For any sufficiently small $\epsilon \in \mathbb{R}_{+}$, there exists $g_{\epsilon}\in L^2(\mathbb{S}^{d-1})$ such that
\[
\|F_k g_{\epsilon} \|_{L^2(\mathbb{S}^{d-1})}=\mathcal{O}(\epsilon) \ \ \text{and} \ \   \|v_{g_{\epsilon},k} \|_{L^2(\Omega)}=\mathcal{O}(1),
\]
where $F_k$ is the far field operator defined by \eqref{eq:ffope} and $v_{g_{\epsilon},k}$ is the Herglotz wave function defined by \eqref{eq:Herglotz} with the kernel $g_{\epsilon}$.
\end{thm}

\begin{proof}
Let $v_k$ be a normalized transmission eigenfunction in $\Omega$ associated to the transmission eigenvalue $k^2$, which means that $v_k$ with $\|v_k\|_{L^2(\Omega)}=1$ is a solution of
\[
\Delta v_k+k^2v_k=0 \quad \text{in } \Omega.
\]
By Lemma \ref{le:HergDense}, for any sufficiently small $\epsilon >0$, there exists $g_{\epsilon}\in L^2(\mathbb{S}^{d-1})$ such that
\begin{equation*}%\label{eq:ApproInProof}
\|v_{g_{\epsilon},k} -v_k \|_{L^2(\Omega)}< \epsilon,
\end{equation*}
where $v_{g_{\epsilon},k}$ is the Herglotz wave function with the kernel $g_{\epsilon}$. Then, by the triangle inequality,
\[
\|v_{g_{\epsilon},k} \|_{L^2(\Omega)}\leq \|v_{g_{\epsilon},k} -v_k \|_{L^2(\Omega)}+\| v_k \|_{L^2(\Omega)}< \epsilon +\| v_k \|_{L^2(\Omega)},
\]
and
\[
\|v_{g_{\epsilon},k} \|_{L^2(\Omega)}\geq \| v_k \|_{L^2(\Omega)}-\|v_{g_{\epsilon},k} -v_k \|_{L^2(\Omega)} > \| v_k \|_{L^2(\Omega)}- \epsilon.
\]
Thus, one must have that $\|v_{g_{\epsilon},k} \|_{L^2(\Omega)} =\mathcal{O}(1)$.

Furthermore, from the definition of the far-field operator, $F_k g_{\epsilon}$ is the far-field pattern produced by the incident wave $v_{g_{\epsilon},k}$. According to Proposition 4.2 in \cite{BL1}, one has that
\[
\|F_k g_{\epsilon} \|_{L^2(\mathbb{S}^{d-1})}<C\epsilon,
\]
where $C=C(\mathbf{n},k)$ is a positive constant.

The proof is complete.
\end{proof}

By Theorem \ref{th:Eigenf} and normalization if necessary, we can say that the following optimization problem:
\begin{equation}\label{eq:opp1}
\min\limits_{g\in L^2(\mathbb{S}^{d-1})} \|F_k g \|_{L^2(\mathbb{S}^{d-1})} \quad \text{    s.t.  } \|v_{g,k} \|_{L^2(\Omega)}=1
\end{equation}
has at least one (approximate) solution $g_0 \in L^2(\mathbb{S}^{d-1})$ when $k\in\mathbb{R}_+$ is a transmission eigenvalue in $\Omega$. However, since $\Omega$ is unknown, the constraint $\|v_{g,k} \|_{L^2(\Omega)}=1$ in the optimization formulation \eqref{eq:opp1} is unpractical. Nevertheless, it is reasonable to address this issue by considering an alternative optimization problem:
\begin{equation}\label{eq:AlterOptimization}
\min\limits_{g\in L^2(\mathbb{S}^{d-1})} \|F_k g \|_{L^2(\mathbb{S}^{d-1})} \quad \text{    s.t.  } \|v_{g,k} \|_{L^2(B_R)}=1,
\end{equation}
where $B_R$ is an a-priori ball containing $\Omega$.

Let $g_0$ be a ``reasonable" solution to the optimization problem \eqref{eq:AlterOptimization}. Next, we show that the corresponding Herglotz wave $v_{g_0,k}$ is generically indeed an approximation to the transmission eigenfunction $v_k$ associated to the transmission eigenvalue $k$.

\begin{thm}\label{th:IndeedAppro}
Suppose $k\in\mathbb{R}_+$ is a transmission eigenvalue in $\Omega$ and $g_0$ is a solution to the optimization problem \eqref{eq:AlterOptimization} satisfying
\begin{equation}\label{eq:ssn1}
\|F_k g_0\|_{L^2(\mathbb{S}^{d-1})}\leq \epsilon\ll 1.
\end{equation}
If we further assume that $\Omega$ is of class $C^{1,1}$ and $|\mathbf{n}^2-1|\geq \delta_0$ for a certain $\delta_0\in\mathbb{R}_+$ in a neighbourhood of $\partial\Omega$,
 then the Herglotz wave $v_{g_0,k}$ is an approximation to a transmission eigenfunction $v_k$ associated with the transmission eigenvalue $k$ in the $H^2(\Omega)$-norm.
\end{thm}

\begin{proof}
Consider the scattering system \eqref{eq:forwardpro}. We let $u^i=v_{g_0, k}$, $u^s_{g_0, k}$ and $u$ be respectively the incident, scattered and total wave fields. It is clear that one has
\begin{equation}\label{eq:mediuSys}
\left\{
\begin{array}{ll}
\Delta u+k^2\mathbf{n}^2(x)u=0 \quad &\text{in} \ \Omega, \medskip \\
\Delta v_{g_0,k}+k^2 v_{g_0,k}=0\quad &\text{in} \ \Omega, \medskip \\
u=v_{g_0,k}+u^s_{g_0,k} \quad &\text{on} \ \partial \Omega, \medskip \\
\displaystyle \frac{\partial u}{\partial \nu} =\frac{\partial v_{g_0,k}}{\partial \nu}+\frac{\partial u^s_{g_0,k}}{\partial \nu} \quad &\text{on} \ \partial \Omega.
\end{array}
\right.
\end{equation}

According to our earlier discussion, $F_kg_0$ is the far-field pattern of $u_{g_0, k}^s$. By virtue of \eqref{eq:ssn1} as well as the quantitative Rellich theorem established in \cite{Blasten16}, one has
\begin{equation}\label{eq:ssn2}
\| u^s_{g_0,k} \|_{H^{3/2}(\partial \Omega)}+\left\| \frac{\partial u^s_{g_0,k}}{\partial \nu} \right\|_{H^{1/2}(\partial \Omega)}\leq \psi (\epsilon),
\end{equation}
where $\psi$ is the stability function in \cite{Blasten16}, which is of double logarithmic type and satisfies $\psi (\epsilon)\to 0$ as $\epsilon\to +0$.

Consider the transmission eigenvalue problem \eqref{eq:trans1}. Setting $W=w-v$, $V=k^2 v$ and $\lambda = -k^2$, \eqref{eq:trans1} can be rewritten as (cf. \cite{Luc13, Sylvester12}):
\begin{equation*}%\label{eq:coupletrans}
\left\{
\begin{array}{ll}
(\Delta -\lambda \mathbf{n}^2(x))W +(\mathbf{n}^2(x)-1)V=0 \quad &\text{in} \ \Omega, \medskip \\
(\Delta -\lambda)V=0 \quad &\text{in} \ \Omega, \medskip \\
W=0 \quad &\text{on} \ \partial \Omega, \medskip \\
\displaystyle \frac{\partial W}{\partial \nu}=0 \quad &\text{on} \ \partial \Omega.
\end{array}
\right.
\end{equation*}
Let $\Delta_{00}$ denote the Laplacian with domain $H_0^2(\Omega)$ and $\Delta_{--}$ denote the Laplacian with domain $H^2(\Omega)$. By \cite{Luc13, Sylvester12}, the squares of interior transmission eigenvalues are the spectrum of the generalized eigenvalue problem
\begin{equation}\label{eq:EigenPro}
(A-\lambda I_{\mathbf{n}})\begin{pmatrix} W\\ V \end{pmatrix} :=\left(
\begin{matrix}
\Delta_{00}  & \mathbf{n}^2-1 \\
0 &  \Delta_{--}
\end{matrix}
   \right)\begin{pmatrix} W\\ V \end{pmatrix}-\lambda
\left(
\begin{matrix}
\mathbf{n}^2 & 0 \\
0 & 1
\end{matrix}
   \right)\begin{pmatrix} W\\ V \end{pmatrix}=0,
\end{equation}
where $(W, V)\in H_0^2(\Omega)\oplus H^2(\Omega)$.

Now, we consider the PDE system \eqref{eq:mediuSys}. By the standard Sobolev extension as well as noting \eqref{eq:ssn2}, we let $\zeta\in H^2(\Omega)$ be such that
\begin{equation}\label{eq:ext1}
\zeta=u_{g_0, k}^s,\quad \frac{\partial\zeta}{\partial\nu}=\frac{\partial u_{g_0, k}^s}{\partial\nu}\ \ \mbox{on}\ \ \partial\Omega,
\end{equation}
and
\begin{equation}\label{eq:ext2}
\|\zeta\|_{H^2(\Omega)}\leq C\psi(\epsilon),
\end{equation}
where $C$ is a generic constant depending on $\Omega$. Introducing
\begin{equation}\label{eq:auxn1}
\Upsilon_1=-(\Delta+k^2\mathbf{n}^2)\zeta,\quad \Upsilon_2=0,\quad \Upsilon=(\Upsilon_1, \Upsilon_2)^T\in L^2(\Omega)\oplus L^2(\Omega),
\end{equation}
and setting $\widetilde u:=u-\zeta\in H^2(\Omega)$, \eqref{eq:mediuSys} can be rewritten as
\begin{equation}\label{eq:auxn2}
\left\{
\begin{array}{ll}
\Delta \widetilde{u}+k^2\mathbf{n}^2(x) \widetilde{u}=\Upsilon_1 \quad &\text{in} \ \Omega, \medskip \\
\Delta v_{g_0,k}+k^2 v_{g_0,k}=\Upsilon_2\quad &\text{in} \ \Omega, \medskip \\
\widetilde{u}=v_{g_0,k} \quad &\text{on} \ \partial \Omega, \medskip \\
\displaystyle \frac{\partial \widetilde{u}}{\partial \nu} =\frac{\partial v_{g_0,k}}{\partial \nu} \quad &\text{on} \ \partial \Omega.
\end{array}
\right.
\end{equation}
Setting $W_{g_0,k}=\widetilde{u}-v_{g_0,k}$ and $V_{g_0,k}=k^2 v_{g_0,k}$, we can rewrite the system \eqref{eq:auxn2} into the operator form
\begin{equation}\label{eq:on1}
(A-\lambda I_{\mathbf{n}})\begin{pmatrix} W_{g_0, k}\\ V_{g_0, k}  \end{pmatrix}=\Upsilon, \ \ W_{g_0, k}\in H_0^2(\Omega), \ V_{g_0, k}\in H^2(\Omega).
\end{equation}

Note that $\lambda=-k^2$ is an eigenvalue to the operator equation \eqref{eq:EigenPro}. It is shown in \cite{Sylvester12} that $I_{\mathbf{n}}^{-1} A$ possesses a UTC (upper triangular compact) resolvent. This enables one to apply the upper triangular analytic Fredholm theorem to the eigenvalue problem \eqref{eq:EigenPro} as well as the operator equation \eqref{eq:on1}, which enjoys the same properties as those within the analytic Fredholm theorem (in the current setup of our study). Next, we shall prove that the operator equation \eqref{eq:on1} solvable in the
quotient space $(H_0^2(\Omega)\oplus H^2(\Omega))/(\mathbb{W}\oplus\mathbb{V})$, where $\mathbb{W}\oplus\mathbb{V}$ is the finite-dimensional eigen-space to \eqref{eq:EigenPro}. In order to apply the Fredholm theorem, it is sufficient for us to show that $\Upsilon\in \mathrm{Ker}[(A-\lambda I_{\mathbf{n}})^*]^\perp$. The kernel of $(A-\lambda I_{\mathbf{n}})^*$ consists of functions $(\xi, \eta)\in H^2(\Omega)\oplus H_0^2(\Omega)$ satisfying
\begin{equation}\label{eq:an1}
\begin{cases}
(\Delta-\lambda) \eta+(\mathbf{n}^2-1) \xi = 0\ \ &\mbox{in}\ \ \Omega,\medskip\\
(\Delta-\lambda) \xi = \lambda(\mathbf{n}^2-1)\xi\ \ &\mbox{in}\ \ \Omega,
\end{cases}
\end{equation}
which, by introducing $\widetilde{\eta}=\xi+\lambda \eta\in H^2(\Omega)$, is equivalent to the following PDE system
\begin{equation}\label{eq:an2}
\begin{cases}
(\Delta-\lambda)\widetilde\eta=0\ \ &\mbox{in}\ \ \Omega,\medskip\\
(\Delta-\lambda)\xi=\lambda(\mathbf{n}^2-1)\xi\ \ &\mbox{in}\ \ \Omega,\medskip\\
\widetilde\eta=\xi,\ \ \partial_\nu\widetilde\eta=\partial_\nu \xi\ \ &\mbox{on}\ \ \partial\Omega.
\end{cases}
\end{equation}
With the above fact and using \eqref{eq:auxn1}, we have
\begin{align}
\Upsilon\cdot (\xi, \eta)^T=&\int_\Omega\Upsilon_1\cdot \xi \nonumber\\
=& \int_\Omega -(\Delta+k^2\mathbf{n}^2) \zeta\cdot \xi=\int_{\partial\Omega} \partial_\nu\xi\cdot \zeta-\xi\cdot\partial_\nu\zeta\label{eq:ar1}\\
=& \int_{\partial\Omega} \partial_\nu\xi\cdot u_{g_0, k}^s-\xi\cdot\partial_\nu u_{g_0, k}^s,\label{eq:ar2}
\end{align}
where in \eqref{eq:ar1} we have made use of the fact $(\Delta+k^2\mathbf{n}^2)\xi=0$ from \eqref{eq:an1}; and in \eqref{eq:ar2} we have made use of the fact in \eqref{eq:ext1}.
From \eqref{eq:mediuSys}, we see that
\begin{equation*}%\label{eq:ar3}
u_{g_0,k}^s=u-v_{g_0, k},\ \partial_\nu u_{g_0, k}^s=\partial_\nu u-\partial_\nu v_{g_0, k} \quad \mbox{on } \partial\Omega,
\end{equation*}
which readily yields that
\begin{align}
& \int_{\partial\Omega} \partial_\nu\xi\cdot u_{g_0, k}^s-\xi\cdot\partial_\nu u_{g_0, k}^s\nonumber\\
=&\bigg( \int_{\partial\Omega} \partial_\nu\xi\cdot u-\xi\cdot\partial_\nu u\bigg)-\bigg(\int_{\partial\Omega} \partial_\nu\xi\cdot v_{g_0, k}-\xi\cdot\partial_\nu v_{g_0, k}\bigg)\label{eq:ar4}\\
=& \int_{\partial\Omega} \widetilde\eta\cdot\partial_\nu v_{g_0, k}-\partial_\nu\widetilde\eta\cdot v_{g_0, k}=0\label{eq:ar5},
\end{align}
where from \eqref{eq:ar4} to \eqref{eq:ar5}, we have made use of the transmission conditions on $\partial\Omega$ from \eqref{eq:an2}. This implies that
\begin{equation*}%\label{eq:fn1}
\Upsilon\in \mathrm{Ker}[(A-\lambda I_{\mathbf{n}})^*]^\perp.
\end{equation*}
Hence, \eqref{eq:on1} is solvable in $(H_0^2(\Omega)\oplus H^2(\Omega))/(\mathbb{W}\oplus\mathbb{V})$.

Set
\begin{equation}\label{eq:fn2}
\begin{pmatrix} W^*_{g_0, k}\\ V^*_{g_0, k}  \end{pmatrix}=(A-\lambda I_{\mathbf{n}})^{-1} \Upsilon\quad \mbox{in}\ \ (H_0^2(\Omega)\oplus H^2(\Omega))/(\mathbb{W}\oplus\mathbb{V}).
\end{equation}
By \eqref{eq:ext2}, we obviously have
\begin{equation}\label{eq:ext3}
\|\Upsilon\|_{L^2(\Omega)^2}\leq C\psi(\epsilon),
\end{equation}
with $C$ a generic constant depending on $\mathbf{n}, k$ and $\Omega$. Finally, by combining \eqref{eq:fn2} and \eqref{eq:ext3}, one can show that
\begin{equation*}%\label{eq:est10}
\|V_{g_0, k}-V_k\|_{H^2(\Omega)}\leq C\psi(\epsilon)\rightarrow 0\ \ \mbox{as}\ \ \epsilon\rightarrow+0\ \ \mbox{for}\ \ V_k\in\mathbb{V},
\end{equation*}
which readily implies that $v_{g_0, k}$ is an approximation to a transmission eigenfunction $v_k$.

The proof is complete.

\end{proof}

\begin{rem}
It is remarked that in Theorem~\ref{th:IndeedAppro}, the $C^{1,1}$-regularity on $\partial\Omega$ is a technical condition. It is required in \eqref{eq:ssn2} and \eqref{eq:ext1}, and in particular according to \cite{Blasten16}, higher regularity might be required in deriving \eqref{eq:ssn2}, which we choose not to explore further since it is not the focus of this article. Nevertheless, we believe that this regularity assumption on $\partial\Omega$ can be relaxed. In fact, according to our numerical examples in what follows, even if $\partial\Omega$ is only Lipschitz continuous, one can still determine the (approximte) transmission eigenfunctions by solving the optimization problem \eqref{eq:AlterOptimization}.
\end{rem}

\begin{rem}
In the whole paper, we assume that $\mathbf{n}$ is real and $|\mathbf{n}^2-1|\geq \delta_0$. This assumption is mainly required for establishing the surface-localisation of the transmission eigenfunctions. In fact, most of the results presented in this section, say Theorems~\ref{th:Eigenf} and \ref{th:IndeedAppro}, can be extended to the more general case where $\mathbf{n}^2$ satisfies the more general assumptions in \cite{Sylvester12}.
\end{rem}

\subsection{Phase~$\mathrm{III}$: imaging of the scatterer}

In Phases~$\mathrm{I}$ and $\mathrm{II}$, using the far-field data in \eqref{eq:ip1}, we respectively determine the transmission eigenvalues within the interval $I$ and the corresponding transmission eigenfunctions. In this part, we shall show that the transmission eigenfunctions can be used for the qualitative imaging of the shape of the medium scatterer $(\Omega, \mathbf{n}^2)$, namely $\partial\Omega$ independent of $\mathbf{n}^2$. The basic idea can be described as follows. Let $v_{g_0,k}$ be determined from Phase~$\mathrm{II}$ which approximates a transmission eigenfunction within $\Omega$. Then according our study in Section~\ref{sect:2}, $v_k$ can be an SLE if the conditions in Section~\ref{sect:2} are fulfilled. In fact, we know from the numerical examples in Sections~\ref{sect:2.3} and \ref{sect:2.6}, that if $\mathbf{n}$ is sufficiently large in a neighbourhood of $\partial\Omega$, even for a relatively small transmission eigenvalues, the corresponding transmission eigenfunction is an SLE. Furthermore, in such a case, the SLEs occur very often in our extensive numerical experiments and a major part of the calculated transmission eigenfunctions are SLEs. This is a highly interesting phenomenon that is worth further theoretical investigation. Based on such an observation, it naturally leads to the following imaging functional for recovering $\Omega$:
\begin{equation*}%\label{eq:soloIndicator}
I_k^{\text{Res}}(z):= -\text{ln}| v_{g_0,k}(z) |.
\end{equation*}
If the underlying transmission eigenfunction $v_k$ is an SLE, then $v_{g_0, k}$ is referred to as an approximate SLE. It can be easily seen that if $v_{g_0,k}$ an approximate SLE, then $I_k^{\text{Res}}(z)$ possesses a relative large value if $z$ belongs to the interior of $\Omega$, or $z$ is located at the corner/edge/highly-curved place on $\partial \Omega$, whereas it possesses a relatively small value if $z$ is located in the other places around $\partial\Omega$. Since in Phase $\mathrm{II}$, multiple transmission eigenfunctions can be determined, we can superpose the imaging effects by introducing the following imaging functional:
\begin{equation}\label{eq:multiIndicator}
I_{\mathbb{K}_L}^{\text{Res}} (z):= -\text{ln} \sum\limits_{k\in{\mathbb{K}_L}}| v_{g_0,k}(z) |,
\end{equation}
where $\mathbb{K}_L=\{k_1,k_2,\cdots,k_L \}$ denotes the set of $L$ distinct transmission eigenvalues determined in Phase~$\mathrm{I}$. Based on the imaging functional \eqref{eq:multiIndicator}, we then propose the following imaging scheme, which is referred to as imaging by interior resonant modes. Indeed, the proposed scheme is based on using the interior transmission eigen-modes, which are usually ``discarded" or ``avoided" in many existing inverse scattering schemes, say e.g. the linear sampling method \cite{Colton00} and the factorization method \cite{KirGrin}.

\begin{table}[htbp]
\begin{tabular}{p{1.5cm} p{11cm}}
\toprule
\multicolumn{2}{l}{\textbf{Algorithm \uppercase\expandafter{\romannumeral2}:}\ Imaging by interior resonant modes}\\
\midrule
Step 1 & For each resonant wavenumber $k$ found in Algorithm I, solve the optimization problem \eqref{eq:AlterOptimization} by the FTLS method or the GTLS method \cite{LLWW} to obtain the Herglotz kernel $g_0$.\\
Step 2 & Calculate the Herglotz wave function $v_{g_0,k}$ with the Herglotz kernel $g_0$ by the definition \eqref{eq:Herglotz}.\\
Step 3 & Plot the indicator function \eqref{eq:multiIndicator} in a proper domain containing the scatterer $\Omega$ and identify the interior and corners (two dimension) or edges (three dimension) as bright points, and other boundary places as dark points in the graph to obtain the shape of the scatterer $\Omega$.\\
\bottomrule
\end{tabular}
\end{table}

Finally, we would like to emphasize that according to our discussion above, the proposed imaging scheme should work for imaging a medium scatterer whose refractive index is highly-contrast in a neighbourhood of its boundary. This clearly includes the case that the medium scatterer itself already possesses a high-contrast refractive index. On the other hand, for a regular refractive-indexed scatterer, one would need to first coat the scatterer by indirect means with a thin layer of highly refractive-indexed material, then our imaging scheme would work as well. Hence, it is unobjectionable to claim that the proposed method possesses a practical value for generic inverse scattering imaging. In what follows, we present several numerical examples to illustrate the effectiveness of the proposed imaging scheme. In order to simplify the situation, we only consider imaging a medium scatterer possessing a constant refractive index of a relatively high magnitude in two dimensions. It is remarked that the higher the refractive index is, the better imaging effect one can expect to achieve. Moreover, as emphasized above, this highly refractive-indexed medium can be located only in a neighbourhood of $\partial\Omega$.

%\begin{rem}\label{rk:eigenfunc}
%In Step 3 of Algorithm \uppercase\expandafter{\romannumeral2}, we can also propose the indicator function by using the multiple resonant modes:
%\begin{equation}\label{eq:multiIndicator}
%I_{\mathbb{K}_L}^{\text{Res}} (z):= -\text{ln} \sum\limits_{k\in{\mathbb{K}_L}}| v_{g_0,k}(z) |,
%\end{equation}
%where $\mathbb{K}_L=\{k_1,k_2,\cdots,k_L \}$ denotes a set of the $L$ distinct eigenvalues. Due to the common feature of those Herglotz wave functions is that they vanish in the interior of $\Omega$ and near corners/edges of $\partial \Omega$, it is reasonable to expect that the indicator function $I^{\text{Res}}_{\mathbb{K}_L}(z)$ has larger values for these points than other points. This behaviour has been well showed in our later numerical examples.
%\end{rem}

\subsection{Numerical examples}
In this section, we present the numerical experiments as mentioned above. To avoid the inverse crime, we use the finite element method to compute the scattering amplitudes $u^\infty(\hat{x}_i, d_j),\, i=1,2,\cdots,M_0, \, j=1,2,\cdots,N_0$, where $\hat{x}_i$, $i=1, 2, \ldots, M_0$ denote the discrete observation directions and $d_j$ $j=1,2,\ldots, N_0$ denote the discrete incident directions. In the two dimensions, the observation and incident directions are equidistantly distributed on a unit circle. Then we consider the collected far-field matrix $F\in \mathbb{C}^{M_0 \times N_0}$ such that
\begin{equation*}
  F=\left[
      \begin{array}{cccc}
        u^{\infty}(\hat x_1, d_1) & u^{\infty}(\hat x_1, d_2) & \cdots& u^{\infty}(\hat x_1, d_{N_0}) \\
        u^{\infty}(\hat x_2, d_1) &u^{\infty}(\hat x_2, d_2)  & \cdots & u^{\infty}(\hat x_2, d_{N_0})  \\
        \vdots & \vdots & \ddots & \vdots \\
        u^{\infty}(\hat x_{M_0}, d_1)  & u^{\infty}(\hat x_{M_0}, d_2)  & \cdots & u^{\infty}(\hat x_{M_0}, d_{N_0})  \\
      \end{array}
    \right].
\end{equation*}
In order to test the sensibility of the proposed method, we further perturb $F$ with random noise by setting
\begin{equation*}
F^{\delta}\ =\ F +\delta\|F\|\frac{R_1+R_2 \mathrm{i}}{\|R_1+R_2 \mathrm{i}\|},
\end{equation*}
where $\delta>0$ represents the noise level, $R_1$ and $R_2$ are two $M_0 \times N_0$ matrixes containing pseudo-random values drawn from a normal distribution with the mean being zero and the standard deviation being one.  In the following two examples,  the noise level is given by $\delta=1\%$.

\subsubsection{Square domain}

In the first example, we let $\Omega$ be a square with the side length being 2 and the refractive index being $\mathbf{n}=10$. The synthetic far-field data are computed at $100$ observation directions, $100$ incident directions and $3000$ wavenumbers within the interval $I=[0.6, 0.9]$, all equally distributed. Firstly, we use Algorithm \uppercase\expandafter{\romannumeral1} to determine four transmission eigenvalues, such as $k_1=0.6219, \, k_2=0.6896, \,k_3=0.7858$ and $k_4=0.8370$. Next, we can determine 4 approximate transmission eigenfunctions as well. Since $\mathbf{n}$ is relatively large, it turns out that the computed eigenfunctions are all approximate SLEs. We would like to emphasize again that we did not purposely design such a numerical example and indeed, as discussed earlier, the occurrence of the SLEs are very often. We present the reconstruction results in Fig. \ref{fig:super-square}, (a)--(c), by using 1, 2 and 4 SLEs, respectively. One readily sees that the square is already finely reconstructed with 4 SLEs. For comparisons, we also present the reconstruction results by using a sampling type method developed in several works \cite{N1,N2,LLWn1,LiuIP17} by using the multiple frequency scattering data in \eqref{eq:ip1}. The reconstruction results are presented in Fig. \ref{fig:super-square}, (d)--(f). It can be seen that the reconstructions basically yield a spot without any resolution of the square-shape object. In fact, one can also implement the other popular imaging schemes including the linear sampling method or the factorization method \cite{CK}, and the reconstruction effects shall remain almost the same. It is clear that the length of the square is much smaller than the underlying wavelength, $2\pi /k$. This result illustrates that super-resolution reconstruction can be realized by the proposed method. This is unobjectionably expected since we make use of the interior resonant modes for the reconstruction. Finally, it is pointed out if one further performs standard imaging processing to the reconstructed image in Fig~\ref{fig:super-square} (c), one should be able to obtain a nearly-accurate reconstruction of the square-object. However, this is not the focus of this article and we shall not explore further about this point.

\begin{figure}[h]
\centering
\subfigure[$L=1$]{\includegraphics[width=0.33\textwidth]{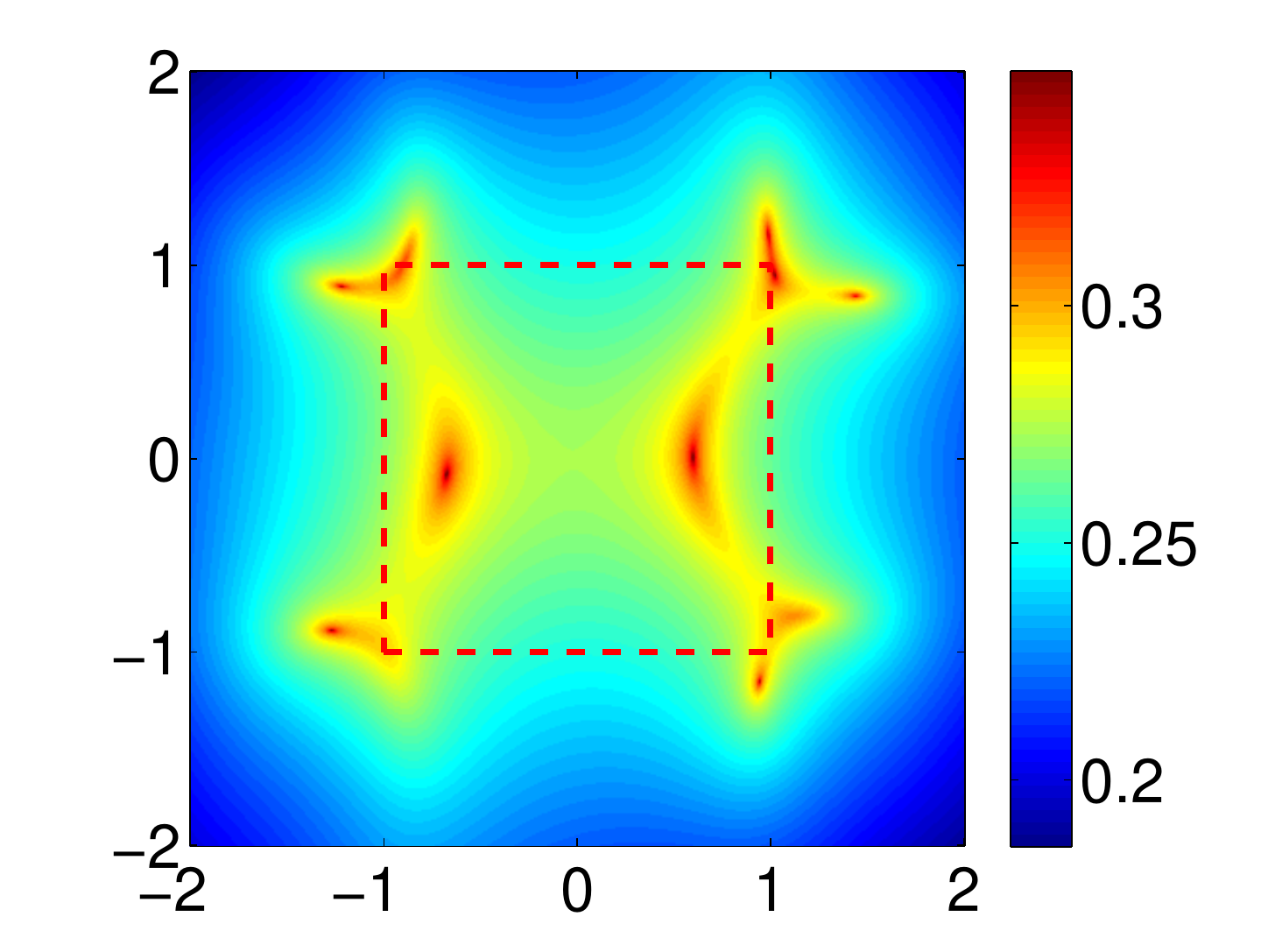}}% lleHere is how to import EPS art
\subfigure[$L=2$]{\includegraphics[width=0.33\textwidth]{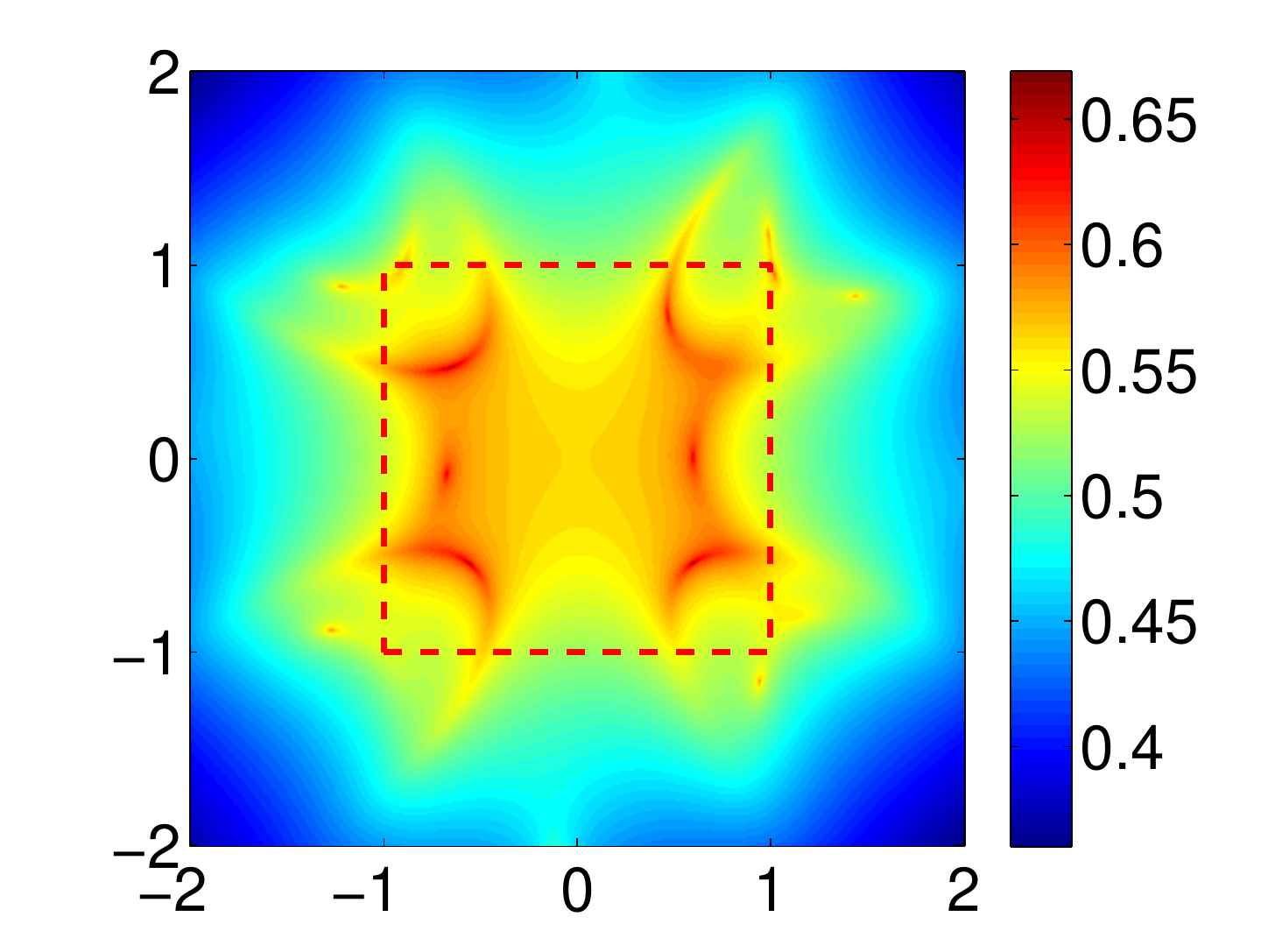}} % lleHere is how to import EPS art
%\subfigure[$N=3$]{\includegraphics[width=0.33\textwidth]{square_K3}}% lleHere is how to import EPS art
\subfigure[$L=4$]{\includegraphics[width=0.33\textwidth]{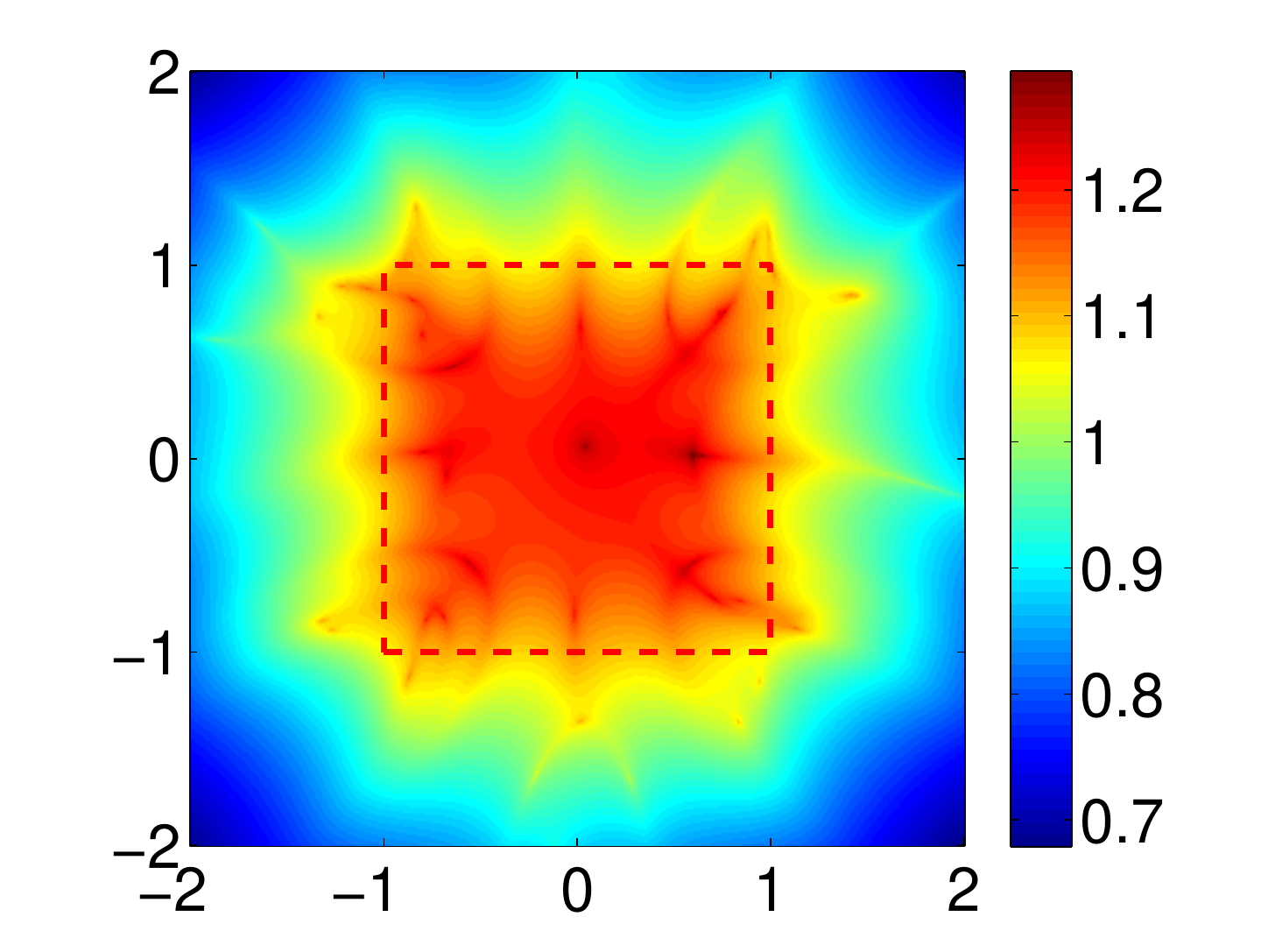}}\\ % lleHere is how to import EPS art
\subfigure[$L=1$]{\includegraphics[width=0.33\textwidth]{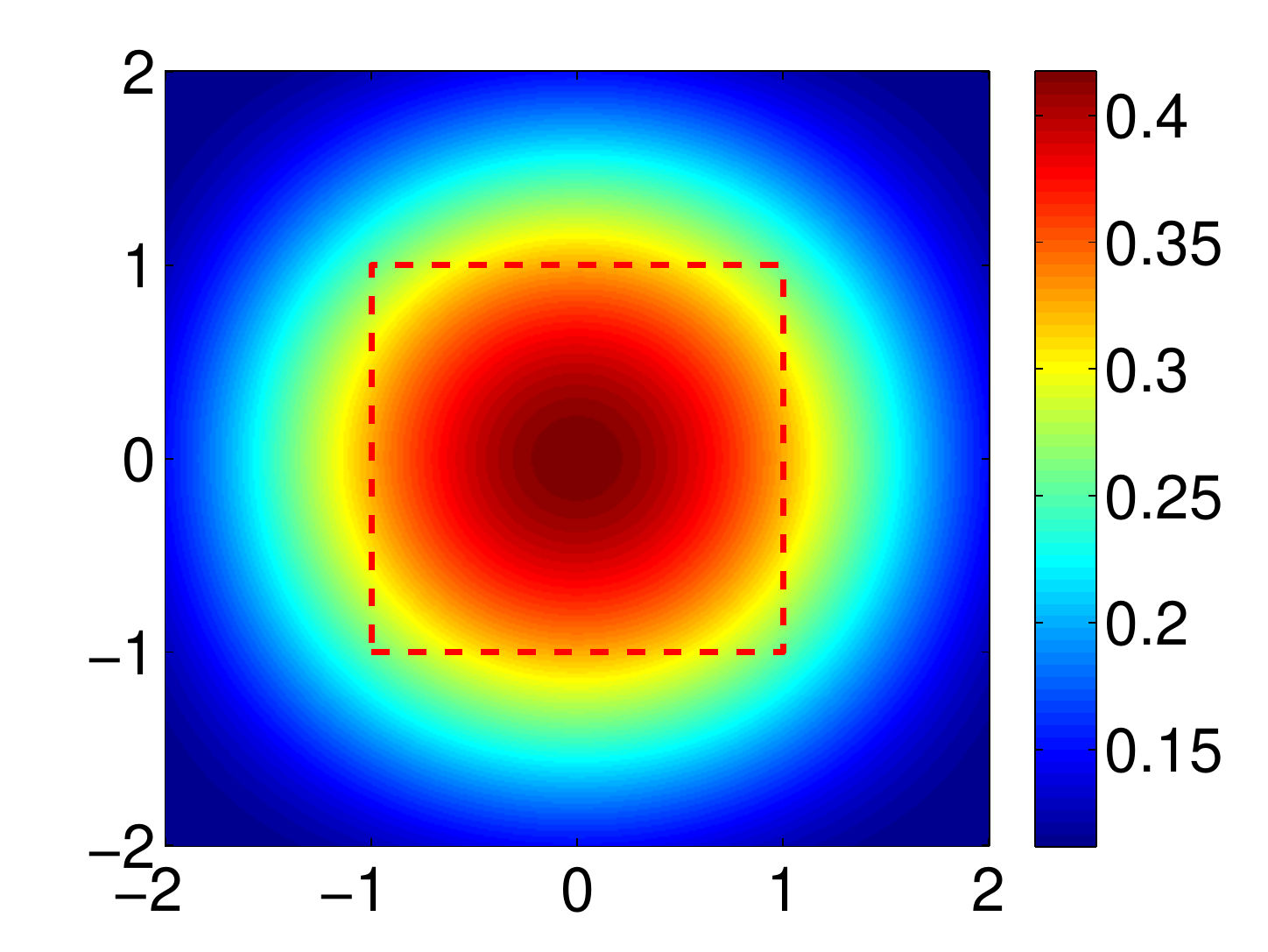}}% lleHere is how to import EPS art
\subfigure[$L=2$]{\includegraphics[width=0.33\textwidth]{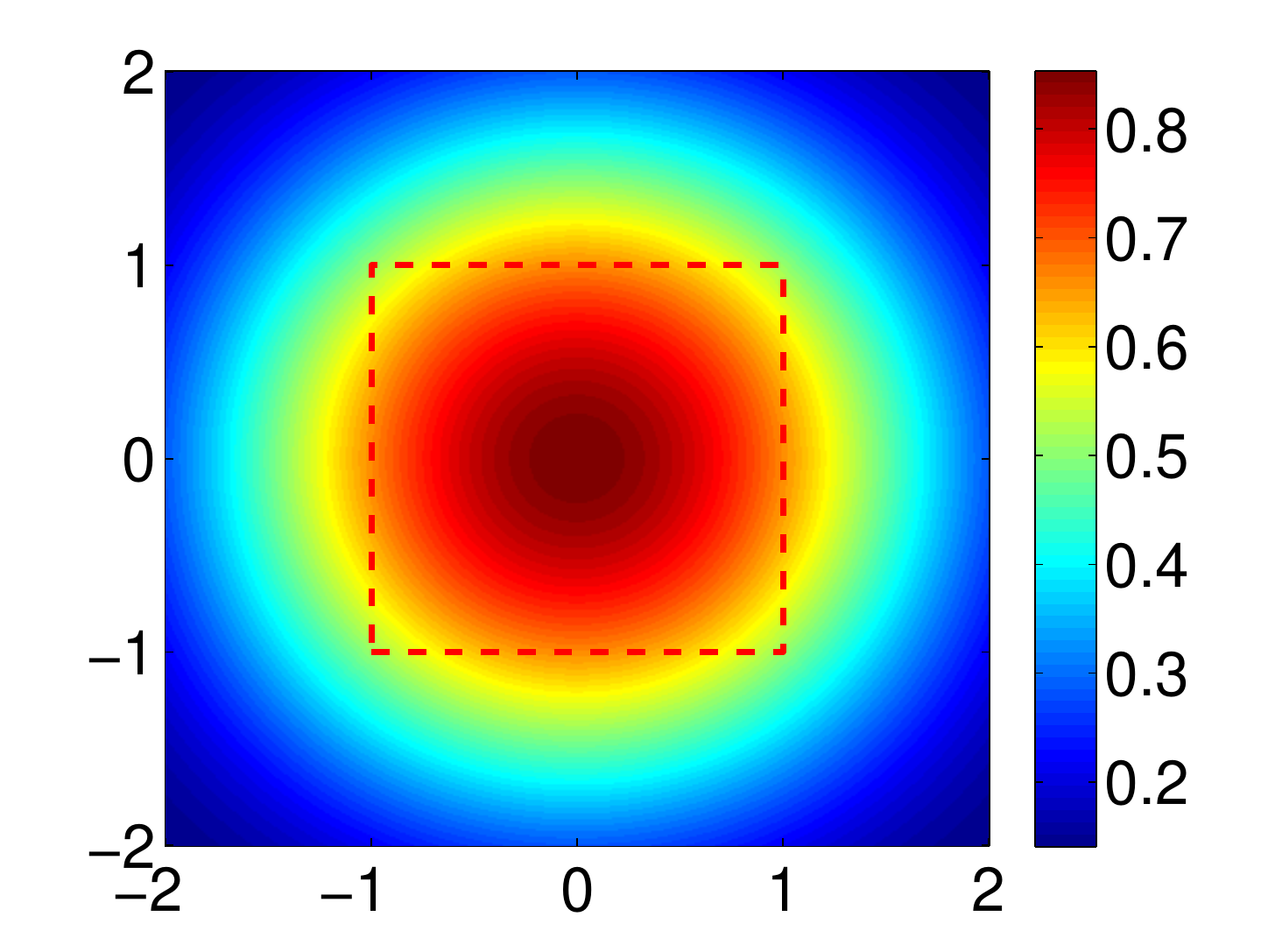}} % lleHere is how to import EPS art
%\subfigure[$N=3$]{\includegraphics[width=0.4\textwidth]{square_K3_direct}}% lleHere is how to import EPS art
\subfigure[$L=4$]{\includegraphics[width=0.33\textwidth]{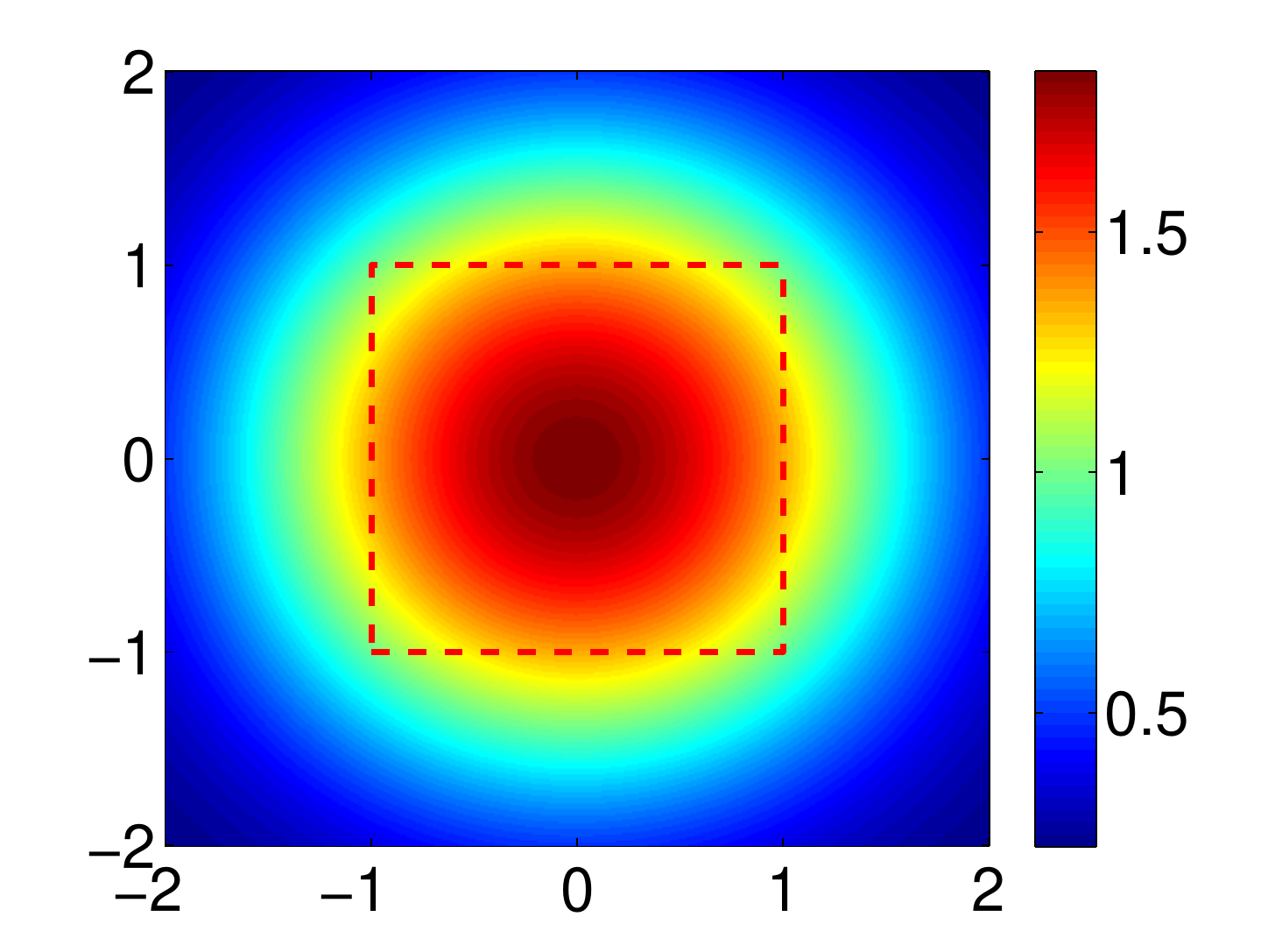}}\\ % lleHere is how to import EPS art
\caption{\label{fig:super-square} Reconstructions of a square-object by multiple SLEs (top row) and multiple-frequency direct sampling method (bottom row), respectively, where $\mathbf{n}=10$ for all cases. }
\end{figure}

\subsubsection{kite-shaped domain}
Fig.~\ref{fig:super-kite} presents another example where the target domain is a kite-shaped object with $\mathbf{n}=10$. The imaging frequencies are chosen within $[1, 3.1]$, and the computed transmission eigenvalues are $k_1=1.2387, \, k_2=1.4771, \,k_3=2.0513, \,k_4=3.0013$. The reconstructions by our proposed method are given by (a)--(c), meanwhile the reconstructions by the sampling-type method are given by (d)--(f). The sub-figures (g)--(i) present the combined results of the above two reconstructions. Clearly, Fig.~\ref{fig:super-kite}, (i) yields a very nice reconstruction of the kite-object, especially the concave part.
\begin{figure}[h]
\centering
\subfigure[$L=1$]{\includegraphics[width=0.33\textwidth]{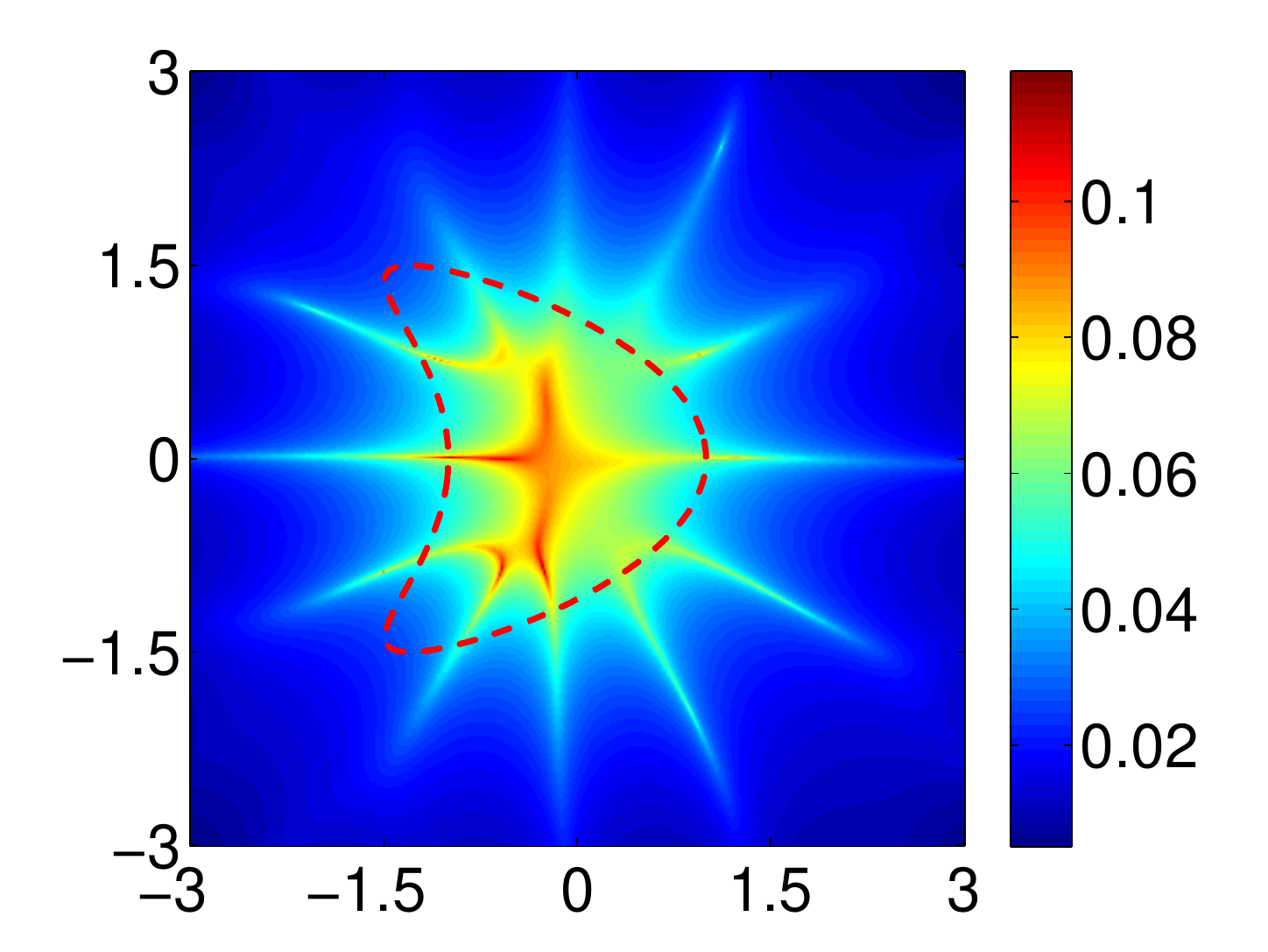}}% lleHere is how to import EPS art
\subfigure[$L=2$]{\includegraphics[width=0.33\textwidth]{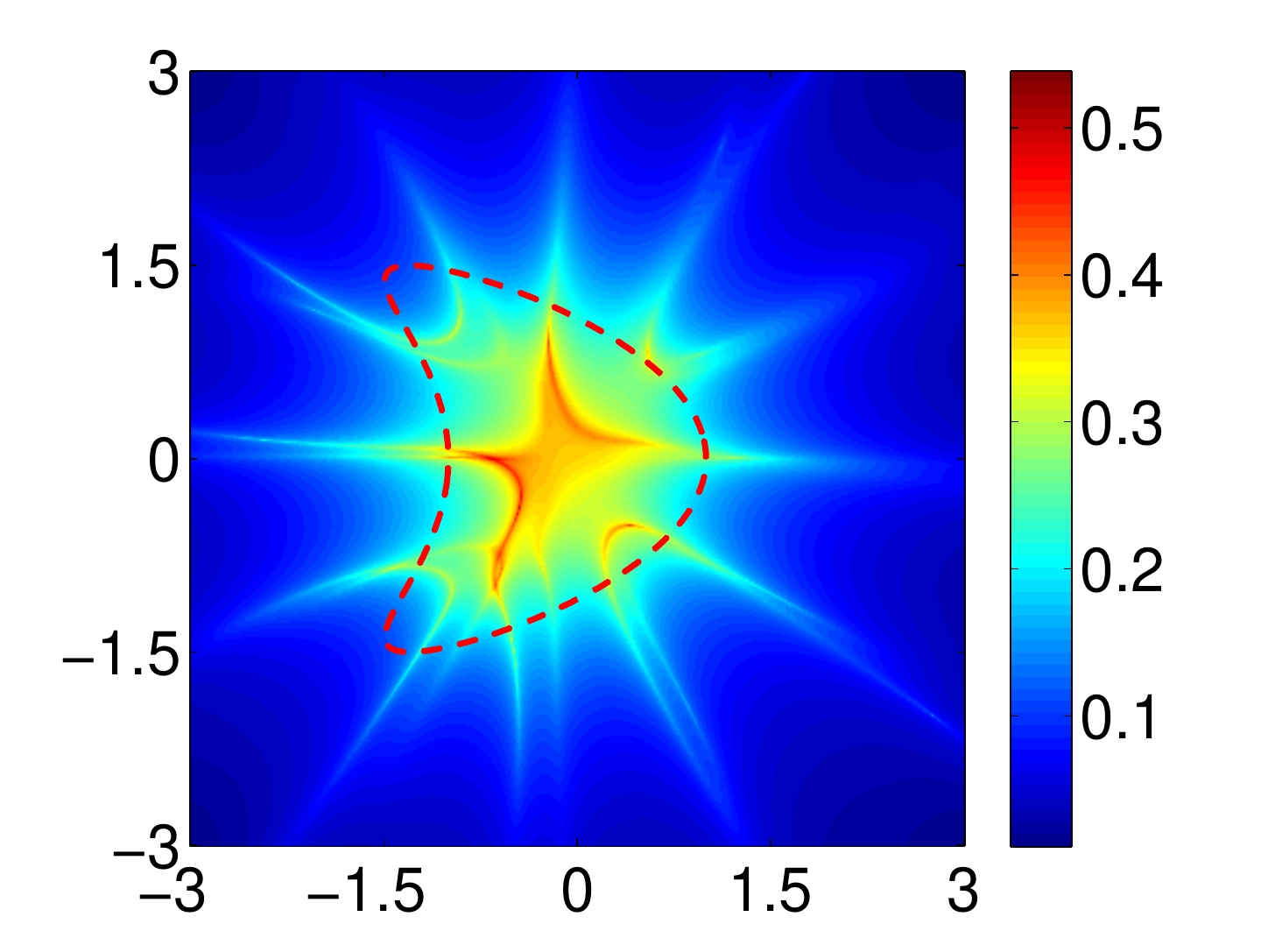}} % lleHere is how to import EPS art
%\subfigure[$N=3$]{\includegraphics[width=0.4\textwidth]{kite_K3}}% lleHere is how to import EPS art
\subfigure[$L=4$]{\includegraphics[width=0.33\textwidth]{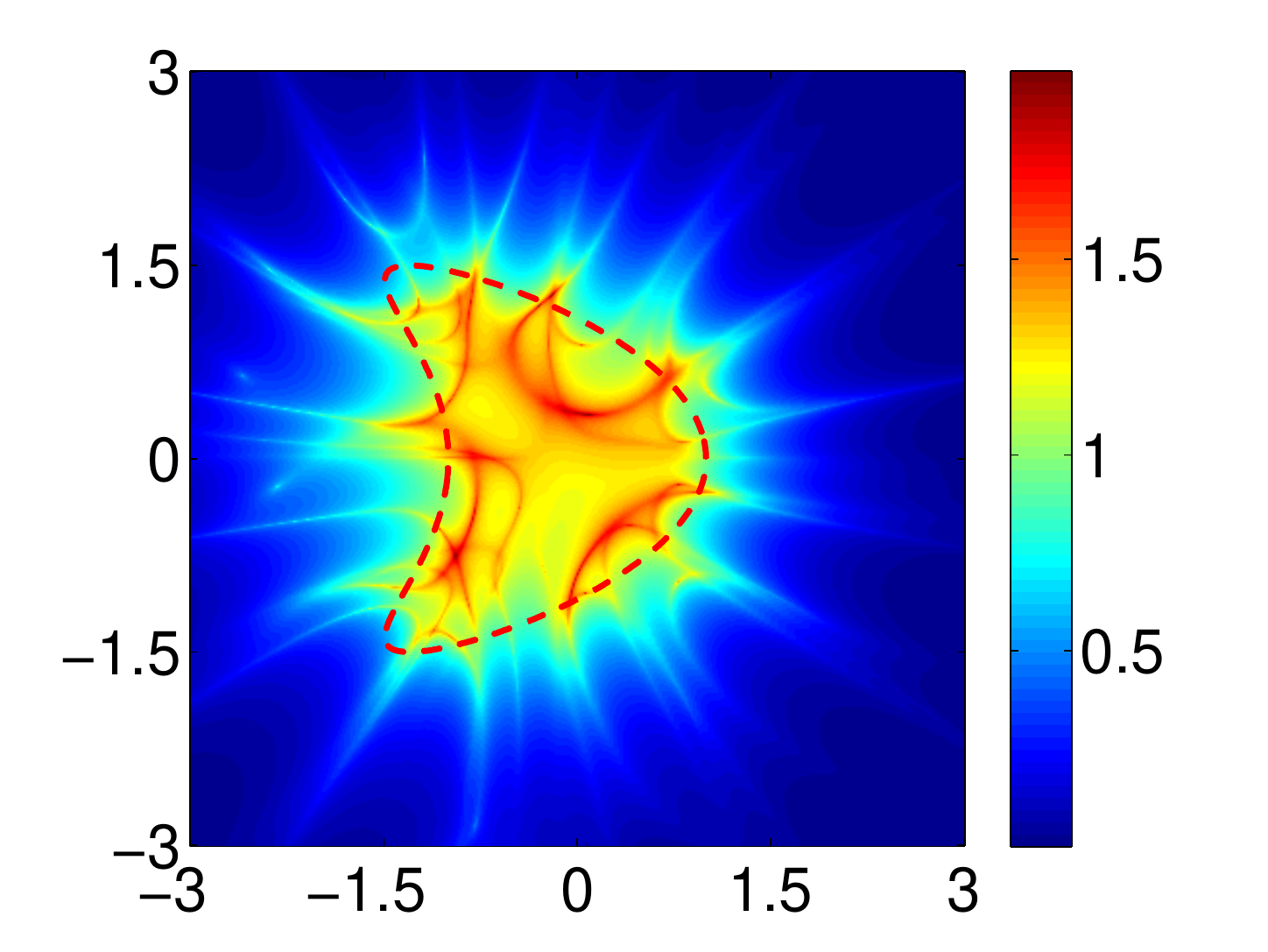}}\\ % lleHere is how to import EPS art
\subfigure[$L=1$]{\includegraphics[width=0.33\textwidth]{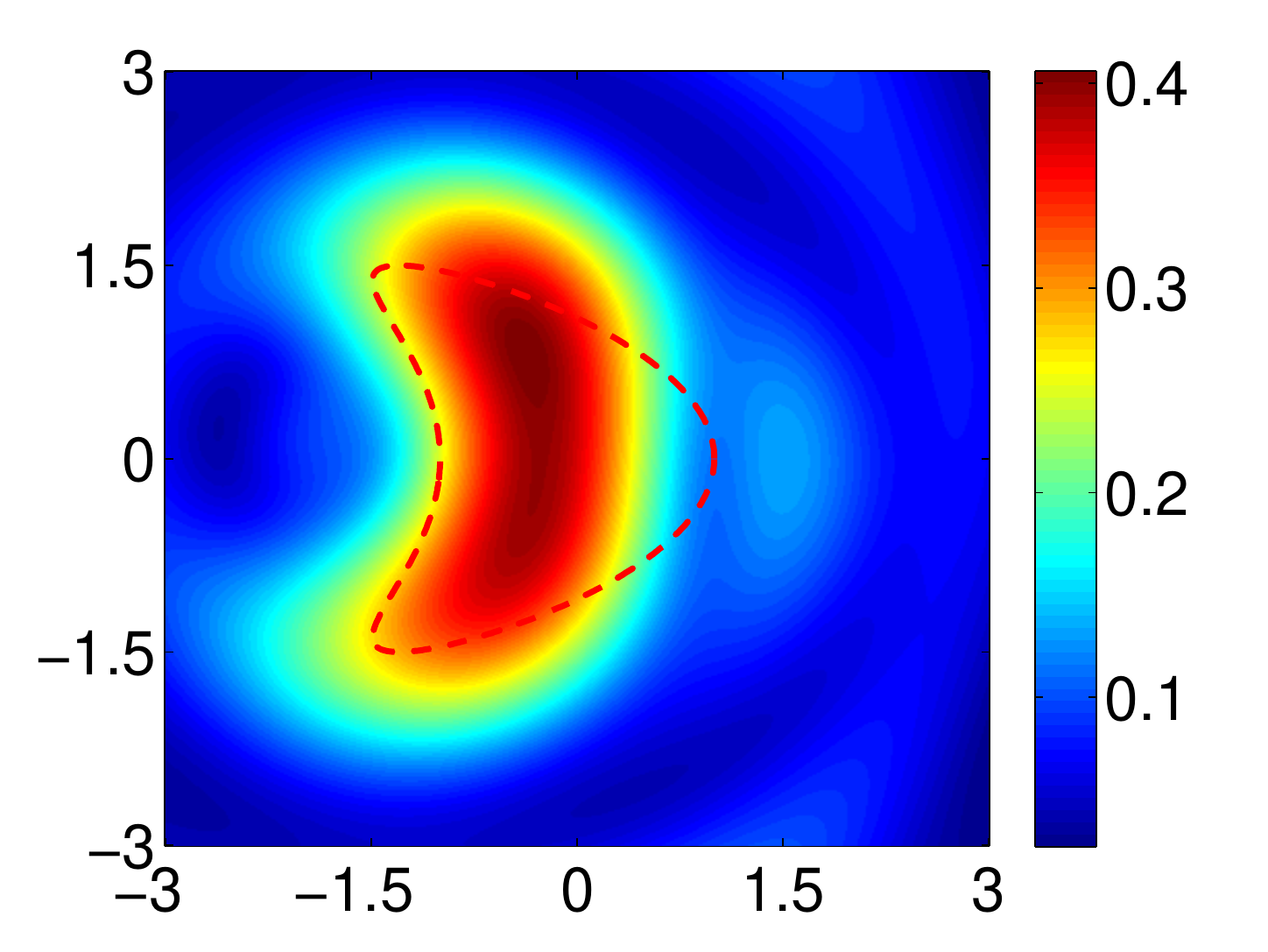}}% lleHere is how to import EPS art
\subfigure[$L=2$]{\includegraphics[width=0.33\textwidth]{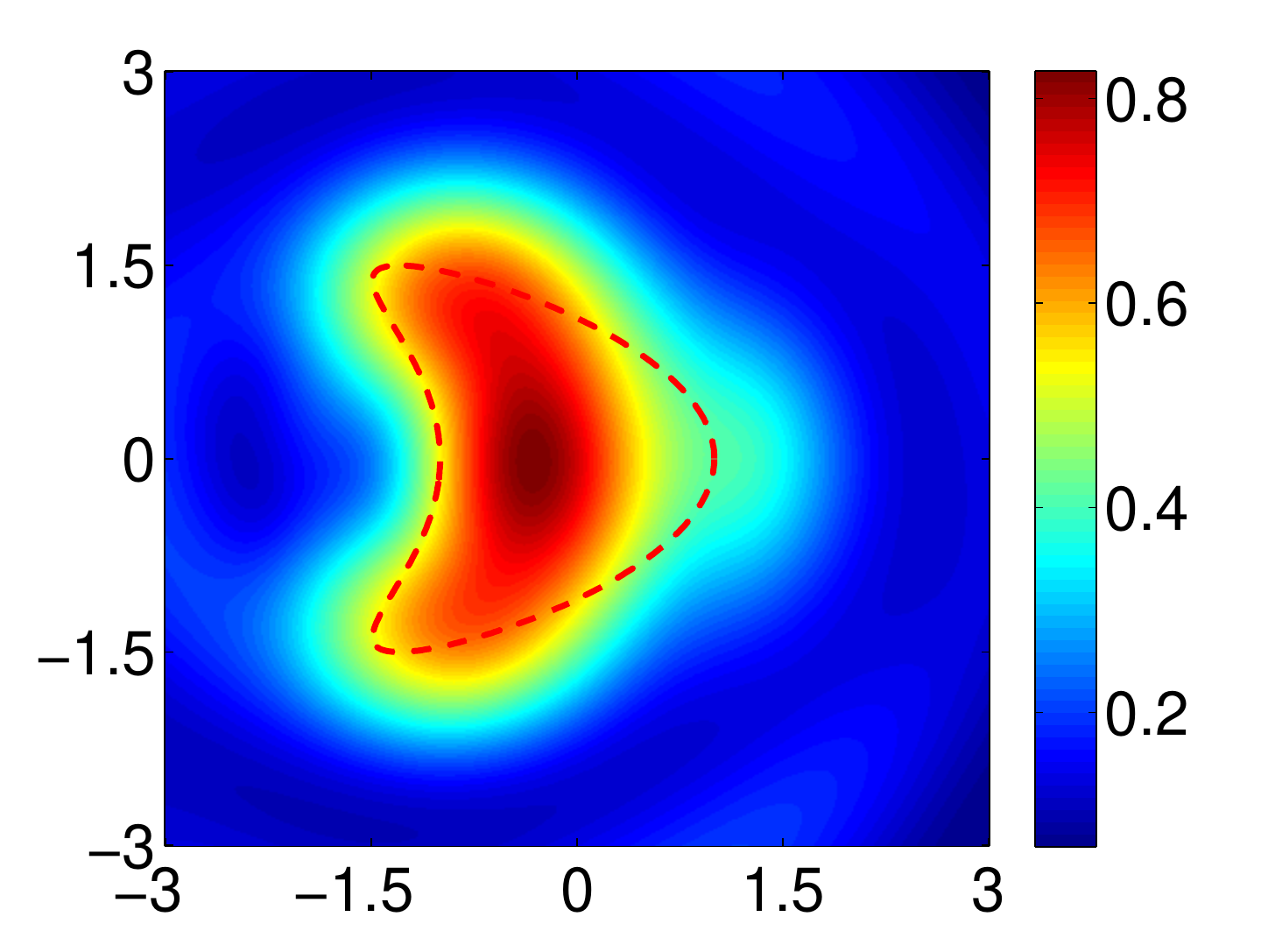}} % lleHere is how to import EPS art
%\subfigure[$N=3$]{\includegraphics[width=0.4\textwidth]{kite_K3}}% lleHere is how to import EPS art
\subfigure[$L=4$]{\includegraphics[width=0.33\textwidth]{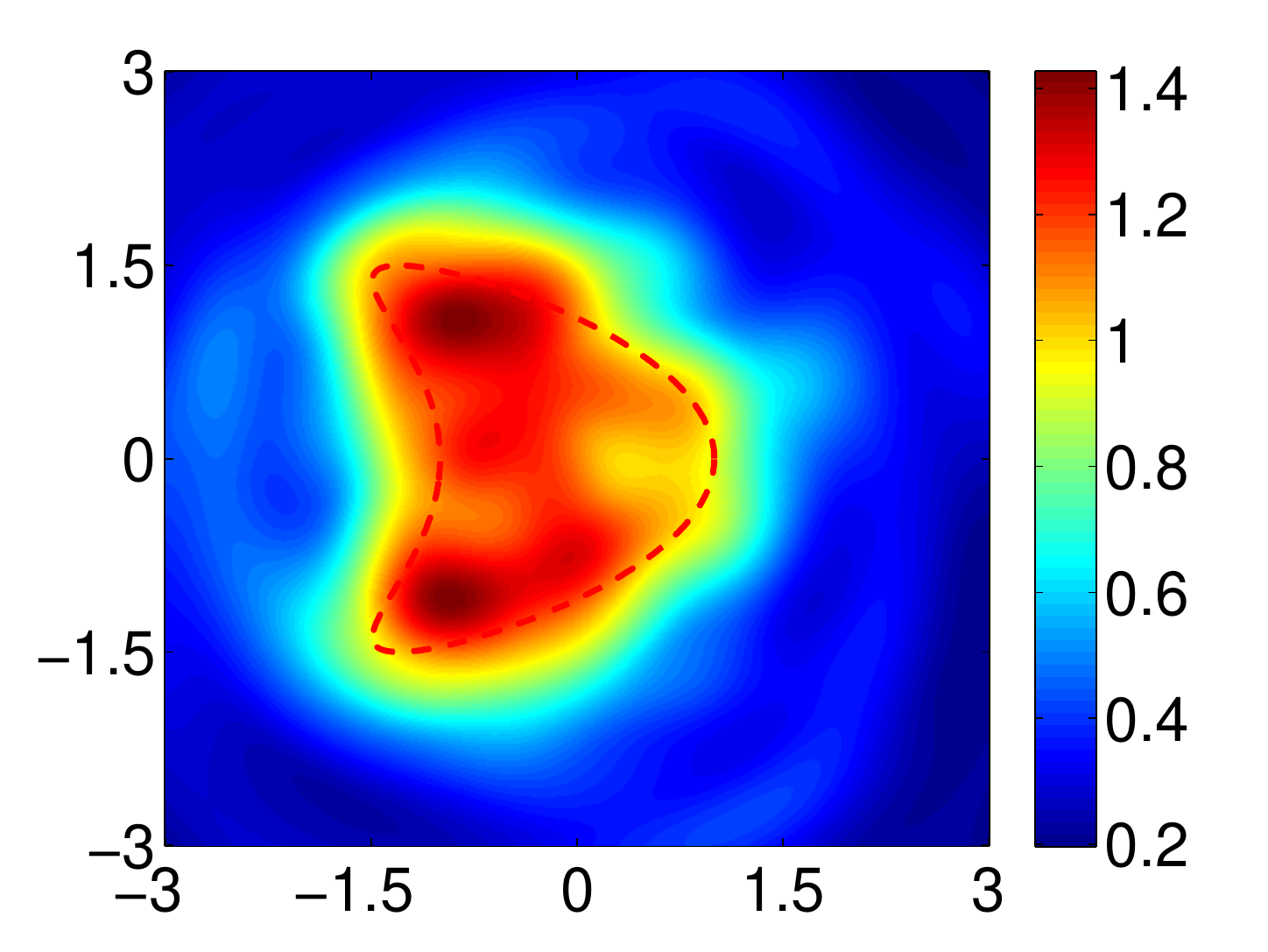}}\\ % lleHere is how to import EPS art
\subfigure[$L=1$]{\includegraphics[width=0.33\textwidth]{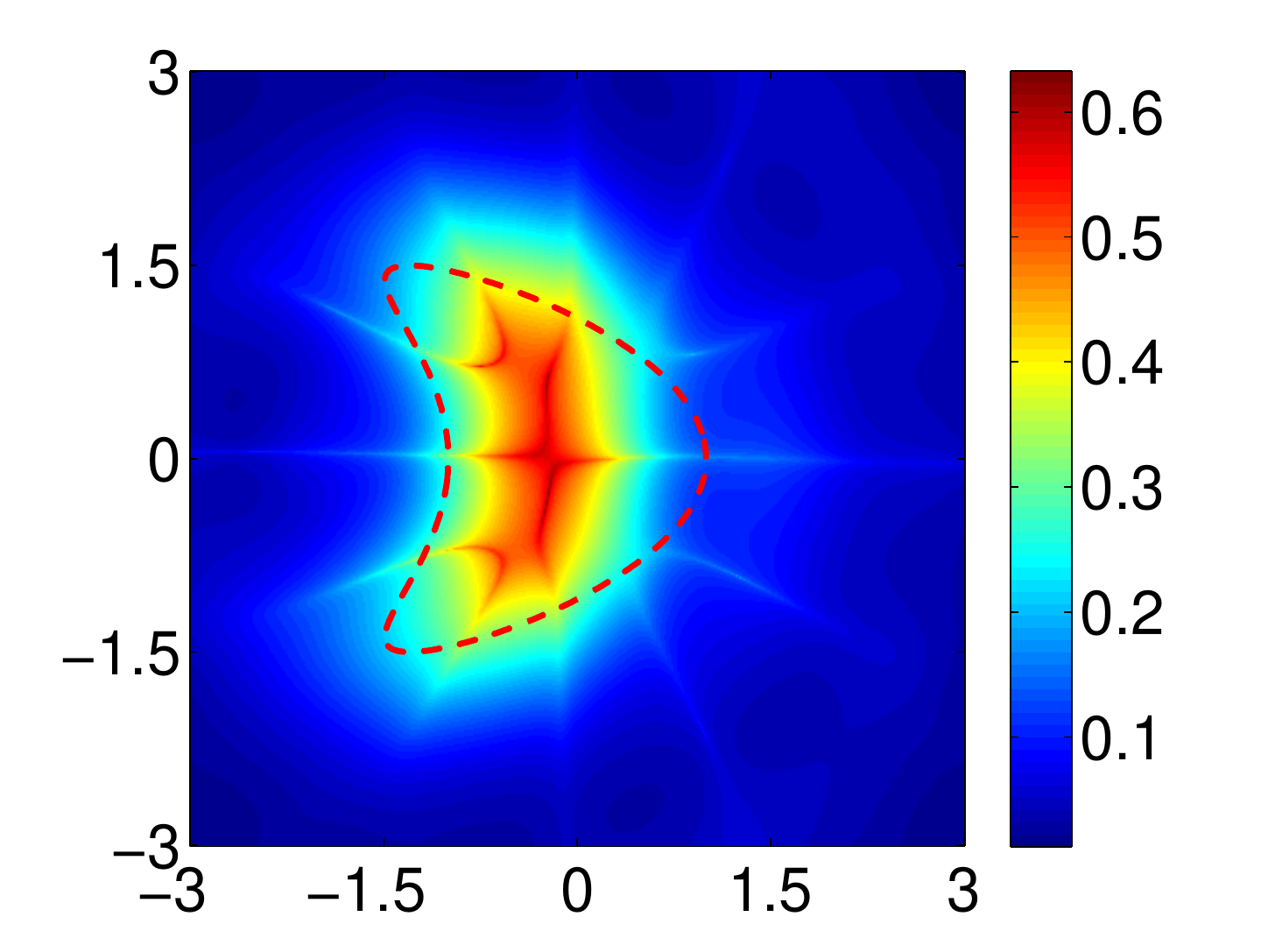}}% lleHere is how to import EPS art
\subfigure[$L=2$]{\includegraphics[width=0.33\textwidth]{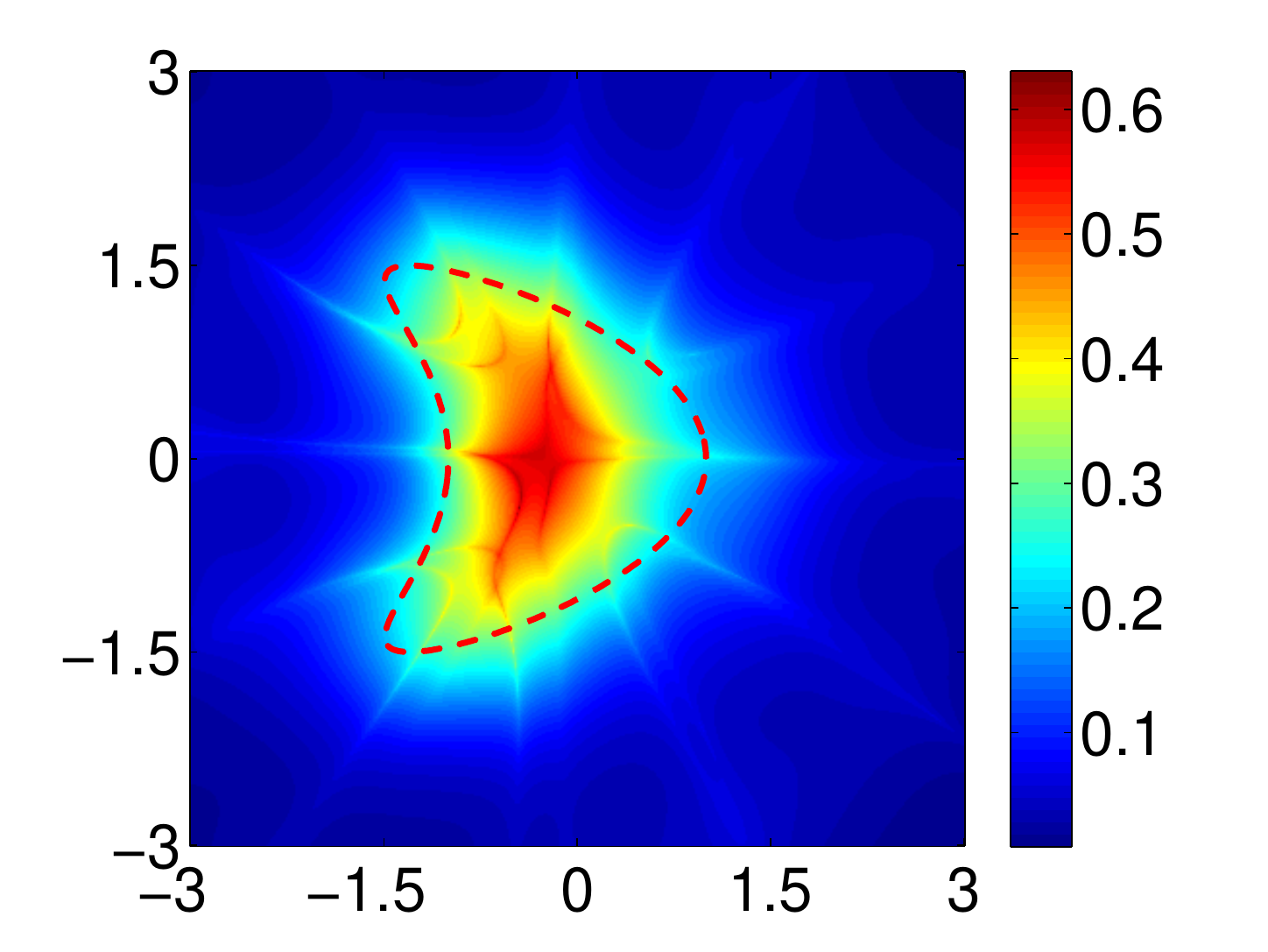}} % lleHere is how to import EPS art
%\subfigure[$N=3$]{\includegraphics[width=0.4\textwidth]{kite_K3}}% lleHere is how to import EPS art
\subfigure[$L=4$]{\includegraphics[width=0.33\textwidth]{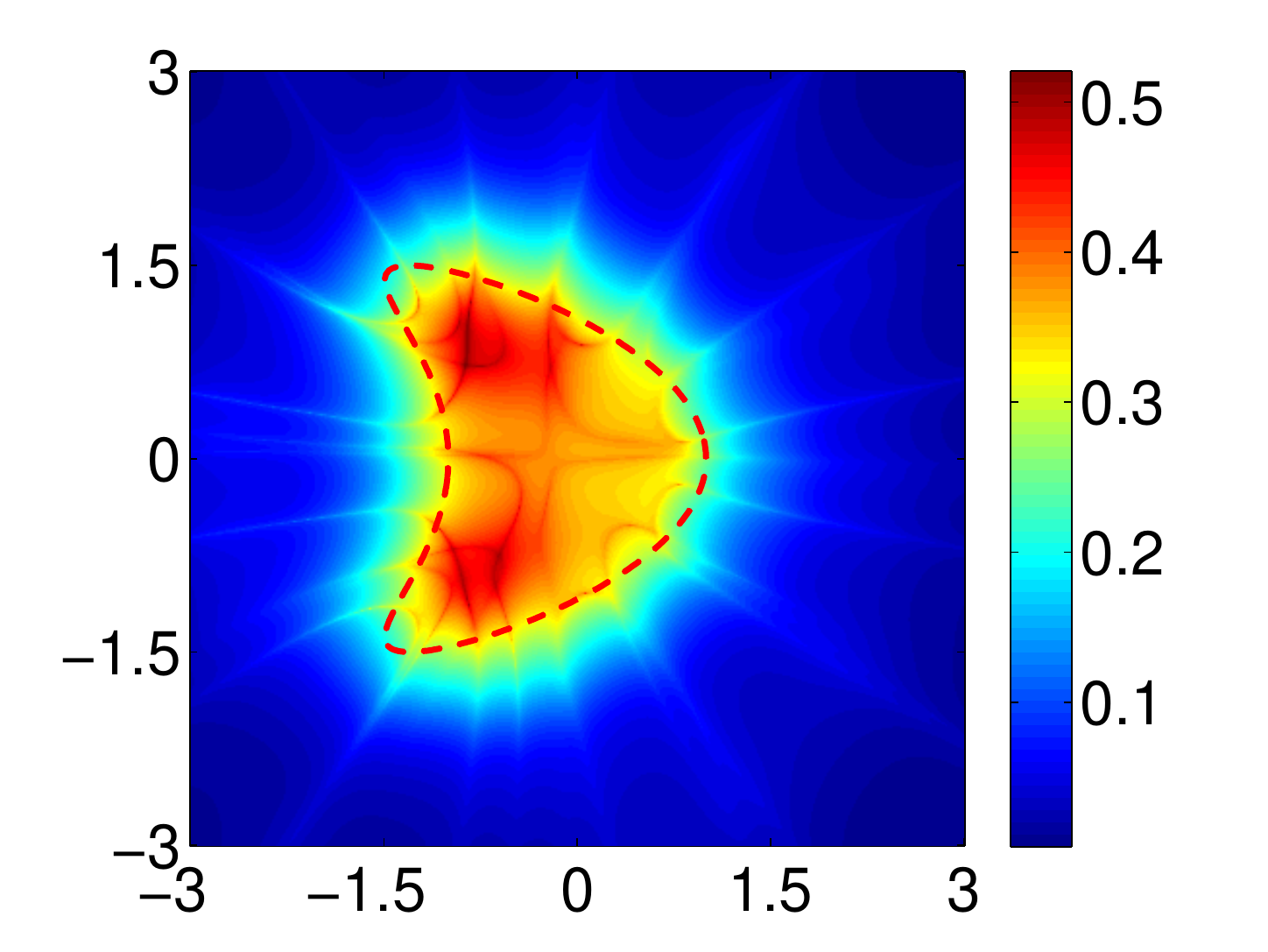}}\\ % lleHere is how to import EPS art
\caption{\label{fig:super-kite} Reconstructions of a kite-object by multiple SLEs (top row) and multiple-frequency direct sampling method (middle row), respectively, where $\mathbf{n}=10$ for all cases. Bottom row gives the corresponding reconstructions of combining the above two reconstruction means.}
\end{figure}

Three remarks are in order. First, it can seen from the reconstructions in Fig.~\ref{fig:super-kite} that the sampling-type method tends to reconstruct a ``larger" object while our proposed method tends to reconstruct a ``smaller" object. This is physically reasonable since the sampling-type method as well as the other traditional inverse scattering schemes make use the measurement data away from the scatterer for the reconstruction, which amounts to ``seeing" the scattering object from its outside, whereas our method makes use of the interior resonant modes, which amounts to ``seeing" the scattering object its inside. Hence, hybridizing the two types of methods can yield a better reconstruction. Second, it is arguable that the super-resolution effect comes from the high-contrast medium parameter $\mathbf{n}$ in this specific example (cf. \cite{Amm0}). Indeed, as discussed earlier, a high-contrast $\mathbf{n}$ leads to a relatively small $k$ that can induce the desired SLE for the reconstruction, which is a matter of fact. However, in practice, for a regular refractive inhomogeneity, one may first coat the object via indirect means with a thin layer of high-contrast medium (cf. Fig.~\ref{fig:medium_layer}), then apply the same reconstruction procedure as above. According to the results in Fig.~\ref{fig:medium_layer}, one would have the same super-resolution reconstructions as in Fig.~\ref{fig:super-square}, (a)--(c). Third, the super-resolution is achieved at the cost of a large amount of computations and a restrictive requirement on the high-precision of the measurement data. This is unobjectionable due to the increasing capabilities of physical apparatus nowadays.

\section{Pseudo surface plasmon resonances and potential applications}

Surface plasmon resonance (SPR) is the resonant oscillation of conducting electrons at the interface between negative and positive permittivity materials stimulated by incident light. It is a non-radiative electromagnetic surface wave that propagates in a direction parallel to the negative permittivity/dielectric material interface \cite{BS, FM, Kli, OI, Z}. Clearly, the SPR wave is a surface-localized mode. It is in this sense that the SLE can be viewed as a certain SPR. Indeed, viewed from the inside of $\Omega$ (this is unobjectionable since $v$ is only supported in $\Omega$), the behaviour of a SLE is very much like a SPR. However, SPR usually occurs in the quasi-static regime (subwavelength scale), whereas SLE can occur in both the quasi-static regime and the high-frequency regime. Moreover, the SPR is usually generated from direct light incidence, whereas the generation of SLEs is rather indirect according to our earlier study. As is known that the SPR can have many industrial and engineering applications including color-based biosensors, different lab-on-a-chip sensors and diatom photosynthesis \cite{Kli}. In what follows, we show that the SLEs can also be generated through direct wave incidences. This will pave the way for the proposal of an interesting SLE sensing that is similar to the SPR sensing.

First, we recall that assuming $\mathbb{R}^d\backslash\overline{\Omega}$ connected, the Herglotz waves of the form \eqref{eq:Herglotz} are dense in the space $\{v\in H^1(\Omega); (\Delta+k^2) v=0\ \mbox{in}\ \Omega\}$. Hence, for any transmission eigenfunction $v$ to \eqref{eq:trans1}, there exists $g\in L^2(\mathbb{S}^{d-1})$ such that $v_g^k\approx v$ in $H^1(\Omega)$. Next, for a refractive inhomogeneity $\mathbf{n}^2$, $0<\mathbf{n}<1$, supported in $\Omega$ with $\mathbb{R}^d\backslash\overline{\Omega}$ connected, we let $k_0$ be an eigenvalue to \eqref{eq:trans1} with the eigenfunctions denoted as $(w_\Omega, v_\Omega)$ such that $w_\Omega$ is an SLE. Let $v_g^{k_{0}}$ be a Herglotz wave function of the form \eqref{eq:Herglotz} such that ${v}_g^{k_{0}}\approx v_\Omega$ in $H^1(\Omega)$. Now, we consider the scattering problem \eqref{eq:forwardpro} with the incident field $u^i={v}_g^{k_{0}}$. It is straightforward to show that if $u_\infty(\hat x, v_{g}^{k_{0}})\equiv 0$ (equivalent to $u^s(x, v_{g}^{k_{0}})=0$ in $\mathbb{R}^d\backslash\overline{\Omega}$ by Rellich's Theorem \cite{CK}), one then has the transmission eigenvalue problem \eqref{eq:trans1} with $k=k_0$, $w_\Omega=u|_{\Omega}$ and $v=u^i|_{\Omega}$, where $u$ is the total field to \eqref{eq:forwardpro}. Conversely, noting that $u^i\approx v_\Omega$ from our earlier construction, one can show (cf. \cite{BL1}) that $u^\infty\approx 0$, and more importantly $u|_{\Omega}\approx w_\Omega$. Since $w_\Omega$ is an SLE, we see that $u|_{\Omega}$ is also an SLE (at least approximately). Set
\begin{equation}\label{eq:sle2}
\widehat{w}=\begin{cases} u-u^i\ \ &\mbox{in}\ \mathbb{R}^d\backslash\overline{\Omega},\\ u\ \ &\mbox{in}\ \Omega.  \end{cases}
\end{equation}
Clearly, $\widehat{w}$ is generated from a direct incidence on the inhomogeneity $\mathbf{n}^2$. $\widehat{w}\approx 0$ in $\mathbb{R}^d\backslash\overline{\Omega}$ and $\widehat{w}\approx w_\Omega$. That is, $\widehat{w}$ is localized around $\partial\Omega$, which exhibits a similar behaviour to the SPR oscillation. In what follows, we refer to $\widehat{w}$ as a pseudo plasmon resonant (PSPR) mode. In Fig. \ref{fig:Plasmon}~(a)--(c), we present a numerical illustration of the generation of a PSPR mode.

We next propose a potential sensing application of the PSPR mode. Let $(\Omega, \mathbf{n}^2)$ be an a-priori known inhomogeneity. Due to a certain reason, it is supposed that $\partial\Omega$ has some fine defects, namely, the support of the inhomogeneity actually becomes $\partial\widetilde{\Omega}$. Following the spirit of SPR sensing, one can detect the boundary defects as follows. Let $u^i$ be an incident field that can generate a PSPR $\widehat{w}$ associated with $(\Omega, \mathbf{n}^2)$ as above. The field impinges on $(\widetilde{\Omega}, \mathbf{n}^2)$, and we let $\widetilde{w}$ be the associated field according to \eqref{eq:sle2}. In Fig. \ref{fig:Plasmon}(e) and \ref{fig:Plasmon}(f), we present the corresponding numerical results. It can be seen that the difference $\widetilde w-\widehat{w}$ is highly sensitive with respect to the boundary defects $\partial\widetilde\Omega-\partial\Omega$. Hence, by the SPRS sensing, one can easily identify the existence of the fine boundary defects. It would be interesting to proceed further to recover such fine boundary defects by using the sensing data $\widetilde w-\widehat{w}$, which we choose to present in a forthcoming paper.

\begin{figure}[h]
\subfigure[incident field $u^i$]
{\includegraphics[width=0.33\textwidth]{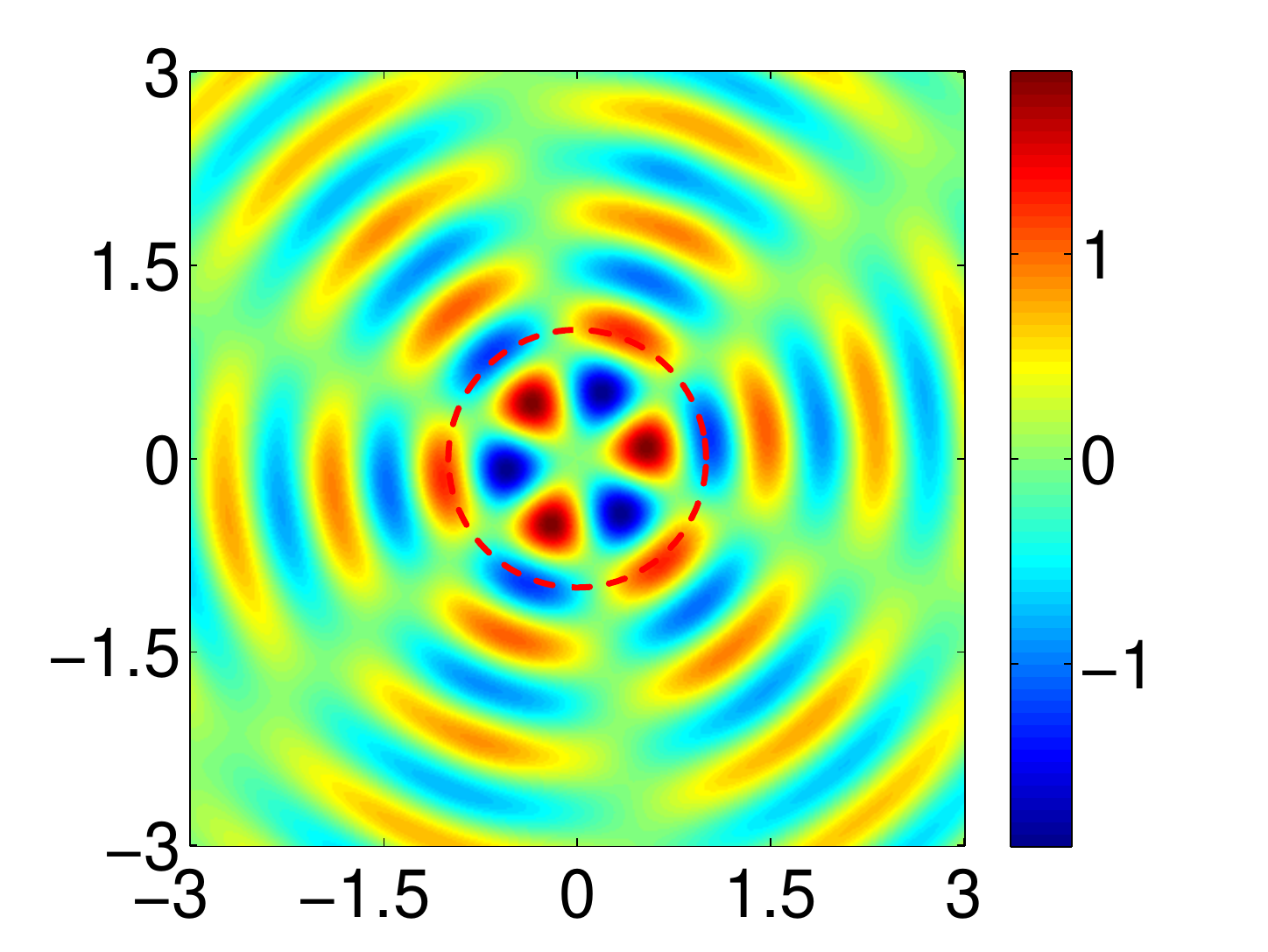}}% lleHere is how to import EPS art
\subfigure[total field $u$]
{\includegraphics[width=0.33\textwidth]{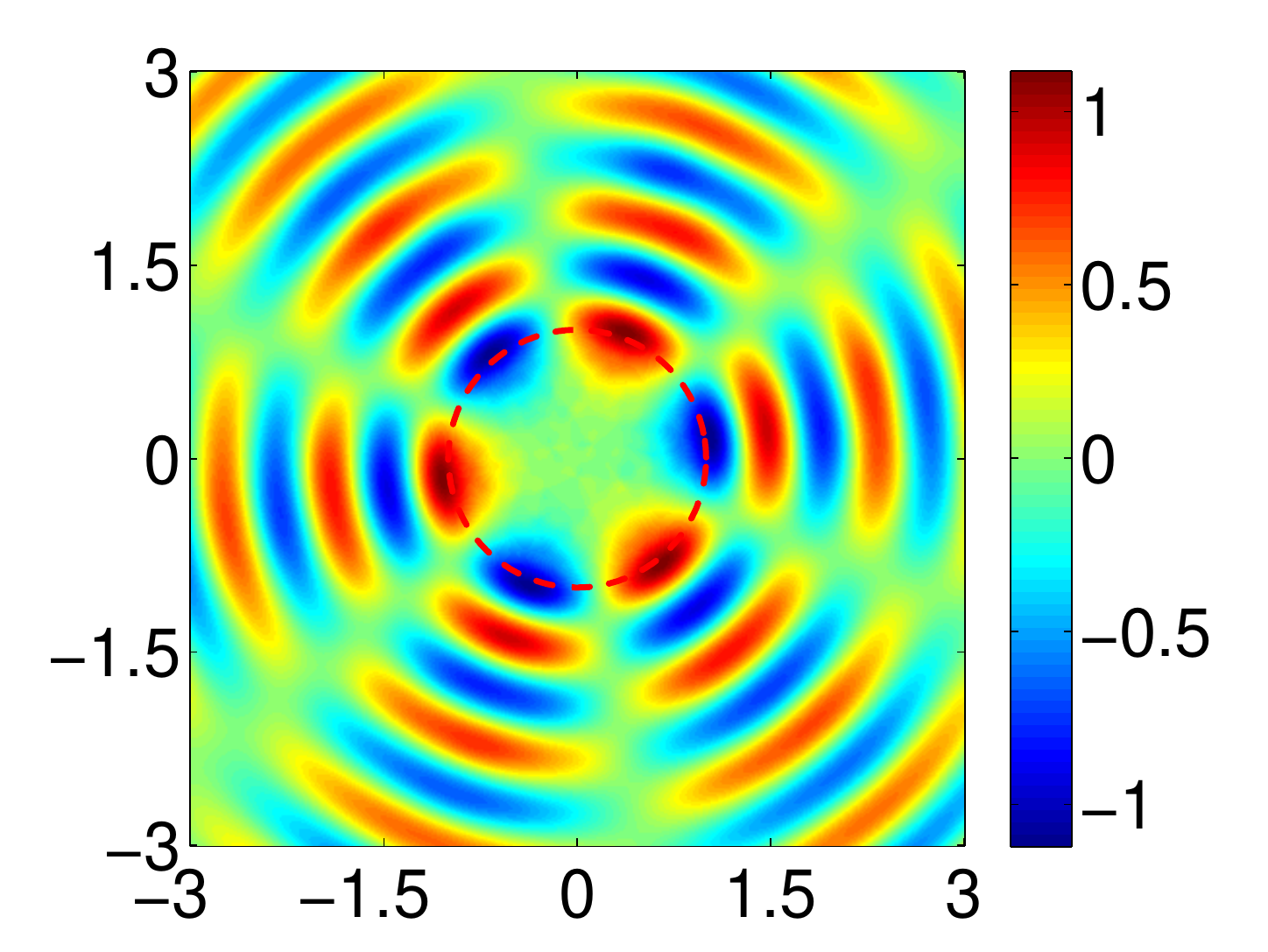}} % lleHere is how to import EPS art
\subfigure[generated PSPR mode $\widehat{w}$]
{\includegraphics[width=0.33 \textwidth]{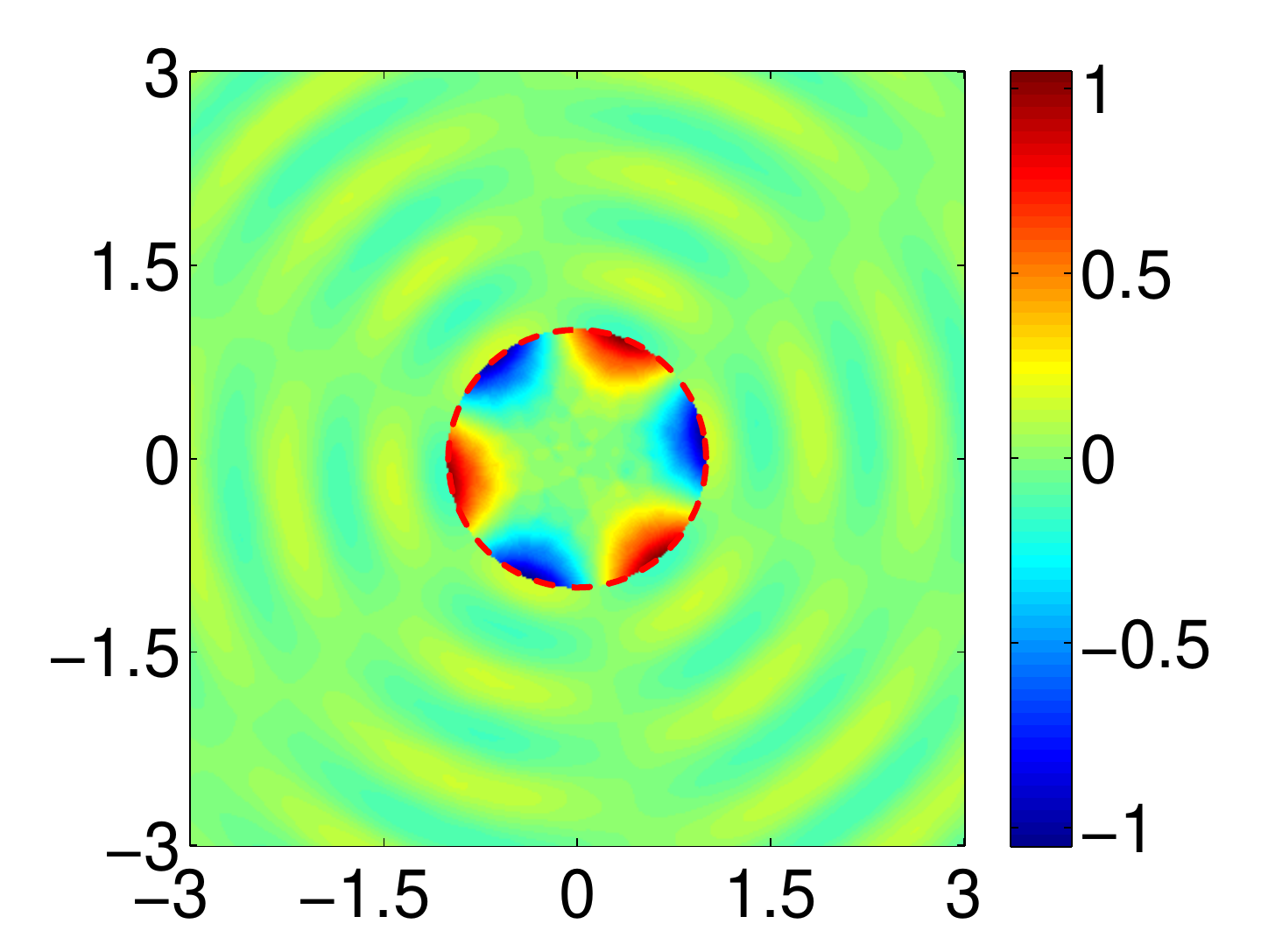}}\\% lleHere is how to import EPS art
\subfigure[incident field $u^i$]
{\includegraphics[width=0.33 \textwidth]{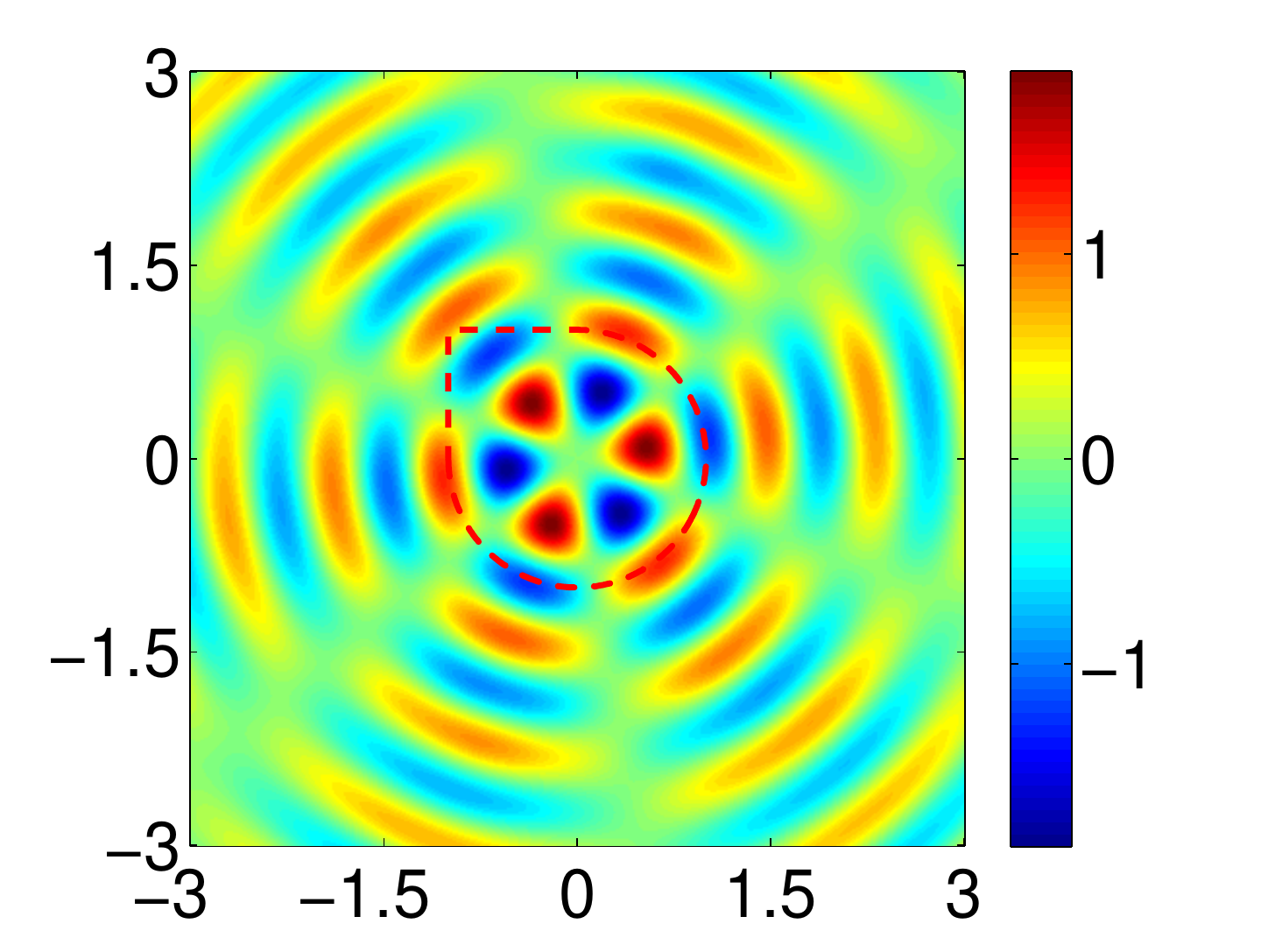}}% lleHere is how to import EPS art
\subfigure[total field $\widetilde{u}$]
{\includegraphics[width=0.33\textwidth]{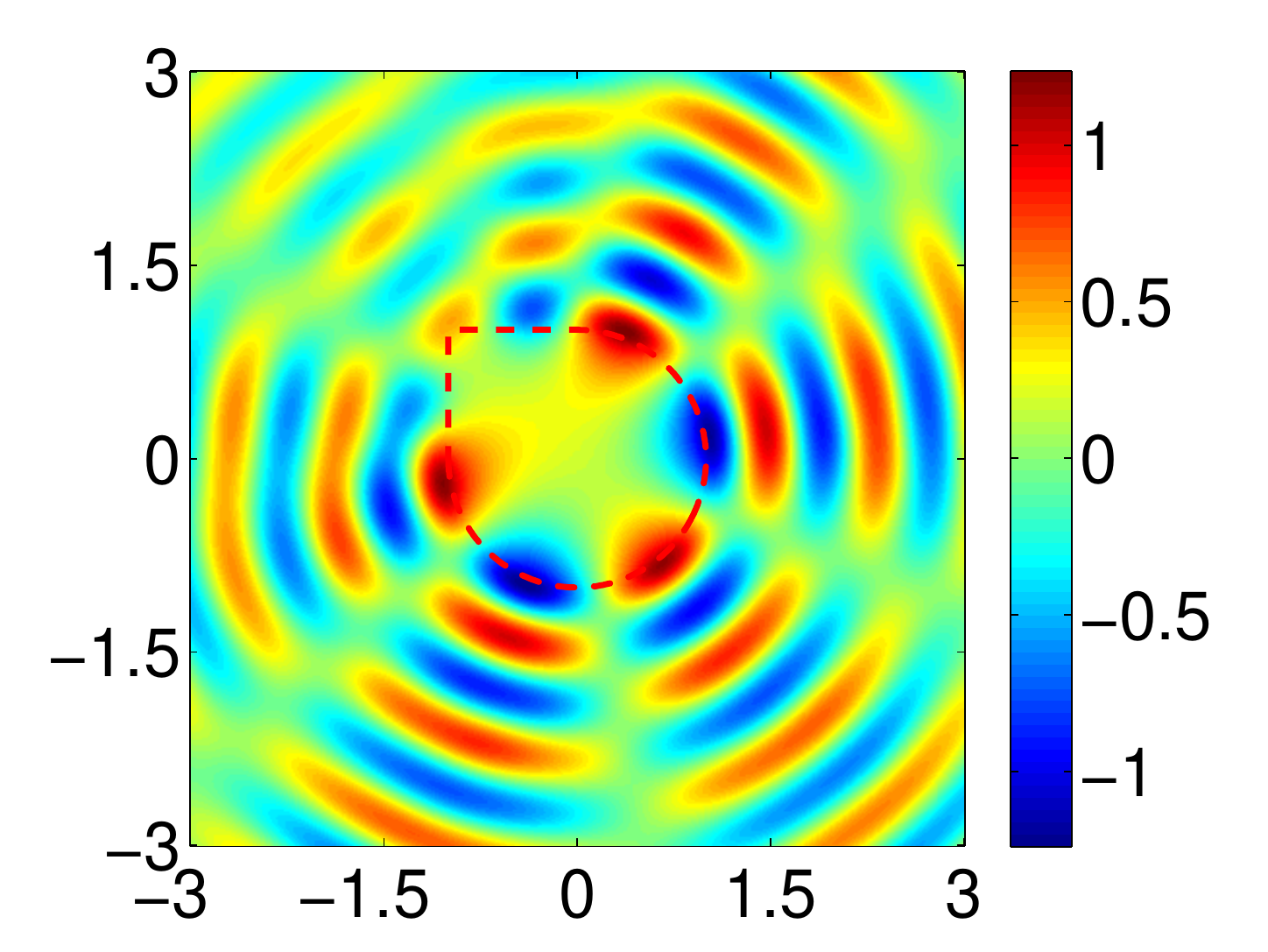}} % lleHere is how to import EPS art
\subfigure[generated PSPR mode $\widetilde{w}$]
{\includegraphics[width=0.33\textwidth]{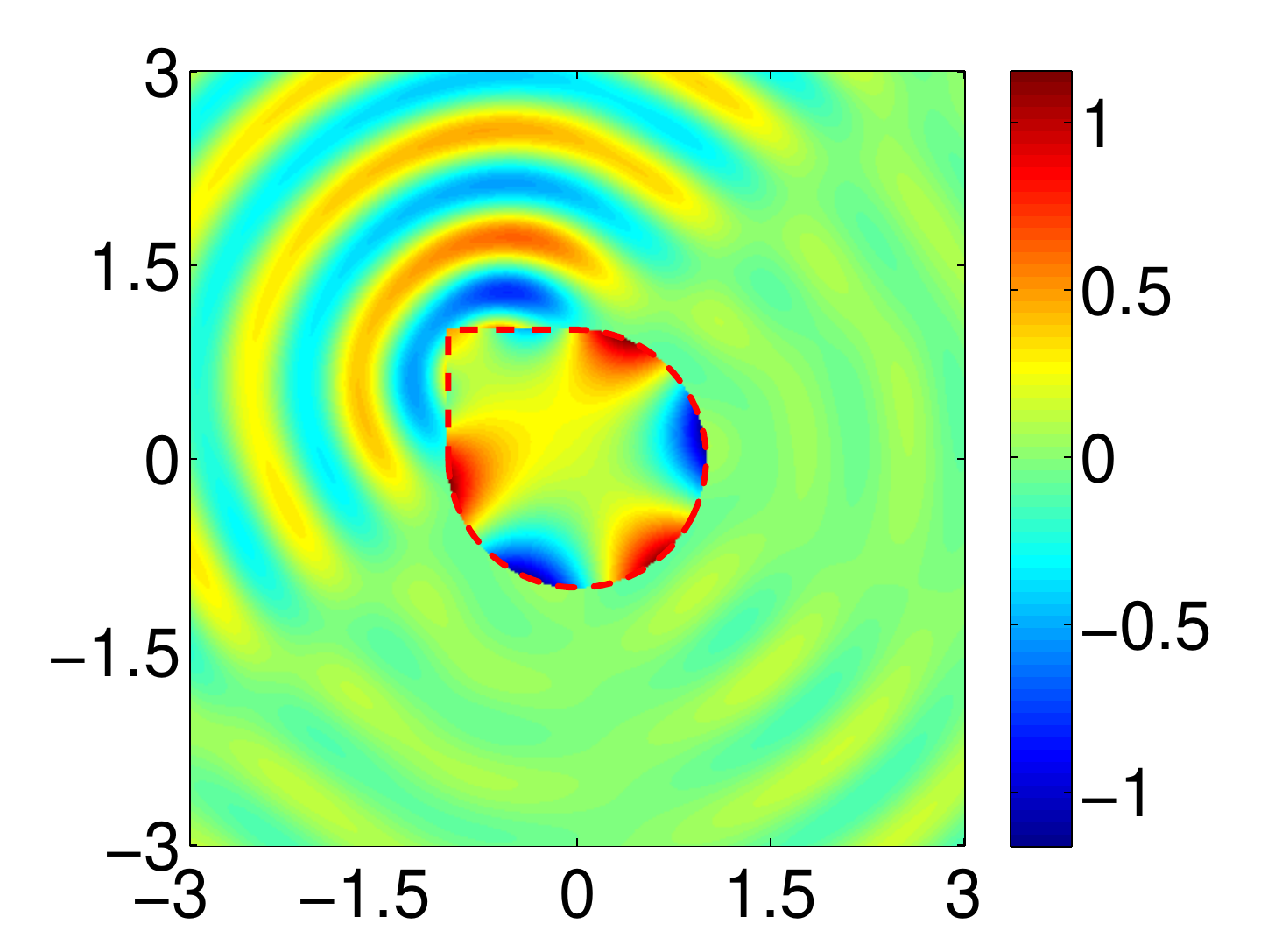}}\\% lleHere is how to import EPS art
\caption{\label{fig:Plasmon} Generation of the PSPR mode, where $k=7.6548$ and $\mathbf{n}=1/4$. $(u, \widehat{w})$ and $(\widetilde u, \widetilde w)$, respectively, denote the corresponding fields associated with $(\Omega, \mathbf{n}^2)$ and $(\widetilde\Omega, \mathbf{n}^2)$. }
\end{figure}

\section{Concluding remarks}

In this paper, we present the discovery of a certain intriguing global geometric structure of the transmission eigenfunctions. It is shown that there exist the so-called SLEs. We rigorously and comprehensively justify this spectral property in the radial geometry case. For the general case, we conducted extensive numerical experiments, which not only verified such a spectral property but also revealed many delicate and subtle quantitative behaviours of the SLEs. The results derived in this paper not only unveil an important spectral phenomenon that was unknown before, but also generat some applications of practical values.  We apply the spectral results to develop a super-resolution wave imaging scheme and also propose a procedure of generating the so-called PSPR mode, which has the potential to be used in sensing technology.

The focus of this paper is to present the discovery of the global geometric structure of the transmission eigenfunctions as well as its implication to the wave localization with potential applications of practical importance. There are many subtle issues for the study in this work that would need to be fully developed: the theoretical justification of the SLEs in the general scenario \cite{CHLW}; more numerical experiments should be conducted for the proposed super-resolution imaging scheme including the 3D case and the variable refractive inhomogeneities, which require a huge amount of computations; and the further recovery of the fine boundary defect by using the PSPR modes. We shall explore those issues in our forthcoming works. Moreover, our study opens up a new field of research on the global properties of transmission eigenfunctions with many potential developments.

\section*{Acknowledgment}
The work of Y. Deng was supported by NSF grant of China No. 11971487 and NSF grant of Hunan No. 2020JJ2038. The work of H. Liu was supported by a startup grant from City University of Hong Kong and Hong Kong RGC General Research Funds (projects 12301218, 12302919 and 12301420).  The work of X. Wang was supported by  the Hong Kong Scholars Program grant under No. XJ2019005.

\end{document}